\documentclass[12pt]{amsart}
\usepackage{amsaddr}
\usepackage{amssymb}
\usepackage{mathtools}
\usepackage{txfonts}
\usepackage{eucal}
\usepackage{extarrows}

\usepackage{color}

\usepackage[all]{xy}

\usepackage{tikz}

\usepackage{graphicx}
\usepackage{float}

\usepackage{xspace}

\usepackage[a4paper,body={16.3cm,22.8cm},centering]{geometry}
\usepackage[colorlinks,final,backref=page,hyperindex,pdftex]{hyperref}

\newcommand{\nc}{\newcommand}
\newcommand{\delete}[1]{}

\nc{\mlabel}[1]{\label{#1}}  
\nc{\mcite}[1]{\cite{#1}}  
\nc{\mref}[1]{\ref{#1}}  
\nc{\mbibitem}[1]{\bibitem{#1}} 

\delete{
\nc{\mlabel}[1]{\label{#1}  
{\hfill \hspace{1cm}{\small\tt{{\ }\hfill(#1)}}}}
\nc{\mcite}[1]{\cite{#1}{\small{\tt{{\ }(#1)}}}}  
\nc{\mref}[1]{\ref{#1}{{\tt{{\ }(#1)}}}}  
\nc{\mbibitem}[1]{\bibitem[\bf #1]{#1}} 
}

\newtheorem{theorem}{Theorem}[section]
\newtheorem{prop}[theorem]{Proposition}
\newtheorem{lemma}[theorem]{Lemma}
\newtheorem{coro}[theorem]{Corollary}

\theoremstyle{definition}
\newtheorem{defn}[theorem]{Definition}
\newtheorem{remark}[theorem]{Remark}
\newtheorem{exam}[theorem]{Example}
\newtheorem{prop-def}{Proposition-Definition}[section]

\newcommand\cal[1]{\mathcal{#1}}

\newcommand\alphlist{a,b,c,d,e,f,g,h,i,j,k,l,m,n,o,p,q,r,s,t,u,v,w,x,y,z}
\newcommand\Alphlist{A,B,C,D,E,F,G,H,I,J,K,L,M,N,O,P,Q,R,S,T,U,V,W,X,Y,Z}
\newcommand\getcmds[3]{\expandafter\newcommand\csname #2#1\endcsname{#3{#1}}}
\makeatletter
\@for\x:=\alphlist\do{\expandafter\getcmds\expandafter{\x}{frak}{\mathfrak}}
\@for\x:=\Alphlist\do{\expandafter\getcmds\expandafter{\x}{frak}{\mathfrak}}
\makeatother

\nc{\bfk}{{\bf k}}
\font\cyr=wncyr10

\newfont{\scyr}{wncyr10 scaled 550}
\nc{\sha}{\mbox{\cyr X}}
\nc{\ssha}{\mbox{\bf \scyr X}}

\nc{\id}{\mathrm{id}}
\nc{\Id}{\mathrm{Id}}
\nc{\lbar}[1]{\overline{#1}}
\nc{\ot}{\otimes}
\nc{\dep}{\mathrm{dep}}
\nc{\leaf}{\mathrm{leaf}}

\usetikzlibrary{calc}
\newlength\xch
\newsavebox\dbox
\sbox\dbox{\tikz{\fill (0,0) circle (0.05cm);}}
\newif\ifqdd
\newif\ifzdd
\newcommand\cddf[3]{%
\coordinate (#2) at ($(#1)+(#3)$);
\draw (#1)--(#2);
\ifqdd\node at (#1) {\usebox\dbox};\fi
\ifzdd\node at (#2) {\usebox\dbox};\fi}
\newcommand\cdx[4][1]{\cddf{#2}{#3}{#4:#1*\xch}}
\newcommand\cdl[2][1]{\cdx[#1]{#2}{#2l}{135}}
\newcommand\cdr[2][1]{\cdx[#1]{#2}{#2r}{45}}
\newcommand\cdlr[2][1]{%
\foreach \i in {#2} {\cdl[#1]{\i}\cdr[#1]{\i}}}
\newcommand\cda[2][1]{\cdx[#1]{#2}{#2a}{90}}
\newcommand\cdb[2][1]{\cdx[#1]{#2}{#2b}{-90}}
\let\treeoo\treeo%
\newcommand\zhongdian[2]{\node at ($(#1)!0.5!(#2)$) {\usebox\dbox};}
\newcommand\zhd[1]{\foreach \i/\j in {#1} {\zhongdian{\i}{\i\j}}}

\newcommand\ocdx[6][1]{%
\node[draw,circle,minimum size=2pt,label={#6:$#5$}]
(#3) at ($(#2)+(#4:#1*\xch)$) {};
\draw (#2)--(#3);}

\newcommand\scopeclip[1]{\begin{scope}
\clip(-1.1,-0.5)rectangle(1.1,1);#1\end{scope}}
\newcommand\XX[2][]{%
\tikz[line width=0.15ex,x=0.5cm,y=0.5cm,baseline,inner sep=1.5pt,
every node/.style={font=\scriptsize},#1]{
\scopeclip{\draw (135:1.5)--(0,0)--(45:1.5) (0,-0.5)--(0,0);}#2}}
\newcommand\xx[3]{%
\scopeclip{\draw(#1/10,#2/10)--+(#3*45:2.5);}}
\newcommand\xxl[2]{\xx{#1}{#2}3}
\newcommand\xxr[2]{\xx{#1}{#2}1}
\newcommand\xxlr[2]{\xxl{#1}{#2}\xxr{#1}{#2}}
\newcommand\xxh[6]{
\draw(#1/10,#2/10)+(0.5*#3*45+0.5*#4*45:#6) node[above] {$#5$};}
\newcommand\xxhu[4][0.15]{\xxh{#2}{#3}13{#4}{#1}}
\newcommand\stree[1]{\XX{\xxhu[0.25]00{#1}}}

\nc{\dnx}{\Delta_n A} \nc{\dx}{\Delta A} \nc{\dgp}{{\rm deg_{P}}}
\nc{\dgt}{{\rm deg_{T}}} \nc{\dg}{{\rm deg}} \nc{\ida}{ID($A$)} \nc{\tu}{\tilde{u}} \nc{\tv}{\tilde{v}}
\nc{\nr}{\calr_n} \nc{\nz}{\calz_n} \nc{\fun}{\cala_{n,d}}
 \nc{\fbase}{\calb} \nc{\LF}{\mathrm{RF}} \nc{\FFA}{\mathrm{LF}} \nc{\irr}{\mathrm{Irr}}
 \nc{\result}{\bfk\mathrm{Irr}(S_n)}  \nc{\I}{I_{\mathrm{ID},n}^0}
 \nc{\nrs}{\calr_n^\star} \nc{\ii}{\mathrm{I}} \nc{\iii}{\mathrm{II}}
\nc{\intl}{{\rm int}}\nc{\ws}[1]{{#1}}\nc{\deleted}[1]{\delete{#1}}\nc{\plas}{placements\xspace}

\nc{\bim}[1]{#1}  \nc{\shaop}{\sha_{\Omega}^{+}}  \nc{\shao}{\sha_{\Omega}}
\nc{\bbim}[2]{#1 #2} \nc{\bbbim}[2]{#1,\, #2} \nc{\RBF}{{\rm RBF}}
\nc{\frb}{F_{\RB}} \nc{\shaf}{\ssha_{\tiny{\Omega}}} \nc{\sham}{\diamond_{\tiny{\Omega}}}
\nc{\lf}{\lfloor} \nc{\rf}{\rfloor} \nc{\shan}{\ssha_{\lambda}}
\nc{\rlex}{{\rm {lex}}} \nc{\bb}{\Box} \nc{\ra}{\rightarrow}
\nc{\e}{{\rm {e}}}
\nc{\DDF}{\mathrm{DD}(X,\,\Omega)}\nc{\DTF}{\mathrm{DT}(X,\,\Omega)} \nc{\DT}{\mathrm{DT}'(\Omega,\,V)}
\nc{\bra}{\mathrm{bra}} \nc{\bre}{\mathrm{bre}}
\nc{\dec}{\mathrm{dec}} \nc{\diamondw}{\diamond_{w}}
\nc{\type}{\mathrm{type}}

\nc\caF[1]{\cal{F}_{#1}(X,\,\Omega)}
\nc\calt{\cal{T}(X,\,\Omega)} \nc\caltn{\cal{T}_n(X,\,\Omega)}
\nc\caltbin{\cal{T}_b(X,\,\Omega)}
\nc\calta{\cal{T}_0(X,\,\Omega)}
\nc\caltb{\cal{T}_1(X,\,\Omega)}
\nc\caltc{\cal{T}_2(X,\,\Omega)}
\nc\caltd{\cal{T}_3(X,\,\Omega)}
\nc\caltm{\cal{T}_m(X,\,\Omega)}
\nc\calf{\cal{F}(X,\,\Omega)}
\nc\fram{\frak{M}(\Omega,\, X)}
\nc\shaw{\sha^{NC}_w(\Omega,\, X)}
\nc\dw{\diamond_w} \nc\dl{\diamond_\ell}
\nc\shal{\sha^{NC}_\ell(X,\, \Omega)} \nc\shav{\sha^{NC}_w(\Omega,\, V)} \nc\shat{\sha^{NC,1}_w(\Omega,\, T^{+}(V))}
\nc{\cfo}{\cal{F}(X,\,\Omega)}
\nc{\frat}{\mathfrak{T}}
\nc{\shh}{\mathrm{Sh}^+_{\Omega}}
\nc{\sh}{\mathrm{Sh}_{\Omega}}

\nc{\lar}{\varinjlim}
\nc\XO{(X,\,\Omega)}
\def\cxo#1#2;{\cal{#1}#2\XO}
\def\cxob#1#2;{\cal{#1}#2_b\XO}
\nc\lrf[2]{B_{#2}^+(#1)}
\nc{\fd}{\mathrm{\text{typed angularly decorated planar rooted trees}}}
\nc{\rb}{\mathrm{RBFWs}} \nc{\dfw}{\mathrm{DFW{(X)}}} \nc{\tfw}{\mathrm{TFW{(X)}}}
\nc{\tfv}{\mathrm{TFW{(V)}}} \nc{\rbf}{\mathrm{RBF}}

\nc{\ets}{\mathrm{ETS}}

\def\Ve#1,#2,#3;{\vee_{#1,\,(#2,\,#3)}}
\def\bigv#1;#2;#3;{\bigvee\nolimits_{#1}^{#2;\,#3}}

\newcommand{\K}{\mathbf{k}}
\renewcommand{\geq}{\geqslant}
\renewcommand{\leq}{\leqslant}

\begin{document}

\title[Typed  trees and generalized Rota-Baxter algebras]{Typed angularly decorated planar rooted trees and generalized Rota-Baxter algebras}


\author{Lo\"\i c Foissy}
\address{Corresponding author\\
Univ. Littoral C\^ote d'Opale, UR 2597 LMPA, Laboratoire de Math\'ematiques Pures et Appliqu\'ees Joseph Liouville F-62100 Calais, France}
\email{loic.foissy@univ-littoral.fr}

\author{Xiao-Song Peng}
\address{School of Mathematics and Statistics, Lanzhou University, Lanzhou, Gansu 730000, China}
\email{pengxiaosong3@163.com}

\date{\today}

\begin{abstract}
We introduce a generalization of parametrized Rota-Baxter algebras, named $\Omega$-Rota-Baxter algebra, which includes family and matching Rota-Baxter algebras.
We study the structure needed on the set $\Omega$ of parameters in order to obtain that free $\Omega$-Rota-Baxter algebras
are described in terms of typed and angularly decorated planar rooted trees:
we obtain the notion of $\lambda$-extended diassociative semigroup, which includes sets (for matching Rota-Baxter algebras)
and semigroups (for family Rota-Baxter algebras), and many other examples.
We also describe free commutative $\Omega$-Rota-Baxter algebras generated by a commutative algebra $A$
in terms of typed words.
\end{abstract}

\subjclass[2020]{17B38, 05C05, 16S10, 08B20}

\keywords{Generalized Rota-Baxter algebra; diassociative semigroups; planar rooted trees; typed words}

\maketitle

\tableofcontents

\setcounter{section}{0}

\allowdisplaybreaks
\section{Introduction}

A Rota-Baxter algebra is an associative algebra $A$ with a linear endomorphism $P:A\longrightarrow A$, such that
for any $a,b\in A$,
\[P(a)P(b)=P(aP(b))+P(P(a)b)+\lambda P(ab),\]
where $\lambda$ is a scalar called the weight of the Rota-Baxter operator $P$. Firstly introduced by Baxter \mcite{Baxter60}
in a context of probability theory and popularized by Rota \cite{Rota1,Rota2,Rota3},
they now appear in numerous fields of mathematics and physics, see for example \cite{EFG082} for examples and more details.

The first appearance of family Rota-Baxter algebras seems to be in \cite{EFGBP07}, in the context of Renormalization
in Quantum Field Theories. This terminology, due to Li Guo \cite{Guo09} refers to an associative algebra $A$
with a family of linear endomorphims $P_\alpha:A\longrightarrow A$ indexed by the elements of a semigroup $(\Omega,*)$,
such that for any $a,b\in A$, for any $\alpha,\beta \in \Omega$,
\[P_\alpha(a)P_\beta(b)=P_{\alpha*\beta}(P_\alpha(a)b+aP_\beta(b)+\lambda ab).\]
This notion of matching Rota-Baxter algebra is introduced in \cite{ZGL20}.
This time, the Rota-Baxter operators are indexed by the elements of a set $\Omega$ with no structure, and the weights
are given by a family of scalars $(\lambda_\alpha)_{\alpha\in \Omega}$. For any $a,b\in A$, for any $\alpha,\beta \in \Omega$,
\[P_\alpha(a)P_\beta(b)=P_\beta(P_\alpha(a)b)+P_\alpha(aP_\beta(b))+\lambda_\beta P_\alpha(ab).\]
These notions have been extended to other types of algebras (Lie, pre-Lie, dendriform$\ldots$), see for example
\cite{ZGL20,ZGG20,ZG19,ZGM}.

Our aim here is a generalization of both family and matching Rota-Baxter algebras, in the spirit of what is made in \cite{Foi20}
for dendriform algebras. We here consider that the set of parameters $\Omega$ is given five operations
$\leftarrow, \rightarrow, \lhd, \rhd$  and $\cdot$, and a family of scalars 
$\lambda=(\lambda_{\alpha,\beta})_{\alpha,\beta \in \Omega}$.
An $\Omega$-Rota-Baxter algebra of weight $\lambda$ is an associative algebra $A$
with a family of linear endomorphisms indexed by $\Omega$ such that for any $a,b\in A$, for any $\alpha,\beta \in \Omega$,
\[P_{\alpha}(a)P_{\beta}(b)=P_{\alpha \rightarrow \beta}(P_{\alpha \rhd \beta}(a)b)+P_{\alpha \leftarrow \beta}(a P_{\alpha \lhd \beta}(b))+ \lambda_{\alpha, \beta} P_{\alpha \cdot \beta} (ab).\]
Taking
\begin{align*}
\alpha\rightarrow \beta&=\alpha\leftarrow \beta=\alpha \cdot \beta=\alpha *\beta,&
\alpha \rhd \beta&=\alpha,&
\alpha\lhd \beta&=\beta,
\end{align*}
and $\lambda_{\alpha,\beta}$ being constant, we recover in this way family Rota-Baxter algebras. Taking
\begin{align*}
\alpha\rightarrow \beta&=\beta,&\alpha\leftarrow \beta&=\alpha,&\alpha \cdot \beta&=\alpha,&
\alpha \rhd \beta&=\alpha,&\alpha\lhd \beta&=\beta,
\end{align*}
and $\lambda_{\alpha,\beta}$ depending only on $\beta$, w recover matching Rota-Baxter algebras.

For any  set $\Omega$ with five operations and any family of scalars $\lambda$, we define an operad  and a category of
$\Omega$-Rota-Baxter algebras (Definition \ref{def:orb}).
 This is far too general, and we impose the extra constraint that the combinatorics
of Rota-Baxter algebras is somehow preserved. To be more precise, as free Rota-Baxter algebras are based on
planar rooted trees \cite{ZGM},
we impose that free $\Omega$-Rota-Baxter algebras own a description in terms of angularly decorated (by the set of generators)
and typed (by $\Omega$) planar rooted trees, that is to say in terms of planar rooted trees
with angles decorated by the generators and internal edges decorated by elements of $\Omega$,
with an inductive description of the associative product and the Rota-Baxter operators being given by the grafting on a new root,
the created internal edge begin of the required type. We show in Theorem \ref{propRB} that this imposes strong constraints
on $\Omega$: we obtain that this combinatorial description holds  if, and only if $\Omega$ is a $\lambda$-$\ets$, as defined in Definition \ref{def:leds}. In particular,
$(\Omega,\leftarrow,\rightarrow)$ has to be a diassocative semigroup: for any $\alpha,\beta,\gamma\in \Omega$,
\begin{align*}
(\alpha \leftarrow \beta) \leftarrow \gamma=&\ \alpha \leftarrow (\beta \leftarrow \gamma)= \alpha \leftarrow (\beta \rightarrow \gamma),\\
(\alpha \rightarrow \beta) \leftarrow \gamma=&\ \alpha \rightarrow (\beta \leftarrow \gamma),\\
(\alpha \rightarrow \beta) \rightarrow \gamma=&\ (\alpha \leftarrow \beta) \rightarrow \gamma=\alpha \rightarrow (\beta \rightarrow \gamma).
\end{align*}
This notion firstly appeared in Loday's work~\cite{Lod01} under the name of (associative) dimonoid;
the free dimonoid is also constructed  in Loday's article.
Moreover, $(\Omega,\leftarrow,\rightarrow,\lhd,\rhd)$ is an extended semigroup (see Definition \ref{defEDS} below),
a notion used in \cite{Foi20} for parametrization of dendriform algebras. Particular examples of $\lambda$-$\ets$ attached to a set
give matching Rota-Baxter algebras (see Example \ref{ex2.4}-(b), with $\psi_\cdot(\alpha \otimes \beta)=\lambda \alpha$)
and particular examples of $\lambda$-$\ets$ attached to a semigroup gives family Rota-Baxter algebras
(see Example \ref{ex2.4}-(c)). In the case of weight 0, we obtain the generalization of the result \cite{EFG082}
establishing that any Rota-Baxter of weight 0 is a dendriform algebra, see Proposition \ref{propdend}.
Moreover, generalizing the construction of free commutative Rota-Baxter algebras,
we obtain that free commutative $\Omega$-Rota-Baxter algebras can be described in terms of $\Omega$-typed words
(Proposition \ref{propcommRB} and Theorem \ref{theocomRB}). \\

This paper is organised as follows. The first section introduces the definitions of EDS, $\lambda$-$\ets$,
$\ets$ and of $\Omega$-Rota-Baxter algebras.
The main result on free Rota-Baxter algebras and $\lambda$-$\ets$
is then proved (Theorem \ref{propRB}), with a description of free $\Omega$-Rota-Baxter algebras
in terms of trees. The last subsection deals with commutative $\Omega$-Rota-Baxter algebras and their description
in terms of typed words (Theorem \ref{theocomRB}). The second section gives more examples
of $\lambda$-$\ets$ and $\ets$, and in particular a classification of these objects of cardinality 2.\\

{\bf Notation.} Throughout this paper, {\bfk} is a unitary commutative ring which will be the base
ring of all modules, algebras, as well as linear maps.

\section{$\Omega$-Rota-Baxter algebras}
\subsection{Definitions}
We first recall the definition of diassociative semigroups and extended diassociative semigroups of~\cite{Foi20},
where these objects were used for parametrized versions of dendriform algebras.

\begin{defn}\cite{Foi20, Lod01}
A {\bf diassociative semigroup} is a family $(\Omega, \leftarrow, \rightarrow)$, where $\Omega$ is a set and $\leftarrow, \rightarrow: \Omega \times \Omega \rightarrow \Omega$ are maps such that
\begin{align*}
(\alpha \leftarrow \beta) \leftarrow \gamma=&\ \alpha \leftarrow (\beta \leftarrow \gamma)= \alpha \leftarrow (\beta \rightarrow \gamma),\\
(\alpha \rightarrow \beta) \leftarrow \gamma=&\ \alpha \rightarrow (\beta \leftarrow \gamma),\\
(\alpha \rightarrow \beta) \rightarrow \gamma=&\ (\alpha \leftarrow \beta) \rightarrow \gamma=\alpha \rightarrow (\beta \rightarrow \gamma),
\end{align*}
for all $\alpha, \beta, \gamma \in \Omega.$
\end{defn}

\begin{defn} \label{defEDS} \cite[Definition~2]{Foi20}
An {\bf extended diassociative semigroup} (abbr. EDS) is a family $(\Omega, \leftarrow, \rightarrow, \lhd,\rhd)$, where $\Omega$ is a set and $\leftarrow,\rightarrow, \lhd,\rhd: \Omega \times \Omega \rightarrow \Omega$ such that $(\Omega, \leftarrow, \rightarrow)$ is a diassociative semigroup and
\begin{align}
\label{EQ1}\alpha \rhd (\beta \leftarrow \gamma)=&\ \alpha \rhd \beta,\\
(\alpha \rightarrow \beta) \lhd \gamma=&\ \beta \lhd \gamma,\\
(\alpha \lhd \beta) \leftarrow ((\alpha \leftarrow \beta) \lhd \gamma)=&\ \alpha \lhd (\beta \leftarrow \gamma),\\
(\alpha \lhd \beta) \lhd ((\alpha \leftarrow \beta) \lhd \gamma)=&\ \beta \lhd \gamma,\\
(\alpha \lhd \beta) \rightarrow ((\alpha \leftarrow \beta) \lhd \gamma)=&\ \alpha \lhd (\beta \rightarrow \gamma),\\
(\alpha \lhd \beta) \rhd ((\alpha \leftarrow \beta) \lhd \gamma)=&\ \beta \rhd \gamma,\\
(\alpha \rhd (\beta \rightarrow \gamma)) \leftarrow (\beta \rhd \gamma)=&\ (\alpha \leftarrow \beta) \rhd \gamma,\\
(\alpha \rhd (\beta \rightarrow \gamma)) \lhd (\beta \rhd \gamma)=&\ \alpha \lhd \beta,\\
(\alpha \rhd (\beta \rightarrow \gamma)) \rightarrow (\beta \rhd \gamma)=&\ (\alpha \rightarrow \beta) \rhd \gamma,\\
\label{EQ10} (\alpha \rhd (\beta \rightarrow \gamma)) \rhd (\beta \rhd \gamma)=&\ \alpha \rhd \beta,
\end{align}
for all $\alpha,\beta, \gamma \in \Omega$.
\end{defn}

We shall use here the notion of $\lambda$-extended triassociative semigroup, where a family of scalars plays the role
of weights.

\begin{defn}\label{def:lambdaETS}
An {\bf $\lambda$-extended triasssociative semigroup}  (abbr. $\lambda$-$\ets$) is a family $(\Omega, \leftarrow, \rightarrow, \lhd,\rhd, \cdot, \ast, \lambda)$, where $(\Omega, \leftarrow, \rightarrow, \lhd, \rhd)$ is an EDS and
$\lambda=(\lambda_{\alpha,\beta})_{\alpha,\beta \in \Omega}$ is a family of elements in ${\bf k}$ indexed by $\Omega^2$ such that
\begin{align}
\label{EQ11}\lambda_{\alpha \rightarrow \beta, \gamma}=&\ \lambda_{\beta, \gamma}\\
\label{EQ12}\lambda_{\alpha \lhd \beta, (\alpha \leftarrow \beta) \lhd \gamma}=&\ \lambda_{\beta, \gamma}\\
\label{EQ13}\lambda_{\alpha \leftarrow \beta, \gamma}=&\ \lambda_{\alpha, \beta \rightarrow \gamma}\\
\label{EQ14}\lambda_{\alpha \rhd (\beta \rightarrow \gamma),\beta \rhd \gamma}=&\ \lambda_{\alpha, \beta}\\
\label{EQ15}\lambda_{\alpha,\beta}=&\ \lambda_{\alpha,\beta \leftarrow \gamma}\\
\label{EQ16} \lambda_{\alpha,\beta} \lambda_{\alpha \cdot \beta, \gamma}=&\ \lambda_{\beta,\gamma}
\lambda_{\alpha, \beta \cdot \gamma}
\end{align}
and, for all $\alpha, \beta, \gamma \in \Omega$:
\begin{enumerate}
\item If $\lambda_{\alpha \rightarrow \beta, \gamma}= \lambda_{\beta, \gamma}\neq 0$, then
\begin{align}
\label{eq17} \alpha \rhd \beta=&\ \alpha \rhd (\beta \cdot \gamma),\\
\label{eq18} (\alpha \rightarrow \beta) \cdot \gamma=&\ \alpha \rightarrow (\beta \cdot \gamma).
\end{align}
\item If $\lambda_{\alpha \lhd \beta, (\alpha \leftarrow \beta) \lhd \gamma}= \lambda_{\beta, \gamma}\neq 0$,
then
\begin{align}
\label{eq19} (\alpha \lhd \beta) \cdot ((\alpha \leftarrow \beta) \lhd \gamma)=&\ \alpha \lhd (\beta \cdot \gamma),\\
\label{eq20} (\alpha \leftarrow \beta) \leftarrow \gamma=&\ \alpha \leftarrow (\beta \cdot \gamma).
\end{align}
\item If $\lambda_{\alpha \rhd (\beta \rightarrow \gamma),\beta \rhd \gamma}= \lambda_{\alpha, \beta}\neq 0$,
then
\begin{align}
\label{eq21}\alpha \rightarrow (\beta \rightarrow \gamma)=&\ (\alpha \cdot \beta) \rightarrow \gamma,\\
\label{eq22} (\alpha \rhd (\beta \rightarrow \gamma)) \cdot (\beta \rhd \gamma)=&\ (\alpha \cdot \beta) \rhd \gamma.
\end{align}
\item  If $\lambda_{\alpha \leftarrow \beta, \gamma}=\ \lambda_{\alpha, \beta \rightarrow \gamma}\neq 0$,
then
\begin{align}
\label{eq23} (\alpha \leftarrow \beta) \cdot \gamma=&\ \alpha \cdot (\beta \rightarrow \gamma),\\
\label{eq24} \alpha \lhd \beta=&\ \beta \rhd \gamma.
\end{align}
\item If $\lambda_{\alpha, \beta}= \lambda_{\alpha, \beta \leftarrow \gamma}\neq 0$,
then
\begin{align}
\label{eq25} (\alpha \cdot \beta) \lhd \gamma=&\ \beta \lhd \gamma,\\
\label{eq26} (\alpha \cdot \beta) \leftarrow \gamma=&\ \alpha \cdot (\beta \leftarrow \gamma).
\end{align}
\item If $\lambda_{\alpha, \beta} \lambda_{\alpha \cdot \beta, \gamma}
= \lambda_{\beta, \gamma} \lambda_{\alpha, \beta \cdot \gamma}\neq 0$, then
\begin{align}
\label{eq27} (\alpha \cdot \beta) \cdot \gamma=&\ \alpha \cdot (\beta \cdot \gamma).
\end{align}
\end{enumerate}
\label{def:leds}
\end{defn}

\begin{exam}\label{ex2.4}
\begin{enumerate}
\item Let $(\Omega,\leftarrow,\rightarrow,\lhd,\rhd)$ be an EDS.
If we put $\lambda_{\alpha,\beta}=0$ for any $\alpha,\beta \in \Omega$,
then for any product $\cdot$, $(\Omega,\leftarrow,\rightarrow,\lhd,\rhd,\cdot,\lambda)$
is a $\lambda$-$\ets$.
\item If for any $\alpha,\beta \in \Omega$,
\[\alpha \leftarrow \beta=\beta \rightarrow \alpha=\beta \lhd \alpha=\alpha \rhd \beta=\alpha,\]
then $(\Omega,\leftarrow,\rightarrow,\lhd,\rhd,\cdot,\lambda)$
is a $\lambda$-$\ets$ if, and only if,  the following map defines an associative product:
\[\psi_\cdot:\left\{\begin{array}{rcl}
\K\Omega\otimes \K\Omega&\longrightarrow&\K\Omega\\
\alpha\otimes \beta&\longrightarrow&\lambda_{\alpha,\beta} \alpha \cdot \beta.
\end{array}\right.\]
Indeed, for any $\alpha$, $\beta$, $\gamma \in \Omega$,
\begin{align*}
\psi_\cdot\circ (\psi_\cdot\otimes \id)(\alpha \otimes \beta \otimes \gamma)
&=\lambda_{\alpha,\beta}\lambda_{\alpha\cdot \beta,\gamma}(\alpha \cdot \beta)\cdot \gamma,\\
\psi_\cdot\circ (\id \otimes \psi_\cdot)(\alpha \otimes \beta \otimes \gamma)
&=\lambda_{\beta,\gamma}\lambda_{\alpha, \beta \cdot\gamma}\alpha \cdot (\beta\cdot \gamma),
\end{align*}
which gives the missing condition (\ref{eq27}).
\item Let $(\Omega,\star)$ be a semigroup and $\lambda \in \K$. We put, for any $\alpha,\beta\in \Omega$:
\begin{align*}
\alpha \leftarrow \beta&=\alpha\star \beta,&\alpha \lhd \beta&=\beta,\\
\alpha \rightarrow \beta&=\alpha \star \beta,&\alpha \rhd \beta&=\alpha,\\
\lambda_{\alpha,\beta}&=\lambda,&\alpha\cdot \beta&=\alpha \star \beta.
\end{align*}
Then $(\Omega,\leftarrow,\rightarrow,\lhd,\rhd,\cdot,\lambda)$
is a $\lambda$-$\ets$.
\item Let $(\Omega,\star)$ be an abelian group and let $\lambda \in \K$. For any $\alpha,\beta \in \Omega$, we put:
\begin{align*}
\alpha \leftarrow \beta&=\alpha,&\alpha \rightarrow \beta&,\beta,\\
\alpha \lhd \beta&=\alpha \star \beta^{\star -1},&
\alpha \rhd \beta&=\alpha^{\star -1}\star \beta,\\
\lambda_{\alpha,\beta}&=\lambda ,&\alpha \cdot \beta&=\alpha.
\end{align*}
Then $(\Omega,\leftarrow,\rightarrow,\lhd,\rhd,\cdot,\lambda)$ is a $\lambda$-$\ets$.
\item Let $\Omega=(\Omega,\leftarrow,\rightarrow,\lhd,\rhd,\cdot,\lambda)$ be a $\lambda$-$\ets$.
For any $\alpha,\beta \in \Omega$, we put
\begin{align*}
\alpha \leftarrow^{op}\beta&=\beta \rightarrow \alpha,&
\alpha \lhd^{op}\alpha&=\beta \rhd \alpha,\\
\alpha \rightarrow^{op}\beta&=\beta \leftarrow \alpha,&
\alpha \rhd^{op}\alpha&=\beta \lhd \alpha,\\
\alpha \cdot^{op}\beta&=\beta \cdot \alpha,&
\lambda^{op}_{\alpha,\beta}&=\lambda_{\beta,\alpha}.
\end{align*}
Then $(\Omega,\leftarrow^{op},\rightarrow^{op},\lhd^{op},\rhd^{op},\cdot^{op},\lambda^{op})$
 is also a $\lambda$-$\ets$, called the \textbf{opposite} of $\Omega$ and denoted by $\Omega^{op}$.
 We shall say that $\Omega$ is commutative if it is equal to its opposite.
\end{enumerate}
\end{exam}

\begin{defn}
A {\bf extended triasssociative semigroup}  (abbr. $\ets$) is a family $(\Omega, \leftarrow, \rightarrow, \lhd,\rhd, \cdot, \ast)$, where $(\Omega, \leftarrow, \rightarrow, \lhd, \rhd)$ is an EDS and
\begin{align}
\label{EQ17}(\alpha \rightarrow \beta) \ast \gamma=&\ \beta \ast \gamma,\\
\tag{\ref{eq17}}\alpha \rhd \beta=&\ \alpha \rhd (\beta \cdot \gamma),\\
\tag{\ref{eq18}}(\alpha \rightarrow \beta) \cdot \gamma=&\ \alpha \rightarrow (\beta \cdot \gamma),\\
\label{EQ20}(\alpha \lhd \beta) \ast ((\alpha \leftarrow \beta) \lhd \gamma)=&\ \beta \ast \gamma,\\
\tag{\ref{eq19}}(\alpha \lhd \beta) \cdot ((\alpha \leftarrow \beta) \lhd \gamma)=&\ \alpha \lhd (\beta \cdot \gamma),\\
\tag{\ref{eq20}}(\alpha \leftarrow \beta) \leftarrow \gamma=&\ \alpha \leftarrow (\beta \cdot \gamma),\\
\label{EQ23}(\alpha \rhd (\beta \rightarrow \gamma))\ast (\beta \rhd \gamma)=&\ \alpha \ast \beta,\\
\tag{\ref{eq21}}\alpha \rightarrow (\beta \rightarrow \gamma)=&\ (\alpha \cdot \beta) \rightarrow \gamma,\\
\tag{\ref{eq22}}(\alpha \rhd (\beta \rightarrow \gamma)) \cdot (\beta \rhd \gamma)=&\ (\alpha \cdot \beta) \rhd \gamma,\\
\label{EQ26}(\alpha \leftarrow \beta)\ast \gamma=&\ \alpha\ast (\beta \rightarrow \gamma),\\
\tag{\ref{eq23}}(\alpha \leftarrow \beta) \cdot \gamma=&\ \alpha \cdot (\beta \rightarrow \gamma),\\
\tag{\ref{eq24}}\alpha \lhd \beta=&\ \beta \rhd \gamma,\\
\label{EQ29}\alpha\ast \beta=&\ \alpha\ast (\beta \leftarrow \gamma),\\
\tag{\ref{eq25}}(\alpha \cdot \beta) \lhd \gamma=&\ \beta \lhd \gamma,\\
\tag{\ref{eq26}}(\alpha \cdot \beta) \leftarrow \gamma=&\ \alpha \cdot (\beta \leftarrow \gamma),\\
\label{EQ32}\alpha \ast \beta=&\ \alpha \ast (\beta \cdot \gamma),\\
\label{EQ33}(\alpha \cdot \beta) \ast \gamma=&\ \beta \ast \gamma,\\
\tag{\ref{eq27}}(\alpha \cdot \beta) \cdot \gamma=&\ \alpha \cdot (\beta \cdot \gamma).
\end{align}
\end{defn}

\begin{exam}
\begin{enumerate}
\item Let $(\Omega,\ast,\cdot)$ be a set with two products such that for any $\alpha,\beta,\gamma \in \Omega$:
\begin{align}
\label{eq32} \alpha \ast \beta=&\ \alpha \ast (\beta \cdot \gamma),\\
\label{eq33} (\alpha \cdot \beta) \ast \gamma=&\ \beta \ast \gamma,\\
\label{eq34} (\alpha \cdot \beta) \cdot \gamma=&\ \alpha \cdot (\beta \cdot \gamma).
\end{align}
We put, for any $\alpha,\beta \in \Omega$:
\begin{align*}
\alpha \leftarrow \beta&=\alpha,&\alpha \lhd \beta&=\beta,\\
\alpha \rightarrow \beta&=\alpha,&\alpha \rhd \beta&=\beta.
\end{align*}
Then $(\Omega, \leftarrow, \rightarrow, \lhd,\rhd, \cdot, \ast)$ is an $\ets$.
\item Let $(\Omega, \leftarrow, \rightarrow, \lhd,\rhd, \cdot, \ast)$ be an $\ets$.
For any $\alpha,\beta \in \Omega$, we put
\begin{align*}
\alpha \leftarrow^{op}\beta&=\beta \rightarrow \alpha,&
\alpha \lhd^{op}\alpha&=\beta \rhd \alpha,\\
\alpha \rightarrow^{op}\beta&=\beta \leftarrow \alpha,&
\alpha \rhd^{op}\alpha&=\beta \lhd \alpha,\\
\alpha \ast^{op}\beta&=\beta\ast\alpha,&
\alpha \cdot^{op}\beta&=\beta \cdot \alpha.
\end{align*}
Then $(\Omega,\leftarrow^{op},\rightarrow^{op},\lhd^{op},\rhd^{op},\ast^{op},\cdot^{op})$
 is also an $\ets$, called the opposite of $\Omega$. We shall say that $\Omega$ is commutative if it is equal to its opposite.
\end{enumerate}
\end{exam}

Actually, each $\ets$ induces a $\lambda$-$\ets$, as the following result indicates:

\begin{prop}
Let $(\Omega, \leftarrow, \rightarrow, \lhd,\rhd, \cdot, \ast)$ be an $\ets$ and let $(\mu_{\alpha})_{\alpha \in \Omega}$
be a family of scalars. For any $\alpha,\beta \in \Omega$, we put:
\[\lambda_{\alpha,\beta}=\mu_{\alpha \ast \beta}.\]
Then  $(\Omega, \leftarrow, \rightarrow, \lhd,\rhd, \cdot, \lambda)$ is a $\lambda$-$\ets$.
\end{prop}

\begin{proof}
Conditions (a)-(f) of Definition \ref{def:lambdaETS} are obviously satisfied by (\ref{eq17})-(\ref{eq27}).
(\ref{EQ11}) is (\ref{EQ17}),     (\ref{EQ12}) is (\ref{EQ20}), (\ref{EQ13}) is (\ref{EQ26}),
(\ref{EQ14}) is (\ref{EQ23}), (\ref{EQ15}) is (\ref{EQ29}), and (\ref{EQ16}) comes from (\ref{EQ32}) and (\ref{EQ33}).
\end{proof}

We now propose the concept of $\Omega$-Rota-Baxter algebras as follows:
\begin{defn}\label{def:orb}
Let $\Omega$ be a set with five products $\leftarrow, \rightarrow, \lhd, \rhd, \cdot$ and
$\lambda=(\lambda_{\alpha,\beta})_{\alpha,\beta \in \Omega}$ be a family of elements in ${\bf k}$ indexed by $\Omega^2$.
An {\bf $\Omega$-Rota-Baxter algebra} of weight $\lambda$ is a family $(A, (P_{\omega})_{\omega \in \Omega})$ where $A$ is an associative algebra and $P_{\omega}: A \otimes A \rightarrow A$ is a linear map for each $\omega \in \Omega$, such that
\begin{align*}
P_{\alpha}(a)P_{\beta}(b)=P_{\alpha \rightarrow \beta}(P_{\alpha \rhd \beta}(a)b)+P_{\alpha \leftarrow \beta}(a P_{\alpha \lhd \beta}(b))+ \lambda_{\alpha, \beta} P_{\alpha \cdot \beta} (ab),
\end{align*}
for all $a, b \in A$ and $\alpha, \beta \in \Omega$. If, further, $A$ is commutative, then $(A, (P_{\omega})_{\omega \in \Omega})$ is a {\bf commutative $\Omega$-Rota-Baxter algebra}.
\end{defn}

Taking all elements of $\lambda$ equal to 0, we get the concept of $\Omega$-Rota-Baxter algebras of weight 0:
\begin{defn}
Let $\Omega$ be a set with four products $\leftarrow, \rightarrow, \lhd, \rhd$. An {\bf $\Omega$-Rota-Baxter algebra of weight} 0 is a family $(A, (P_{\omega})_{\omega \in \Omega})$ where $A$ is an associative algebra and $P_{\omega}: A \otimes A \rightarrow A$ is a linear map for each $\omega \in \Omega$, such that
\begin{align*}
P_{\alpha}(a)P_{\beta}(b)=P_{\alpha \rightarrow \beta}(P_{\alpha \rhd \beta}(a)b)+P_{\alpha \leftarrow \beta}(a P_{\alpha \lhd \beta}(b)),
\end{align*}
for all $a, b \in A$ and $\alpha, \beta \in \Omega$.
\end{defn}

\begin{exam}
\begin{enumerate}
\item If $(\Omega, \star)$ is a semigroup, we recover the definition of Rota-Baxter family algebras~\mcite{Guo09, ZG19} by defining
\begin{align*}
\alpha \leftarrow \beta &= \alpha \rightarrow \beta =\alpha \cdot \beta= \alpha \star\beta, &
\alpha \rhd \beta&=\alpha, &\alpha \lhd \beta&=\beta,
\end{align*}
and requiring all elements of $\lambda$ to be equal. Note that this is the $\lambda$-$\ets$ of Example \ref{ex2.4} (c).

\item For a set $\Omega$, define
\begin{align*}
\alpha \rightarrow \beta&=\alpha \lhd \beta=\beta, &\alpha \rhd \beta&=\alpha \leftarrow \beta=\alpha \cdot \beta=\alpha,
\end{align*}
and $\lambda_{\alpha,\beta}=\lambda_\alpha$, for a family $(\lambda_{\alpha})_{\alpha \in \Omega}$
of elements of ${\bf k }$.
Then we get the concept of matching Rota-Baxter algebra ~\mcite{ZGG20}, up to the change of the product of $A$
into its opposite.
\end{enumerate}
\end{exam}

As we know, Rota-Baxter algebras of weight 0 induce dendriform algebras \cite{EFG082}.
Similarly, we can show that each $\Omega$-Rota-Baxter algebra of weight 0 has a structure of an $\Omega$-dendriform algebra~\cite[definition 11]{Foi20}:

\begin{prop} \label{propdend}
Let $\Omega$ be a set with four products $\leftarrow, \rightarrow, \lhd, \rhd$ and $(A, (P_{\omega})_{\omega \in \Omega})$ an $\Omega$-Rota-Baxter algebra of weight 0. Then $(A, (\prec_{\omega})_{\omega \in \Omega}, (\succ_{\omega})_{\omega \in \Omega})$ is an $\Omega$-dendriform algebra, where
\begin{align*}
a \prec_{\omega} b&:=aP_{\omega}(b), & a \succ_{\omega} b&:=P_{\omega}(a)b,
\end{align*}
for all $a, b \in A$ and $\omega \in \Omega$.
\end{prop}

\begin{proof}
For $a,b,c \in A$ and $\alpha, \beta \in \Omega$,
\begin{align*}
(a \prec_{\alpha} b) \prec_{\beta} c=&\ \big(aP_{\alpha}(b)\big)P_{\beta}(c)=a\big(P_{\alpha}(b)P_{\beta}(c)\big)=a\big(P_{\alpha \rightarrow \beta}(P_{\alpha \rhd \beta}(b)c)+ P_{\alpha \leftarrow \beta} (b P_{\alpha \lhd \beta}(c))\big)\\
=&\ aP_{\alpha \rightarrow \beta}(P_{\alpha \rhd \beta}(b)c)+ aP_{\alpha \leftarrow \beta} (b P_{\alpha \lhd \beta}(c))=a \prec_{\alpha \rightarrow \beta} (b \succ_{\alpha \rhd \beta} c)+ a\prec_{\alpha \leftarrow \beta} (b \prec_{\alpha \lhd \beta} c),\\
a\succ_{\alpha}(b \prec_{\beta} c)=&\ P_{\alpha}(a)\big(bP_{\beta}(c)\big)=\big(P_{\alpha}(a)b\big)P_{\beta}(c)=(a \succ_{\alpha} b) \prec_{\beta} c,\\
a \succ_{\alpha} (b \succ_{\beta} c)=&\ P_{\alpha}(a) \big(P_{\beta}(b)c \big)=\big(P_{\alpha}(a)P_{\beta}(b)\big)c=\big(P_{\alpha \rightarrow \beta} (P_{\alpha \rhd \beta}(a)b)+ P_{\alpha \leftarrow \beta} (a P_{\alpha \lhd \beta}(b)) \big)c\\
=&\ P_{\alpha \rightarrow \beta} (P_{\alpha \rhd \beta}(a)b)c+P_{\alpha \leftarrow \beta} (a P_{\alpha \lhd \beta}(b)c=(a \succ_{\alpha \rhd \beta} b) \succ_{\alpha \rightarrow \beta} c+(a \prec_{\alpha \lhd \beta} b) \succ_{\alpha \leftarrow \beta} c.
\qedhere \end{align*}
\end{proof}

\subsection{$\Omega$-Rota-Baxter algebras on typed angularly decorated planar rooted trees}
First, let us recall some notations on planar rooted trees (see ~\mcite{ZGM} for more details).
For a planar rooted tree $T$, we shall consider the root and the leaves of $T$ as edges rather than vertices. Denote by $IE(T)$ the set of internal edges of $T$, i.e. edges which are neither leaves nor the root and denote by $V(T)$ the set of vertices of $T$. For each vertex $v$ yields a (possibly empty) set of angles $A(v)$, an angle being a pair $(e, e')$ of adjacent incoming edges for v.
Let $\displaystyle A(T)=\bigsqcup_{v \in V(T)}A(v)$ be the set of angles of $T$. Then:

\begin{defn}\cite[Definition~2.2]{ZGM}
Let $X$ and $\Omega$ be two sets. An $X$-{\bf angularly decorated $\Omega$-typed} (abbr. {\bf typed angularly decorated}) {\bf planar rooted tree} is a triple $T = (T, \dec, \type)$, where $T$ is a planar rooted tree,  $\dec : A(T) \rightarrow X$ and $\type : IE(T) \rightarrow  \Omega$ are maps.
\end{defn}

 For $n\geq 0$, let $\caltn$ denote the set of  $X$-angularly decorated $\Omega$-typed planar rooted trees with $n+1$ leaves and at least one internal vertex such that internal edges are decorated by elements of $\Omega$. We put
\begin{align*}
\calt&:= \bigsqcup_{n\geq 0}\caltn&&\text{ and }&\bfk\calt& :=\bigoplus_{n\geq 0}\bfk\caltn.
\end{align*}
For example,
\begin{align*}
\calta=& \left\{\treeoo{\cda[1.5] o\zhd{o/a}
\node[right] at ($(o)!0.7!(oa)$) {};
},\,
\treeoo{\cda[1.0] o\zhd{o/a}\cdb[1.0]o\zhd{o/b}
\node[right] at ($(o)!0.7!(oa)$) {};
\node[right] at ($(o)!0.2!(ob)$) {$\alpha$};
},\,
\treeoo{\cda[1.0] o\zhd{o/a}\cdb[1.0]o\zhd{o/b}
\node[right] at ($(o)!0.7!(oa)$) {};
\node[right] at ($(o)!0.2!(ob)$) {$\alpha$};
\cdb{ob}\zhd{ob/b}
\node[right] at ($(ob)!0.1!(obb)$) {$\beta$};
},\,
\ldots\big|\alpha,\beta,\cdots\in\Omega\right\},\,\,\\
\caltb =&\left\{\stree x,\,\treeoo{\cdb[1.5] o\zhd{o/b}\cdlr o
\node[above=1pt] at (o) {$x$};
\node[right] at ($(o)!0.3!(ob)$) {$\alpha$};
},\,
\treeoo{\cdb[1.5] o\zhd{o/b}\cdlr[1.5] o \zhd{o/l}
\node[above=1pt] at (o) {$x$};
\node[right] at ($(o)!0.3!(ob)$) {$\beta$};
\node[left] at ($(o)!0.1!(ol)$) {$\alpha$};
},\,
\treeoo{\cdb[1.5] o\zhd{o/b}\cdlr[1.5] o \zhd{o/l}\zhd{o/r}
\node[above=1pt] at (o) {$x$};
\node[right] at ($(o)!0.3!(ob)$) {$\gamma$};
\node[left] at ($(o)!0.1!(ol)$) {$\alpha$};
\node[right] at ($(o)!0.1!(or)$) {$\beta$};
},\,
\treeoo{\cdb[1.5] o\zhd{o/b}\cdlr[1.5] o 
\node at ($(o)!0.33!(ol)$) {\usebox\dbox};
\node at ($(o)!0.66!(ol)$) {\usebox\dbox};
\node[above=1pt] at (o) {$x$};
\node[right] at ($(o)!0.3!(ob)$) {$\gamma$};
\node[left] at ($(o)!0.01!(ol)$) {$\alpha$};
\node[left] at ($(o)!0.3!(ol)$) {$\beta$};
},\,
\treeoo{\cdb[1.5] o\zhd{o/b}\cdlr o
\node[above=1pt] at (o) {$x$};
\node[right] at ($(o)!0.3!(ob)$) {$\beta$};
\cdb{ob}\zhd{ob/b}
\node[right] at ($(ob)!0.1!(obb)$) {$\alpha$};
},
\ldots\big| x\in X,\alpha,\beta,\gamma,\cdots\in\Omega \right\},\\
 \caltc=& \ \left\{
\treeoo{\cdb o\zhd{o/b}\cdlr{o,ol}[1.5]\zhd{o/r}
\node[above=1pt] at (ol) {$x$};
\node[above=1pt] at (o) {$y$};
\node[right] at ($(o)!0.3!(ob)$) {$\gamma$};
\node[left=1pt] at ($(o)!0.3!(ol)$) {$\beta$};
\node[right] at ($(o)!0.1!(or)$) {$\alpha$};
},
\treeoo{\cdb o\cdlr{o,or}
\node[above=1pt] at (or) {$y$};
\node[above=1pt] at (o) {$x$};
\node[right] at ($(o)!0.3!(or)$) {$\alpha$};
},
\treeoo{\cdb o\zhd{o/b}\cdlr{o,or}
\node[above=1pt] at (or) {$y$};
\node[above=1pt] at (o) {$x$};
\node[right] at ($(o)!0.3!(ob)$) {$\beta$};
\node[right] at ($(o)!0.3!(or)$) {$\alpha$};
},
\treeoo{\cdb o\cda o
\cdx[1.5]{o}{ol}{150}\cdx[1.5]{o}{or}{30}\zhd{o/r}
\node[left] at ($(o)!0.5!(oa)$) {$x$};
\node[right] at ($(o)!0.5!(oa)$) {$y$};
\node[right] at ($(o)!0.1!(or)$) {$\alpha$};
},
\treeoo{\cdb o\cda o\zhd{o/b}
\cdx{o}{ol}{150}\cdx{o}{or}{30}
\node[left] at ($(o)!0.5!(oa)$) {$x$};
\node[right] at ($(o)!0.5!(oa)$) {$y$};
\node[right] at ($(o)!0.1!(ob)$) {$\alpha$};
},\,
\treeoo{\cdb o\cda o\zhd{o/b}
\cdx{o}{ol}{150}\cdx{o}{or}{30}
\node[left] at ($(o)!0.5!(oa)$) {$x$};
\node[right] at ($(o)!0.5!(oa)$) {$y$};
\node[right] at ($(o)!0.1!(ob)$) {$\beta$};
\cdb{ob}\zhd{ob/b}
\node[right] at ($(ob)!0.1!(obb)$) {$\alpha$};
},
\ldots
\big| x,y\in X, \alpha,\beta,\gamma,\cdots\in\Omega\right\},\\
\end{align*}
Graphically, an element $T\in \cxo T;$ is of the form:
\begin{equation*}
T = \treeoo{\cdb o\ocdx[2]{o}{a1}{160}{T_1}{left}
\ocdx[2]{o}{a2}{120}{T_2}{above}
\ocdx[2]{o}{a3}{60}{T_n}{above}
\ocdx[2]{o}{a4}{20}{T_{n+1}}{right}
\node at (140:\xch) {$x_1$};
\node at (90:\xch) {$\cdots$};
\node at (40:\xch) {$x_n$};
\node[below] at ($(o)!0.7!(a1)$) {$\alpha_1$};
\node[below] at ($(o)!0.7!(a4)$) {$\alpha_{n+1}$};
\node[right] at ($(o)!0.85!(a2)$) {$\alpha_2$};
\node[left] at ($(o)!0.85!(a3)$) {$\alpha_n\!$};
},\,\text{ with } n\geq 0, \text{ where }\, x_1,\cdots,x_n\in X,\, \,\alpha_i \in \Omega\, \text{ if $T_i\neq |$ and otherwise}
\end{equation*}
$\alpha_i$ does not exist for $1 \leq i \leq n+1$.

For each $\omega\in\Omega$, there is a grafting operator $B^{+}_{\omega}: \bfk \calt \ra \bfk \calt$ which add a new root to a tree and an new internal edge typed by $\omega$ between the new root and the root of the tree.

For example,
\begin{align*}
B^+_\omega\Big(\treeoo{\cda[1.0] o\zhd{o/a}\cdb[1.0]o\zhd{o/b}
\node[right] at ($(o)!0.7!(oa)$) {};
\node[right] at ($(o)!0.15!(ob)$) {$\alpha$};
}  \Big)&=\treeoo{\cda[1.0] o\zhd{o/a}\cdb[1.0]o\zhd{o/b}
\node[right] at ($(o)!0.7!(oa)$) {};
\node[right] at ($(o)!0.2!(ob)$) {$\alpha$};
\cdb{ob}\zhd{ob/b}
\node[right] at ($(ob)!0.1!(obb)$) {$\omega$};
}
,&B^{+}_\omega\Big(\treeoo{\cdb[1.5] o\zhd{o/b}\cdlr o
\node[above=1pt] at (o) {$x$};
\node[left] at ($(o)!0.3!(ob)$) {$\alpha$};
} \Big)&=\treeoo{\cdb[1.5] o\zhd{o/b}\cdlr o
\node[above=1pt] at (o) {$x$};
\node[right] at ($(o)!0.3!(ob)$) {$\alpha$};
\cdb{ob}\zhd{ob/b}
\node[right] at ($(ob)!0.1!(obb)$) {$\omega$};
}.
\end{align*}
The {\bf depth} $\dep{(T)}$ of a rooted tree $T$ is the maximal length of linear chains from the root to the leaves of the tree.
For example,
\begin{align*}
\dep\big(\treeoo{\cda[1.5] o\zhd{o/a}
\node[right] at ($(o)!0.7!(oa)$) {};
}\Big) &= \dep\Big(\stree x\Big) = 1&&\text{ and }& \dep \Big(\treeoo{\cda[1.0] o\zhd{o/a}\cdb[1.0]o\zhd{o/b}
\node[right] at ($(o)!0.7!(oa)$) {};
\node[right] at ($(o)!0.15!(ob)$) {$\omega$};
} \Big)=\dep\Big(\treeoo{\cdb[1.5] o\zhd{o/b}\cda o\cdb o
\cdx{o}{ol}{150}\cdx{o}{or}{30}
\node[left] at ($(o)!0.5!(oa)$) {$x$};
\node[right] at ($(o)!0.5!(oa)$) {$y$};
\node[left] at ($(o)!0.3!(ob)$) {$\alpha$};
}\Big) &= 2.
\end{align*}
We also consider the trivial tree $\vert$ and put by convention $\dep(\vert) :=0$.
For each typed angularly decorated planar rooted tree $T$, define the number of branches of $T$ to be $\bra(T)=0$ if $T=\vert$. Otherwise, $\dep(T)\geq 1$ and $T$ is of the form
\[T = \treeoo{\cdb o\ocdx[2]{o}{a1}{160}{T_1}{left}
\ocdx[2]{o}{a2}{120}{T_2}{above}
\ocdx[2]{o}{a3}{60}{T_n}{above}
\ocdx[2]{o}{a4}{20}{T_{n+1}}{right}
\node at (140:\xch) {$x_1$};
\node at (90:\xch) {$\cdots$};
\node at (40:\xch) {$x_n$};
\node[below] at ($(o)!0.7!(a1)$) {$\alpha_1$};
\node[below] at ($(o)!0.7!(a4)$) {$\alpha_{n+1}$};
\node[right] at ($(o)!0.85!(a2)$) {$\alpha_2$};
\node[left] at ($(o)!0.85!(a3)$) {$\alpha_n\!$};
}\,\text{ with }\, n\geq 0, \]
where $T_j\in \calt\sqcup\{\vert\}, j=1,\ldots, n+1$. We define $\bra(T):=n+1$.
For example,
\begin{align*}
 \bra\Big(\treeoo{\cda[1.5] o\zhd{o/a}
\node[right] at ($(o)!0.7!(oa)$) {};
}\Big)&=1,&\bra\Big(\treeoo{\cdb o\cdlr{o,or}
\node[above=1pt] at (or) {$y$};
\node[above=1pt] at (o) {$x$};
\node[right] at ($(o)!0.3!(or)$) {$\alpha$};
}\Big) &= 2
&&\text{ and }&
\bra\Big(\treeoo{\cdb o\cda o
\cdx[1.5]{o}{ol}{150}\cdx[1.5]{o}{or}{30}\zhd{o/r}
\node[left] at ($(o)!0.5!(oa)$) {$x$};
\node[right] at ($(o)!0.5!(oa)$) {$y$};
\node[right] at ($(o)!0.1!(or)$) {$\alpha$};
}\Big) &= 3.
\end{align*}

Let $X$ be a set, $(\Omega,\leftarrow, \rightarrow, \lhd, \rhd, \cdot)$ be  a set with five products,
and $\lambda=(\lambda_{\alpha,\beta})_{(\alpha,\beta) \in \Omega^2}$ be a family of elements in $\bfk$ indexed by $\Omega^2$.
By analogy with the construction of free Rota-Baxter algebras, 
we define a product $\diamond$ on $\bfk \calt$ as follows.
For $T,T'\in\calt$, we define $T\diamond T'$ by induction on $\dep(T)+\dep(T')\geq 2$.
For the initial step $\dep(T)+\dep(T')=2$, we have $\dep(T)=\dep(T')=1$ and $T, T'$ are of the form
\begin{align*}
T &=\treeoo{\cdb o\cdx[1.5]{o}{a1}{160}
\cdx[1.5]{o}{a2}{120}
\cdx[1.5]{o}{a3}{60}
\cdx[1.5]{o}{a4}{20}
\node at (140:1.2*\xch) {$x_1$};
\node at (90:1.2*\xch) {$\cdots$};
\node at (40:1.2*\xch) {$x_m$};
}&&\text{ and }&T'&=\treeoo{\cdb o\cdx[1.5]{o}{a1}{160}
\cdx[1.5]{o}{a2}{120}
\cdx[1.5]{o}{a3}{60}
\cdx[1.5]{o}{a4}{20}
\node at (140:1.2*\xch) {$y_1$};
\node at (90:1.2*\xch) {$\cdots$};
\node at (40:1.2*\xch) {$y_n$};
},&&\text{ with }\,m,n\geq 0.
\end{align*}
Define
\begin{equation}
T\diamond T':=\treeoo{\cdb o\cdx[1.5]{o}{a1}{160}
\cdx[1.5]{o}{a2}{120}
\cdx[1.5]{o}{a3}{60}
\cdx[1.5]{o}{a4}{20}
\node at (140:1.2*\xch) {$x_1$};
\node at (90:1.2*\xch) {$\cdots$};
\node at (40:1.2*\xch) {$x_m$};
}\diamond\treeoo{\cdb o\cdx[1.5]{o}{a1}{160}
\cdx[1.5]{o}{a2}{120}
\cdx[1.5]{o}{a3}{60}
\cdx[1.5]{o}{a4}{20}
\node at (140:1.2*\xch) {$y_1$};
\node at (90:1.2*\xch) {$\cdots$};
\node at (40:1.2*\xch) {$y_n$};
}:=\treeoo{\cdb o
\cdx[2.2]{o}{a1}{170}
\cdx[2.2]{o}{a2}{145}
\cdx[2.2]{o}{a3}{115}
\cdx[2.2]{o}{a4}{90}
\cdx[2.2]{o}{a5}{65}
\cdx[2.2]{o}{a6}{35}
\cdx[2.2]{o}{a7}{10}
\node at ($(a1)!0.5!(a2)$) {$x_1$};
\node[rotate=40] at ($(a2)!0.5!(a3)$) {$\cdots$};
\node at ($(a3)!0.5!(a4)$) {$x_m$};
\node at ($(a4)!0.5!(a5)$) {$y_1$};
\node[rotate=-40] at ($(a5)!0.5!(a6)$) {$\cdots$};
\node at ($(a6)!0.5!(a7)$) {$y_n$};
}.
\mlabel{eq:def1}
\end{equation}
For the induction step $\dep(T)+\dep(T')\geq 3$,  the trees $T$ and $T'$ are of the form
$$T=\treeoo{\cdb o\ocdx[2]{o}{a1}{160}{T_1}{left}
\ocdx[2]{o}{a2}{120}{T_2}{above}
\ocdx[2]{o}{a3}{60}{T_m}{above}
\ocdx[2]{o}{a4}{20}{T_{m+1}}{right}
\node at (140:\xch) {$x_1$};
\node at (90:\xch) {$\cdots$};
\node at (40:\xch) {$x_m$};
\node[below] at ($(o)!0.7!(a1)$) {$\alpha_1$};
\node[below] at ($(o)!0.7!(a4)$) {$\alpha_{m+1}$};
\node[right] at ($(o)!0.85!(a2)$) {$\alpha_2$};
\node[left] at ($(o)!0.85!(a3)$) {$\alpha_m\!$};
}\,\text{ and }\, T'=\treeoo{\cdb o\ocdx[2]{o}{a1}{160}{T'_1}{left}
\ocdx[2]{o}{a2}{120}{T'_2}{above}
\ocdx[2]{o}{a3}{60}{T'_n}{above}
\ocdx[2]{o}{a4}{20}{T'_{n+1}}{right}
\node at (140:\xch) {$y_1$};
\node at (90:\xch) {$\cdots$};
\node at (40:\xch) {$y_n$};
\node[below] at ($(o)!0.7!(a1)$) {$\beta_1$};
\node[below] at ($(o)!0.7!(a4)$) {$\beta_{n+1}$};
\node[right] at ($(o)!0.85!(a2)$) {$\beta_2$};
\node[left] at ($(o)!0.85!(a3)$) {$\beta_n\!$};
}\,\text{ with some }\, T_i\neq |\,\text{ or some}\, T'_j\neq |.$$ There are four cases to consider.

\noindent{\bf Case 1:} $T_{m+1}=|=T'_1$. Define

\begin{equation}
 T\diamond T':= \treeoo{\cdb o\ocdx[2]{o}{a1}{160}{T_1}{left}
\ocdx[2]{o}{a2}{120}{T_2}{above}
\ocdx[2]{o}{a3}{60}{T_m}{above}
\cdx[2]{o}{a4}{20}
\node at (140:\xch) {$x_1$};
\node at (90:\xch) {$\cdots$};
\node at (40:\xch) {$x_m$};
\node[below] at ($(o)!0.7!(a1)$) {$\alpha_1$};
\node[right] at ($(o)!0.85!(a2)$) {$\alpha_2$};
\node[left] at ($(o)!0.85!(a3)$) {$\alpha_m\!$};
}\diamond \treeoo{\cdb o
\cdx[2]{o}{a1}{160}
\ocdx[2]{o}{a2}{120}{T'_2}{above}
\ocdx[2]{o}{a3}{60}{T'_n}{above}
\ocdx[2]{o}{a4}{20}{T'_{n+1}}{right}
\node at (140:\xch) {$y_1$};
\node at (90:\xch) {$\cdots$};
\node at (40:\xch) {$y_n$};
\node[below] at ($(o)!0.7!(a4)$) {$\beta_{n+1}$};
\node[right] at ($(o)!0.85!(a2)$) {$\beta_2$};
\node[left] at ($(o)!0.85!(a3)$) {$\beta_n\!$};
}
:=\treeoo{\cdb o
\ocdx[2.5]{o}{a1}{170}{T_1}{170}
\ocdx[2.5]{o}{a2}{145}{T_2}{145}
\ocdx[2.5]{o}{a3}{115}{T_m}{115}
\cdx[2.5]{o}{a4}{90}{-90}
\ocdx[2.5]{o}{a5}{65}{T'_2}{65}
\ocdx[2.5]{o}{a6}{35}{T'_n}{35}
\ocdx[2.5]{o}{a7}{10}{T'_{n+1}}{10}
\node at (157:1.3*\xch) {$x_1$};
\node[rotate=40] at (129:2.5*\xch) {$\cdots$};
\node at (102:1.5*\xch) {$x_m$};
\node at (75:1.5*\xch) {$y_1$};
\node[rotate=-40] at (49:2.5*\xch) {$\cdots$};
\node at (20:1.5*\xch) {$y_n$};
\node[below] at ($(o)!0.7!(a1)$) {$\alpha_1$};
\node[below] at ($(o)!0.8!(a7)$) {\tiny$\beta_{n+1}$};
\node[below] at ($(o)!0.9!(a2)$) {\tiny$\alpha_2$};
\node[left=-1pt] at ($(o)!0.65!(a3)$) {\tiny$\alpha_m$};
\node[right=1pt] at ($(o)!0.65!(a5)$) {\tiny$\beta_2$};
\node[right=1pt] at ($(o)!0.65!(a6)$) {\tiny$\beta_n$};
}.
\mlabel{eq:def3}
\end{equation}

\noindent{\bf Case 2:} $T_{m+1}\neq|=T'_1$. Define
\begin{equation}
T\diamond T':=\treeoo{\cdb o\ocdx[2]{o}{a1}{160}{T_1}{left}
\ocdx[2]{o}{a2}{120}{T_2}{above}
\ocdx[2]{o}{a3}{60}{T_m}{above}
\ocdx[2]{o}{a4}{20}{T_{m+1}}{right}
\node at (140:\xch) {$x_1$};
\node at (90:\xch) {$\cdots$};
\node at (40:\xch) {$x_m$};
\node[below] at ($(o)!0.7!(a1)$) {$\alpha_1$};
\node[below] at ($(o)!0.7!(a4)$) {$\alpha_{m+1}$};
\node[right] at ($(o)!0.85!(a2)$) {$\alpha_2$};
\node[left] at ($(o)!0.85!(a3)$) {$\alpha_m\!$};
}\diamond \treeoo{\cdb o
\cdx[2]{o}{a1}{160}
\ocdx[2]{o}{a2}{120}{T'_2}{above}
\ocdx[2]{o}{a3}{60}{T'_n}{above}
\ocdx[2]{o}{a4}{20}{T'_{n+1}}{right}
\node at (140:\xch) {$y_1$};
\node at (90:\xch) {$\cdots$};
\node at (40:\xch) {$y_n$};
\node[below] at ($(o)!0.7!(a4)$) {$\beta_{n+1}$};
\node[right] at ($(o)!0.85!(a2)$) {$\beta_2$};
\node[left] at ($(o)!0.85!(a3)$) {$\beta_n\!$};
}:=\treeoo{\cdb o
\ocdx[2.5]{o}{a1}{170}{T_1}{170}
\ocdx[2.5]{o}{a2}{145}{T_2}{145}
\ocdx[2.5]{o}{a3}{115}{T_m}{115}
\ocdx[2.5]{o}{a4}{90}{T_{m+1}}{90}
\ocdx[2.5]{o}{a5}{65}{T'_2}{65}
\ocdx[2.5]{o}{a6}{35}{T'_n}{35}
\ocdx[2.5]{o}{a7}{10}{T'_{n+1}}{10}
\node at (157:1.3*\xch) {$x_1$};
\node[rotate=40] at (129:2.5*\xch) {$\cdots$};
\node at (102:1.5*\xch) {$x_m$};
\node at (75:1.5*\xch) {$y_1$};
\node[rotate=-40] at (49:2.5*\xch) {$\cdots$};
\node at (20:1.5*\xch) {$y_n$};
\node[below] at ($(o)!0.7!(a1)$) {$\alpha_1$};
\node[below] at ($(o)!0.8!(a7)$) {\tiny$\beta_{n+1}$};
\node[below] at ($(o)!0.9!(a2)$) {\tiny$\alpha_2$};
\node[left=-1pt] at ($(o)!0.65!(a3)$) {\tiny$\alpha_m$};
\node[right=1pt] at ($(o)!0.65!(a5)$) {\tiny$\beta_2$};
\node[right=1pt] at ($(o)!0.65!(a6)$) {\tiny$\beta_n$};
\node[right] at ($(o)!0.8!(a4)$) {\tiny$\alpha_{m+1}$};
}.
\mlabel{eq:def4}
\end{equation}

\noindent{\bf Case 3:} $T_{m+1}=|\neq T'_1$. Define
\begin{equation}
 T\diamond T':=\treeoo{\cdb o\ocdx[2]{o}{a1}{160}{T_1}{left}
\ocdx[2]{o}{a2}{120}{T_2}{above}
\ocdx[2]{o}{a3}{60}{T_m}{above}
\cdx[2]{o}{a4}{20}
\node at (140:\xch) {$x_1$};
\node at (90:\xch) {$\cdots$};
\node at (40:\xch) {$x_m$};
\node[below] at ($(o)!0.7!(a1)$) {$\alpha_1$};
\node[right] at ($(o)!0.85!(a2)$) {$\alpha_2$};
\node[left] at ($(o)!0.85!(a3)$) {$\alpha_m\!$};
}\diamond \treeoo{\cdb o\ocdx[2]{o}{a1}{160}{T'_1}{left}
\ocdx[2]{o}{a2}{120}{T'_2}{above}
\ocdx[2]{o}{a3}{60}{T'_n}{above}
\ocdx[2]{o}{a4}{20}{T'_{n+1}}{right}
\node at (140:\xch) {$y_1$};
\node at (90:\xch) {$\cdots$};
\node at (40:\xch) {$y_n$};
\node[below] at ($(o)!0.7!(a1)$) {$\beta_1$};
\node[below] at ($(o)!0.7!(a4)$) {$\beta_{n+1}$};
\node[right] at ($(o)!0.85!(a2)$) {$\beta_2$};
\node[left] at ($(o)!0.85!(a3)$) {$\beta_n\!$};
}:=\treeoo{\cdb o
\ocdx[2.5]{o}{a1}{170}{T_1}{170}
\ocdx[2.5]{o}{a2}{145}{T_2}{145}
\ocdx[2.5]{o}{a3}{115}{T_m}{115}
\ocdx[2.5]{o}{a4}{90}{T'_1}{90}
\ocdx[2.5]{o}{a5}{65}{T'_2}{65}
\ocdx[2.5]{o}{a6}{35}{T'_n}{35}
\ocdx[2.5]{o}{a7}{10}{T'_{n+1}}{10}
\node at (157:1.3*\xch) {$x_1$};
\node[rotate=40] at (129:2.5*\xch) {$\cdots$};
\node at (102:1.5*\xch) {$x_m$};
\node at (75:1.5*\xch) {$y_1$};
\node[rotate=-40] at (49:2.5*\xch) {$\cdots$};
\node at (20:1.5*\xch) {$y_n$};
\node[below] at ($(o)!0.7!(a1)$) {$\alpha_1$};
\node[below] at ($(o)!0.8!(a7)$) {\tiny$\beta_{n+1}$};
\node[below] at ($(o)!0.9!(a2)$) {\tiny$\alpha_2$};
\node[left=-1pt] at ($(o)!0.65!(a3)$) {\tiny$\alpha_m$};
\node[right=1pt] at ($(o)!0.65!(a5)$) {\tiny$\beta_2$};
\node[right=1pt] at ($(o)!0.65!(a6)$) {\tiny$\beta_n$};
\node[right] at ($(o)!0.8!(a4)$) {\tiny$\beta_{1}$};
} .
\mlabel{eq:def5}
\end{equation}

\noindent{\bf Case 4:} $T_{m+1}\neq|\neq T'_1$. Define
{\small{\begin{align} \mlabel{eq:def6}
 T\diamond T'
:=& \ \treeoo{\cdb o\ocdx[2]{o}{a1}{160}{T_1}{left}
\ocdx[2]{o}{a2}{120}{T_2}{above}
\ocdx[2]{o}{a3}{60}{T_m}{above}
\ocdx[2]{o}{a4}{20}{T_{m+1}}{right}
\node at (140:\xch) {$x_1$};
\node at (90:\xch) {$\cdots$};
\node at (40:\xch) {$x_m$};
\node[below] at ($(o)!0.7!(a1)$) {$\alpha_1$};
\node[below] at ($(o)!0.7!(a4)$) {$\alpha_{m+1}$};
\node[right] at ($(o)!0.85!(a2)$) {$\alpha_2$};
\node[left] at ($(o)!0.85!(a3)$) {$\alpha_m\!$};
}\diamond \treeoo{\cdb o\ocdx[2]{o}{a1}{160}{T'_1}{left}
\ocdx[2]{o}{a2}{120}{T'_2}{above}
\ocdx[2]{o}{a3}{60}{T'_n}{above}
\ocdx[2]{o}{a4}{20}{T'_{n+1}}{right}
\node at (140:\xch) {$y_1$};
\node at (90:\xch) {$\cdots$};
\node at (40:\xch) {$y_n$};
\node[below] at ($(o)!0.7!(a1)$) {$\beta_1$};
\node[below] at ($(o)!0.7!(a4)$) {$\beta_{n+1}$};
\node[right] at ($(o)!0.85!(a2)$) {$\beta_2$};
\node[left] at ($(o)!0.85!(a3)$) {$\beta_n\!$};
}\nonumber\\
:=&\ \Biggl(\treeoo{\cdb o\ocdx[2]{o}{a1}{160}{T_1}{left}
\ocdx[2]{o}{a2}{120}{T_2}{above}
\ocdx[2]{o}{a3}{60}{T_m}{above}
\cdx[2]{o}{a4}{20}
\node at (140:\xch) {$x_1$};
\node at (90:\xch) {$\cdots$};
\node at (40:\xch) {$x_m$};
\node[below] at ($(o)!0.7!(a1)$) {$\alpha_1$};
\node[right] at ($(o)!0.85!(a2)$) {$\alpha_2$};
\node[left] at ($(o)!0.85!(a3)$) {$\alpha_m\!$};
}\diamond \Big(B^{+}_{\alpha_{m+1}}(T_{m+1})\diamond B^{+}_{\beta_1}(T'_1)\Big)\Biggl)\diamond
\treeoo{\cdb o\cdx[2]{o}{a1}{160}
\ocdx[2]{o}{a2}{120}{T'_2}{above}
\ocdx[2]{o}{a3}{60}{T'_n}{above}
\ocdx[2]{o}{a4}{20}{T'_{n+1}}{right}
\node at (140:\xch) {$y_1$};
\node at (90:\xch) {$\cdots$};
\node at (40:\xch) {$y_n$};
\node[below] at ($(o)!0.7!(a4)$) {$\beta_{n+1}$};
\node[right] at ($(o)!0.85!(a2)$) {$\beta_2$};
\node[left] at ($(o)!0.85!(a3)$) {$\beta_n\!$};
}\nonumber\\
:=& \ \Biggl(\treeoo{\cdb o\ocdx[2]{o}{a1}{160}{T_1}{left}
\ocdx[2]{o}{a2}{120}{T_2}{above}
\ocdx[2]{o}{a3}{60}{T_m}{above}
\cdx[2]{o}{a4}{20}
\node at (140:\xch) {$x_1$};
\node at (90:\xch) {$\cdots$};
\node at (40:\xch) {$x_m$};
\node[below] at ($(o)!0.7!(a1)$) {$\alpha_1$};
\node[right] at ($(o)!0.85!(a2)$) {$\alpha_2$};
\node[left] at ($(o)!0.85!(a3)$) {$\alpha_m\!$};
}\diamond \bigg(B^{+}_{\alpha_{m+1}\rightarrow \beta_1}\big(B^{+}_{\alpha_{m+1} \rhd \beta_1}(T_{m+1})\diamond T'_1 \big)
+B^{+}_{\alpha_{m+1} \leftarrow \beta_1} \big( T_{m+1} \diamond B^{+}_{\alpha_{m+1} \lhd \beta_1}(T'_1)\big)\\
&\ + \lambda_{\alpha_{m+1}, \beta_1} B^{+}_{\alpha_{m+1} \cdot \beta_1} \big(T_{m+1} \diamond T'_{1} \big) \bigg) \Biggl) \diamond\treeoo{\cdb o
\cdx[2]{o}{a1}{160}
\ocdx[2]{o}{a2}{120}{T'_2}{above}
\ocdx[2]{o}{a3}{60}{T'_n}{above}
\ocdx[2]{o}{a4}{20}{T'_{n+1}}{right}
\node at (140:\xch) {$y_1$};
\node at (90:\xch) {$\cdots$};
\node at (40:\xch) {$y_n$};
\node[below] at ($(o)!0.7!(a4)$) {$\beta_{n+1}$};
\node[right] at ($(o)!0.85!(a2)$) {$\beta_2$};
\node[left] at ($(o)!0.85!(a3)$) {$\beta_n\!$};
}. \nonumber
\end{align}}}
\noindent Here the first $\diamond$ is defined by Case 3, the second, third and fourth $\diamond$ are defined by induction and the last $\diamond$ is defined by Case 2.
This inductively define the multiplication $\diamond$ on $\calt$.
We then extend $\diamond$  by linearity to $\bfk\calt$. We then have the following result:

\begin{lemma}
Let $(\Omega, \leftarrow, \rightarrow, \lhd, \rhd, \cdot, \lambda)$ be a $\lambda$-$\ets$. Then $(\bfk \calt, \diamond)$ is an associative algebra with identity  $\treeoo{\cda[1.5] o\zhd{o/a}
\node[right] at ($(o)!0.7!(oa)$) {};
}$.
\label{lem:asso}
\end{lemma}

\begin{proof}
By the construction of $\diamond$, $\bfk \calt$ is closed under $\diamond$ and $\treeoo{\cda[1.5] o\zhd{o/a}
\node[right] at ($(o)!0.7!(oa)$) {};
}$ is the identity of $\diamond$.

Now we show the associativity of $\diamond$, i.e.
\begin{align}
(T_1 \diamond T_2) \diamond T_3=T_1 \diamond (T_2 \diamond T_3), \mlabel{eq:asso}
\end{align}
 for all $T_1,T_2,T_3 \in \calt$
We prove Eq.~(\ref{eq:asso}) by induction on the sum of depths $p:=\dep(T_1)+\dep(T_2)+\dep(T_3)$.
If $p=3$, then $\dep(T_1)=\dep(T_2)=\dep(T_3)=1$ and $T_1, T_2, T_3$ are of the form
\begin{align*}
T_1 &=\treeoo{\cdb o\cdx[1.5]{o}{a1}{160}
\cdx[1.5]{o}{a2}{120}
\cdx[1.5]{o}{a3}{60}
\cdx[1.5]{o}{a4}{20}
\node at (140:1.2*\xch) {$x_1$};
\node at (90:1.2*\xch) {$\cdots$};
\node at (40:1.2*\xch) {$x_l$};
},&T_2&=\treeoo{\cdb o\cdx[1.5]{o}{a1}{160}
\cdx[1.5]{o}{a2}{120}
\cdx[1.5]{o}{a3}{60}
\cdx[1.5]{o}{a4}{20}
\node at (140:1.2*\xch) {$y_1$};
\node at (90:1.2*\xch) {$\cdots$};
\node at (40:1.2*\xch) {$y_m$};
},&& \text{ and } & T_3&=\treeoo{\cdb o\cdx[1.5]{o}{a1}{160}
\cdx[1.5]{o}{a2}{120}
\cdx[1.5]{o}{a3}{60}
\cdx[1.5]{o}{a4}{20}
\node at (140:1.2*\xch) {$z_1$};
\node at (90:1.2*\xch) {$\cdots$};
\node at (40:1.2*\xch) {$z_n$};
}&&\text{ with }\,l,m,n\geq 0.
\end{align*}
Then $(T_1 \diamond T_2) \diamond T_3=T_1 \diamond (T_2 \diamond T_3)$ by a direct calculation.

 For the induction step $p \geq 4$, we use induction on the sum of branches $q:=\bra(T_1)+\bra(T_2)+\bra(T_3)$. If $q=3$ and one of $T_1,T_2, T_3$ has depth 1, then this tree must be of the form $\treeoo{\cda[1.5] o\zhd{o/a}
\node[right] at ($(o)!0.7!(oa)$) {};
}$ and the associativity of $\diamond$ follows directly. Assume
\begin{align*}
T_1=B^{+}_{\alpha}(T'_1),\, T_2=B^{+}_{\beta}(T'_2), \, T_3=B^{+}_{\gamma}(T'_3) \,\, \text{ for some $\alpha,\beta, \gamma \in \Omega$ and $T'_1,T'_2,T'_3 \in \calt$},
\end{align*}
then
\begin{align*}
&\ (T_1 \diamond T_2) \diamond T_3=\big(B^{+}_{\alpha}(T'_1) \diamond B^{+}_{\beta}(T'_2) \big) \diamond B^{+}_{\gamma}(T'_3)\\
=&\ B^{+}_{\alpha \rightarrow \beta}(B^{+}_{\alpha \rhd \beta}(T'_1) \diamond T'_2) \diamond B^{+}_{\gamma}(T'_3) +B^{+}_{\alpha \leftarrow \beta}(T'_1 \diamond B^{+}_{\alpha \lhd \beta}(T'_2)) \diamond B^{+}_{\gamma} (T'_3)+ \lambda_{\alpha, \beta} B^{+}_{\alpha \cdot \beta}(T'_1 \diamond T'_2) \diamond B^{+}_{\gamma}(T'_3) \\
=&\ B^{+}_{(\alpha \rightarrow \beta) \rightarrow \gamma} \big(B^{+}_{(\alpha \rightarrow \beta) \rhd \gamma}(B^{+}_{\alpha \rhd \beta}(T'_1) \diamond T'_2) \diamond T'_3 \big)+B^{+}_{(\alpha \rightarrow \beta) \leftarrow \gamma} \big((B^{+}_{\alpha \rhd \beta}(T'_1) \diamond T'_2) \diamond B^{+}_{(\alpha \rightarrow \beta)\lhd \gamma}(T'_3) \big)\\
&\ +\lambda_{\alpha \rightarrow \beta, \gamma} B^{+}_{(\alpha \rightarrow \beta) \cdot \gamma} \big((B^{+}_{\alpha \rhd \beta}(T'_1) \diamond T'_2) \diamond T'_3 \big)+B^{+}_{(\alpha \leftarrow \beta) \rightarrow \gamma} \big(B^{+}_{(\alpha \leftarrow \beta) \rhd \gamma}(T'_1 \diamond B^{+}_{\alpha \lhd \beta}(T'_2)) \diamond T'_3 \big)\\
&\ +B^{+}_{(\alpha \leftarrow \beta) \leftarrow \gamma} \big((T'_1 \diamond B^{+}_{\alpha \lhd \beta}(T'_2)) \diamond B^{+}_{(\alpha \leftarrow \beta) \lhd \gamma}(T'_3) \big)+ \lambda_{\alpha \leftarrow \beta, \gamma} B^{+}_{(\alpha \leftarrow \beta) \cdot \gamma} \big((T'_1 \diamond B^{+}_{\alpha \lhd \beta}(T'_2)) \diamond T'_3 \big) \\
&\ +\lambda_{\alpha, \beta} B^{+}_{(\alpha \cdot \beta) \rightarrow \gamma} \big(B^{+}_{(\alpha \cdot \beta) \rhd \gamma}(T'_1 \diamond T'_2) \diamond T'_3 \big)+ \lambda_{\alpha, \beta} B^{+}_{(\alpha \cdot \beta) \leftarrow \gamma} \big((T'_1 \diamond T'_2) \diamond B^{+}_{(\alpha \cdot \beta) \lhd \gamma}(T'_3) \big)\\
&\ +\lambda_{\alpha, \beta} \lambda_{\alpha \cdot \beta, \gamma}B^{+}_{(\alpha \cdot \beta) \cdot \gamma} \big((T'_1 \diamond T'_2) \diamond T'_3 \big)\\
=&\ B^{+}_{(\alpha \rightarrow \beta) \rightarrow \gamma} \big(B^{+}_{(\alpha \rightarrow \beta) \rhd \gamma}(B^{+}_{\alpha \rhd \beta}(T'_1) \diamond T'_2) \diamond T'_3 \big)+B^{+}_{(\alpha \rightarrow \beta) \leftarrow \gamma} \big((B^{+}_{\alpha \rhd \beta}(T'_1) \diamond T'_2) \diamond B^{+}_{(\alpha \rightarrow \beta)\lhd \gamma}(T'_3) \big)\\
&\ +\lambda_{\alpha \rightarrow \beta, \gamma} B^{+}_{(\alpha \rightarrow \beta) \cdot \gamma} \big((B^{+}_{\alpha \rhd \beta}(T'_1) \diamond T'_2) \diamond T'_3 \big) +B^{+}_{(\alpha \leftarrow \beta) \rightarrow \gamma} \big(B^{+}_{(\alpha \leftarrow \beta) \rhd \gamma}(T'_1 \diamond B^{+}_{\alpha \lhd \beta}(T'_2)) \diamond T'_3 \big)\\
&\ +B^{+}_{(\alpha \leftarrow \beta) \leftarrow \gamma} \big(T'_1 \diamond (B^{+}_{\alpha \lhd \beta}(T'_2) \diamond B^{+}_{(\alpha \leftarrow \beta) \lhd \gamma}(T'_3)) \big)+ \lambda_{\alpha \leftarrow \beta, \gamma} B^{+}_{(\alpha \leftarrow \beta) \cdot \gamma} \big((T'_1 \diamond B^{+}_{\alpha \lhd \beta}(T'_2)) \diamond T'_3 \big) \\
&\ +\lambda_{\alpha, \beta} B^{+}_{(\alpha \cdot \beta) \rightarrow \gamma} \big(B^{+}_{(\alpha \cdot \beta) \rhd \gamma}(T'_1 \diamond T'_2) \diamond T'_3 \big)+ \lambda_{\alpha, \beta} B^{+}_{(\alpha \cdot \beta) \leftarrow \gamma} \big((T'_1 \diamond T'_2) \diamond B^{+}_{(\alpha \cdot \beta) \lhd \gamma}(T'_3) \big)\\
&\ +\lambda_{\alpha, \beta} \lambda_{\alpha \cdot \beta, \gamma}B^{+}_{(\alpha \cdot \beta) \cdot \gamma} \big((T'_1 \diamond T'_2) \diamond T'_3 \big)\quad \quad {\text{(by the induction hypothesis)}}\\
=&\  B^{+}_{(\alpha \rightarrow \beta) \rightarrow \gamma} \big(B^{+}_{(\alpha \rightarrow \beta)\rhd \gamma} (B^{+}_{\alpha \rhd \gamma}(T'_1) \diamond T'_2) \diamond T'_3 \big)+B^{+}_{(\alpha \rightarrow \beta) \leftarrow \gamma} \big((B^{+}_{\alpha \rhd \beta} (T'_1) \diamond T'_2) \diamond B^{+}_{(\alpha \rightarrow \beta) \lhd \gamma}(T'_3) \big)\\
&\ +\lambda_{\alpha \rightarrow \beta, \gamma} B^{+}_{(\alpha \rightarrow \beta) \cdot \gamma} \big((B^{+}_{\alpha \rhd \beta}(T'_1) \diamond T'_2) \diamond T'_3 \big) +B^{+}_{(\alpha \leftarrow \beta) \rightarrow \gamma} \big(B^{+}_{(\alpha \leftarrow \beta) \rhd \gamma}(T'_1 \diamond B^{+}_{\alpha \lhd \beta}(T'_2)) \diamond T'_3 \big)\\
&\ +B^{+}_{(\alpha \leftarrow \beta) \leftarrow \gamma} \big(T'_1 \diamond B^{+}_{(\alpha \lhd \beta) \rightarrow ((\alpha \leftarrow \beta) \lhd \gamma)}(  B^{+}_{(\alpha \lhd \beta) \rhd ((\alpha \leftarrow \beta) \lhd \gamma)} (T'_2) \diamond T'_3) \big)\\
&\ +B^{+}_{(\alpha \leftarrow \beta) \leftarrow \gamma}(T'_1 \diamond B^{+}_{(\alpha \lhd \beta) \leftarrow ((\alpha \leftarrow \beta)\lhd \gamma)} (T'_2 \diamond B^{+}_{(\alpha \lhd \beta) \lhd ((\alpha \leftarrow \beta) \lhd \gamma)}(T'_3)))\\
&\ +\lambda_{\alpha \lhd \beta, (\alpha \leftarrow \beta) \lhd \gamma} B^{+}_{(\alpha \leftarrow \beta) \leftarrow \gamma} \big(T'_1 \diamond B^{+}_{(\alpha \lhd \beta) \cdot ((\alpha \leftarrow \beta) \lhd \gamma)} (T'_2 \diamond T'_3) \big)+ \lambda_{\alpha \leftarrow \beta, \gamma} B^{+}_{(\alpha \leftarrow \beta) \cdot \gamma} \big((T'_1 \diamond B^{+}_{\alpha \lhd \beta}(T'_2)) \diamond T'_3 \big) \\
&\ +\lambda_{\alpha, \beta} B^{+}_{(\alpha \cdot \beta) \rightarrow \gamma} \big(B^{+}_{(\alpha \cdot \beta) \rhd \gamma}(T'_1 \diamond T'_2) \diamond T'_3 \big)+\lambda_{\alpha, \beta} B^{+}_{(\alpha \cdot \beta) \leftarrow \gamma} \big((T'_1 \diamond T'_2) \diamond B^{+}_{(\alpha \cdot \beta) \lhd \gamma}(T'_3) \big)\\
&\ +\lambda_{\alpha, \beta} \lambda_{\alpha \cdot \beta, \gamma}B^{+}_{(\alpha \cdot \beta) \cdot \gamma} \big((T'_1 \diamond T'_2) \diamond T'_3 \big),
\end{align*}
and
\begin{align*}
&\ T_1 \diamond (T_2 \diamond T_3)=B^{+}_{\alpha}(T'_1) \diamond \big(B^{+}_{\beta}(T'_2) \diamond B^{+}_{\gamma}(T'_3) \big)\\
=&\ B^{+}_{\alpha}(T'_1) \diamond B^{+}_{\beta \rightarrow \gamma}(B^{+}_{\beta \rhd \gamma}(T'_2) \diamond T'_3)+ B^{+}_{\alpha}(T'_1) \diamond B^{+}_{\beta \leftarrow \gamma}(T'_2 \diamond B^{+}_{\beta \lhd \gamma} (T'_3))+ \lambda_{\beta, \gamma} B^{+}_{\alpha}(T'_1) \diamond B^{+}_{\beta \cdot \gamma} (T'_2 \diamond T'_3)\\
=&\ B^{+}_{\alpha \rightarrow (\beta \rightarrow \gamma)} \big(B^{+}_{\alpha \rhd (\beta \rightarrow \gamma)}(T'_1) \diamond (B^{+}_{\beta \rhd \gamma}(T'_2) \diamond T'_3) \big)+B^{+}_{\alpha \leftarrow (\beta \rightarrow \gamma)} \big(T'_1 \diamond B^{+}_{\alpha \lhd (\beta \rightarrow \gamma)}(B^{+}_{\beta \rhd \gamma}(T'_2) \diamond T'_3) \big)\\
&\ +\lambda_{\alpha, \beta \rightarrow \gamma} B^{+}_{\alpha \cdot (\beta \rightarrow \gamma)} \big(T'_1 \diamond (B^{+}_{\beta \rhd \gamma}(T'_2) \diamond T'_3) \big)+B^{+}_{\alpha \rightarrow (\beta \leftarrow \gamma)}\big(B^{+}_{\alpha \rhd (\beta \leftarrow \gamma)} (T'_1) \diamond (T'_2 \diamond B^{+}_{\beta \lhd \gamma}(T'_3)) \big)\\
&\ +B^{+}_{\alpha \leftarrow (\beta \leftarrow \gamma)} \big(T'_1 \diamond B^{+}_{\alpha \lhd (\beta \leftarrow \gamma)}(T'_2 \diamond B^{+}_{\beta \lhd \gamma}(T'_3)) \big)+ \lambda_{\alpha, \beta \leftarrow \gamma} B^{+}_{\alpha \cdot (\beta \leftarrow \gamma)} \big(T'_1 \diamond (T'_2 \diamond B^{+}_{\beta \lhd \gamma} (T'_3)) \big)\\
&\ +\lambda_{\beta, \gamma} B^{+}_{\alpha \rightarrow (\beta \cdot \gamma)} \big(B^{+}_{\alpha \rhd(\beta \cdot \gamma)}(T'_1) \diamond (T'_2 \diamond T'_3) \big)+ \lambda_{\beta, \gamma} B^{+}_{\alpha \leftarrow (\beta \cdot \gamma)} \big(T'_1 \diamond B^{+}_{\alpha \leftarrow (\beta \cdot \gamma)}(T'_2 \diamond T'_3) \big)\\
&\ +\lambda_{\beta, \gamma} \lambda_{\alpha, \beta \cdot \gamma} B^{+}_{\alpha \cdot (\beta \cdot \gamma)} \big(T'_1 \diamond (T'_2 \diamond T'_3) \big)\\
=&\ B^{+}_{\alpha \rightarrow (\beta \rightarrow \gamma)} \big((B^{+}_{\alpha \rhd (\beta \rightarrow \gamma)}(T'_1) \diamond B^{+}_{\beta \rhd \gamma}(T'_2)) \diamond T'_3 \big)+B^{+}_{\alpha \leftarrow (\beta \rightarrow \gamma)} \big(T'_1 \diamond B^{+}_{\alpha \lhd (\beta \rightarrow \gamma)}(B^{+}_{\beta \rhd \gamma}(T'_2) \diamond T'_3) \big)\\
&\ +\lambda_{\alpha, \beta \rightarrow \gamma} B^{+}_{\alpha \cdot (\beta \rightarrow \gamma)} \big(T'_1 \diamond (B^{+}_{\beta \rhd \gamma}(T'_2) \diamond T'_3) \big)+B^{+}_{\alpha \rightarrow (\beta \leftarrow \gamma)}\big(B^{+}_{\alpha \rhd (\beta \leftarrow \gamma)} (T'_1) \diamond (T'_2 \diamond B^{+}_{\beta \lhd \gamma}(T'_3)) \big)\\
&\ +B^{+}_{\alpha \leftarrow (\beta \leftarrow \gamma)} \big(T'_1 \diamond B^{+}_{\alpha \lhd (\beta \leftarrow \gamma)}(T'_2 \diamond B^{+}_{\beta \lhd \gamma}(T'_3)) \big)+ \lambda_{\alpha, \beta \leftarrow \gamma} B^{+}_{\alpha \cdot (\beta \leftarrow \gamma)} \big(T'_1 \diamond (T'_2 \diamond B^{+}_{\beta \lhd \gamma} (T'_3)) \big)\\
&\ +\lambda_{\beta, \gamma} B^{+}_{\alpha \rightarrow (\beta \cdot \gamma)} \big(B^{+}_{\alpha \rhd(\beta \cdot \gamma)}(T'_1) \diamond (T'_2 \diamond T'_3) \big)+ \lambda_{\beta, \gamma} B^{+}_{\alpha \leftarrow (\beta \cdot \gamma)} \big(T'_1 \diamond B^{+}_{\alpha \leftarrow (\beta \cdot \gamma)}(T'_2 \diamond T'_3) \big)\\
&\ +\lambda_{\beta, \gamma} \lambda_{\alpha, \beta \cdot \gamma} B^{+}_{\alpha \cdot (\beta \cdot \gamma)} \big(T'_1 \diamond (T'_2 \diamond T'_3) \big)\quad \quad {\text{(by the induction hypothesis)}}\\
=&\ B^{+}_{\alpha \rightarrow (\beta \rightarrow \gamma)} \big( B^{+}_{(\alpha \rhd (\beta \rightarrow \gamma)) \rightarrow (\beta \rhd \gamma)}(B^{+}_{(\alpha \rhd (\beta \rightarrow \gamma)) \rhd (\beta \rhd \gamma)}(T'_1) \diamond T'_2) \diamond T'_3 \big)\\
&\ +B^{+}_{\alpha \rightarrow (\beta \rightarrow \gamma)} \big(B^{+}_{\alpha \rhd (\beta \rightarrow \gamma) \leftarrow (\beta \rhd \gamma)}(T'_1 \diamond B^{+}_{(\alpha \rhd (\beta \rightarrow \gamma) \lhd (\beta \rhd \gamma))}(T'_2)) \diamond T'_3  \big)\\
&\ + \lambda_{\alpha \rhd (\beta \rightarrow \gamma),\beta \rhd \gamma} B^{+}_{\alpha \rightarrow (\beta \rightarrow \gamma)} (B^{+}_{(\alpha \rhd (\beta \rightarrow \gamma)) \cdot (\beta \rhd \gamma)} (T'_1 \diamond T'_2) \diamond T'_3)+B^{+}_{\alpha \leftarrow (\beta \rightarrow \gamma)} \big(T'_1 \diamond B^{+}_{\alpha \lhd (\beta \rightarrow \gamma)}(B^{+}_{\beta \rhd \gamma}(T'_2) \diamond T'_3) \big)\\
&\ +\lambda_{\alpha, \beta \rightarrow \gamma} B^{+}_{\alpha \cdot (\beta \rightarrow \gamma)} \big(T'_1 \diamond (B^{+}_{\beta \rhd \gamma}(T'_2) \diamond T'_3) \big)+B^{+}_{\alpha \rightarrow (\beta \leftarrow \gamma)}\big(B^{+}_{\alpha \rhd (\beta \leftarrow \gamma)} (T'_1) \diamond (T'_2 \diamond B^{+}_{\beta \lhd \gamma}(T'_3)) \big)\\
&\ +B^{+}_{\alpha \leftarrow (\beta \leftarrow \gamma)} \big(T'_1 \diamond B^{+}_{\alpha \lhd (\beta \leftarrow \gamma)}(T'_2 \diamond B^{+}_{\beta \lhd \gamma}(T'_3)) \big)+ \lambda_{\alpha, \beta \leftarrow \gamma} B^{+}_{\alpha \cdot (\beta \leftarrow \gamma)} \big(T'_1 \diamond (T'_2 \diamond B^{+}_{\beta \lhd \gamma} (T'_3)) \big)\\
&\ +\lambda_{\beta, \gamma} B^{+}_{\alpha \rightarrow (\beta \cdot \gamma)} \big(B^{+}_{\alpha \rhd(\beta \cdot \gamma)}(T'_1) \diamond (T'_2 \diamond T'_3) \big)+ \lambda_{\beta, \gamma} B^{+}_{\alpha \leftarrow (\beta \cdot \gamma)} \big(T'_1 \diamond B^{+}_{\alpha \leftarrow (\beta \cdot \gamma)}(T'_2 \diamond T'_3) \big)\\
&\ +\lambda_{\beta, \gamma} \lambda_{\alpha, \beta \cdot \gamma} B^{+}_{\alpha \cdot (\beta \cdot \gamma)} \big(T'_1 \diamond (T'_2 \diamond T'_3) \big).
\end{align*}
By the induction hypothesis and $(\Omega, \leftarrow, \rightarrow, \lhd, \rhd,\cdot,  \lambda)$ being a $\lambda$-$\ets$, we get
\begin{align*}
(T_1 \diamond T_2) \diamond T_3= T_1 \diamond (T_2 \diamond T_3).
\end{align*}

If $q >3$, then at least one of $T_1,T_2, T_3$ have branches greater than or equal to 2. If $\bra(T_1) \geq 2$, then there exist
$T'_1, T''_1$ of the form
\begin{align*}
T'_1&=\treeoo{\cdb o\ocdx[2]{o}{a1}{160}{\overline{T}'_1}{left}
\ocdx[2]{o}{a2}{120}{\overline{T}'_2}{above}
\ocdx[2]{o}{a3}{60}{\overline{T}'_m}{above}
\cdx[2]{o}{a4}{20}
\node at (140:\xch) {$x_1$};
\node at (90:\xch) {$\cdots$};
\node at (40:\xch) {$x_m$};
\node[below] at ($(o)!0.7!(a1)$) {$\alpha_1$};
\node[right] at ($(o)!0.85!(a2)$) {$\alpha_2$};
\node[left] at ($(o)!0.85!(a3)$) {$\alpha_m\!$};
}&& \text{ and }& T''_1&=\treeoo{\cdb o
\ocdx[2]{o}{a1}{160}{\overline{T}''_1}{left}
\ocdx[2]{o}{a2}{120}{\overline{T}''_2}{above}
\ocdx[2]{o}{a3}{60}{\overline{T}''_n}{above}
\ocdx[2]{o}{a4}{20}{\overline{T}''_{n+1}}{right}
\node at (140:\xch) {$y_1$};
\node at (90:\xch) {$\cdots$};
\node at (40:\xch) {$y_n$};
\node[below] at ($(o)!0.7!(a4)$) {$\beta_{n+1}$};
\node[right] at ($(o)!0.85!(a2)$) {$\beta_2$};
\node[left] at ($(o)!0.85!(a3)$) {$\beta_n\!$};
}
\end{align*}
such that $T_1=T'_1 \diamond T''_1$. Hence
\begin{align*}
&\ (T_1 \diamond T_2) \diamond T_3=((T'_1 \diamond T''_1) \diamond T_2) \diamond T_3\\
=&\ (T'_1 \diamond (T''_1 \diamond T_2)) \diamond T_3 &&\text{(by the induction hypothesis)}\\
=&\ T'_1 \diamond ((T''_1 \diamond T_2) \diamond T_3) && \text{(by the form of $T'_1$ and the definition of $\diamond$)}\\
=&\ T'_1 \diamond (T''_1 \diamond (T_2 \diamond T_3)) &&\text{(by the induction hypothesis)}\\
=&\ T'_1 \diamond T''_1 \diamond (T_2 \diamond T_3) && \text{(by the form of $T'_1$ and the definition of $\diamond$)}\\
=&\ T_1 \diamond (T_2 \diamond T_3).
\end{align*}
If $\bra(T_2) \geq 2$ or $\bra(T_3) \geq 2$, the associativity can be proved similarly.
\end{proof}

Let $i:X \rightarrow \bfk \calt, x \mapsto \stree x$ be the natural inclusion. Then

\begin{theorem} \label{propRB}
Let $\Omega$ be a set with five products $\leftarrow, \rightarrow, \lhd, \rhd, \cdot$ and $\lambda$ a family of elements in ${\bfk}$ indexed by $\Omega^2$. Then the following conditions are equivalent:
\begin{enumerate}
\item \label{it1} $(\bfk \calt, \diamond, (B^{+}_{\omega})_{\omega \in \Omega})$ together with the map $i$ is the free $\Omega$-Rota-Baxter algebra generated by $X$.
\item \label{it2} $(\bfk \calt, \diamond, (B^{+}_{\omega})_{\omega \in \Omega})$ is an $\Omega$-Rota-Baxter algebra.
\item \label{it3} $(\Omega, \leftarrow, \rightarrow, \lhd, \rhd,\cdot,\lambda)$ is a $\lambda$-$\ets$.
\end{enumerate}
\end{theorem}

\begin{proof}
\ref{it1} $\Longrightarrow$ \ref{it2} It is obvious.

\ref{it2} $\Longrightarrow$ \ref{it3} For $\alpha, \beta, \gamma \in \Omega$ and $\stree x, \stree y, \stree z \in \calt$, we have
\begin{align*}
&\ \big(B^{+}_{\alpha}(\stree x) \diamond B^{+}_{\beta}(\stree y) \big) \diamond B^{+}_{\gamma}(\stree z)\\
=&\ B^{+}_{\alpha \rightarrow \beta}(B^{+}_{\alpha \rhd \beta}(\stree x) \diamond \stree y) \diamond B^{+}_{\gamma}(\stree z) +B^{+}_{\alpha \leftarrow \beta}(\stree x \diamond B^{+}_{\alpha \lhd \beta}(\stree y)) \diamond B^{+}_{\gamma} (\stree z)\\
&\ + \lambda_{\alpha, \beta} B^{+}_{\alpha \cdot \beta}(\stree x \diamond \stree y) \diamond B^{+}_{\gamma}(\stree z) \\
=&\ B^{+}_{(\alpha \rightarrow \beta) \rightarrow \gamma} \big(B^{+}_{(\alpha \rightarrow \beta) \rhd \gamma}(B^{+}_{\alpha \rhd \beta}(\stree x) \diamond \stree y) \diamond \stree z \big)\\
&\ +B^{+}_{(\alpha \rightarrow \beta) \leftarrow \gamma} \big((B^{+}_{\alpha \rhd \beta}(\stree x) \diamond \stree y) \diamond B^{+}_{(\alpha \rightarrow \beta)\lhd \gamma}( \stree z) \big)\\
&\ +\lambda_{\alpha \rightarrow \beta, \gamma} B^{+}_{(\alpha \rightarrow \beta) \cdot \gamma} \big((B^{+}_{\alpha \rhd \beta}(\stree x) \diamond \stree y) \diamond \stree z \big)\\
&\ +B^{+}_{(\alpha \leftarrow \beta) \rightarrow \gamma} \big(B^{+}_{(\alpha \leftarrow \beta) \rhd \gamma}(\stree x \diamond B^{+}_{\alpha \lhd \beta}(\stree y)) \diamond \stree z \big)\\
&\ +B^{+}_{(\alpha \leftarrow \beta) \leftarrow \gamma} \big((\stree x \diamond B^{+}_{\alpha \lhd \beta}(\stree y)) \diamond B^{+}_{(\alpha \leftarrow \beta) \lhd \gamma}(\stree z) \big)\\
&\ + \lambda_{\alpha \leftarrow \beta, \gamma} B^{+}_{(\alpha \leftarrow \beta) \cdot \gamma} \big((\stree x \diamond B^{+}_{\alpha \lhd \beta}(\stree y)) \diamond \stree z \big) \\
&\ +\lambda_{\alpha, \beta} B^{+}_{(\alpha \cdot \beta) \rightarrow \gamma} \big(B^{+}_{(\alpha \cdot \beta) \rhd \gamma}(\stree x \diamond \stree y) \diamond \stree z \big)\\
&\ + \lambda_{\alpha, \beta} B^{+}_{(\alpha \cdot \beta) \leftarrow \gamma} \big((\stree x \diamond \stree y) \diamond B^{+}_{(\alpha \cdot \beta) \lhd \gamma}(\stree z) \big)\\
&\ +\lambda_{\alpha, \beta} \lambda_{\alpha \cdot \beta, \gamma}B^{+}_{(\alpha \cdot \beta) \cdot \gamma} \big((\stree x \diamond \stree y) \diamond \stree z \big)\\
=&\  B^{+}_{(\alpha \rightarrow \beta) \rightarrow \gamma} \big(B^{+}_{(\alpha \rightarrow \beta)\rhd \gamma} (B^{+}_{\alpha \rhd \gamma}(\stree x) \diamond \stree y) \diamond \stree z \big)\\
&\ +B^{+}_{(\alpha \rightarrow \beta) \leftarrow \gamma} \big((B^{+}_{\alpha \rhd \beta} (\stree x) \diamond \stree y) \diamond B^{+}_{(\alpha \rightarrow \beta) \lhd \gamma}(\stree z) \big)\\
&\ +\lambda_{\alpha \rightarrow \beta, \gamma} B^{+}_{(\alpha \rightarrow \beta) \cdot \gamma} \big((B^{+}_{\alpha \rhd \beta}(\stree x) \diamond \stree y) \diamond \stree z \big)\\
 &\ +B^{+}_{(\alpha \leftarrow \beta) \rightarrow \gamma} \big(B^{+}_{(\alpha \leftarrow \beta) \rhd \gamma}(\stree x \diamond B^{+}_{\alpha \lhd \beta}(\stree y)) \diamond \stree z \big)\\
&\ +B^{+}_{(\alpha \leftarrow \beta) \leftarrow \gamma} \big(\stree x \diamond B^{+}_{(\alpha \lhd \beta) \rightarrow ((\alpha \leftarrow \beta) \lhd \gamma)}(  B^{+}_{(\alpha \lhd \beta) \rhd ((\alpha \leftarrow \beta) \lhd \gamma)} (\stree y) \diamond \stree z) \big)\\
&\ +B^{+}_{(\alpha \leftarrow \beta) \leftarrow \gamma}(\stree x \diamond B^{+}_{(\alpha \lhd \beta) \leftarrow ((\alpha \leftarrow \beta)\lhd \gamma)} (\stree y \diamond B^{+}_{(\alpha \lhd \beta) \lhd ((\alpha \leftarrow \beta) \lhd \gamma)}(\stree z)))\\
&\ +\lambda_{\alpha \lhd \beta, (\alpha \leftarrow \beta) \lhd \gamma} B^{+}_{(\alpha \leftarrow \beta) \leftarrow \gamma} \big(\stree x \diamond B^{+}_{(\alpha \lhd \beta) \cdot ((\alpha \leftarrow \beta) \lhd \gamma)} (\stree y \diamond \stree z) \big)\\
&\ + \lambda_{\alpha \leftarrow \beta, \gamma} B^{+}_{(\alpha \leftarrow \beta) \cdot \gamma} \big((\stree x \diamond B^{+}_{\alpha \lhd \beta}(\stree y)) \diamond \stree z \big) \\
&\ +\lambda_{\alpha, \beta} B^{+}_{(\alpha \cdot \beta) \rightarrow \gamma} \big(B^{+}_{(\alpha \cdot \beta) \rhd \gamma}(\stree x \diamond \stree y) \diamond \stree z \big)\\
&\ +\lambda_{\alpha, \beta} B^{+}_{(\alpha \cdot \beta) \leftarrow \gamma} \big((\stree x \diamond \stree y) \diamond B^{+}_{(\alpha \cdot \beta) \lhd \gamma}(\stree z) \big)\\
&\ +\lambda_{\alpha, \beta} \lambda_{\alpha \cdot \beta, \gamma}B^{+}_{(\alpha \cdot \beta) \cdot \gamma} \big((\stree x \diamond \stree y) \diamond \stree z \big)\\
=&\ \XX[scale=1.6]{\xxr{-4}4\xxr{-7.5}{7.5}
\node at (0,-0.3) {$\bullet$};
\node at (-1.2,0.1) {\tiny $(\alpha \rightarrow \beta) \rhd \gamma$};
\node at (-1.1,0.45) {\tiny $\alpha \rhd \beta$};
\node at (-1.2,-0.15) {\tiny $(\alpha \rightarrow \beta) \rightarrow \gamma$};
\xxhu00{z} \xxhu[0.1]{-4}4{y} \xxhu[0.1]{-7.5}{7.5}{x}
}+
\XX[scale=1.6]{\xxr{-6}6\xxl66
\node at (0,-0.3) {$\bullet$};
\node at (-0.9,0.2) {\tiny $\alpha \rhd \beta$};
\node at (1.3,0.2) {\tiny $(\alpha \rightarrow \beta) \lhd \gamma$};
\node at (-1.2,-0.15) {\tiny $(\alpha \rightarrow \beta) \leftarrow \gamma$};
\xxhu00{y} \xxhu66{z} \xxhu{-6}6{x}
}+
\lambda_{\alpha \rightarrow \beta, \gamma} \XX[scale=1.6]{\xxr{-7.5}{7.5} \draw (0,0)--(0,1);
\node at (-1,-0.15) {\tiny $(\alpha \rightarrow \beta) \cdot \gamma$};
\node at (-1,0.4) {\tiny $\alpha \rhd \beta$};
\node at (-0.15,0.4){$y$};
\node at (0.15,0.4){$z$};
\xxhu[0.1]{-7.5}{7.5}{x}
\node at (0,-0.3) {$\bullet$};
}
+
\XX[scale=1.6]{\xxr{-5}5\xxl{-2}8
\node at (0,-0.3) {$\bullet$};
\node at (-1,0.1) {\tiny $(\alpha \leftarrow \beta) \rhd \gamma$};
\node at (0,0.55) {\tiny $\alpha \rhd \beta$};
\node at (-1.2,-0.15) {\tiny $(\alpha \leftarrow \beta) \rightarrow \gamma$};
\xxhu00z
\xxhu[0.1]{-5}5{x} \xxhu[0.1]{-2}8{y}
}\\
&\
+
\XX[scale=1.6]{\xxl55\xxr28
\node at (0,-0.3) {$\bullet$};
\node at (2,0.15) {\tiny $(\alpha \lhd \beta) \rightarrow ((\alpha \leftarrow \beta) \lhd \gamma)$};
\node at (0.22,0.5) {\tiny $(\alpha \lhd \beta) \rhd ((\alpha \leftarrow \beta) \lhd \gamma)$};
\node at (-1.2,-0.15) {\tiny $(\alpha \leftarrow \beta) \leftarrow \gamma$};
\xxhu[0.1]00{x\,}
\xxhu[0.1]55{\,z} \xxhu[0.1]28{y}
} +
\XX[scale=1.6]{\xxl44\xxl{7.5}{7.5}
\node at (0,-0.3) {$\bullet$};
\node at (2.0,0.1) {\tiny $(\alpha \lhd \beta) \leftarrow ((\alpha \leftarrow \beta) \lhd \gamma)$};
\node at (2.2,0.4) {\tiny $(\alpha \lhd \beta) \lhd ((\alpha \leftarrow \beta) \lhd \gamma)$};
\node at (1.2,-0.15) {\tiny $(\alpha \leftarrow \beta) \leftarrow \gamma$};
\xxhu00{x} \xxhu[0.1]44{y} \xxhu[0.1]{7.5}{7.5}{z}
}+
\lambda_{\alpha \lhd \beta, (\alpha \leftarrow \beta) \lhd \gamma}\XX[scale=1.6]{\xxl{5}{5} \draw (0.5,0.5)--(0.5,1);
\node at (0,-0.3) {$\bullet$};
\node at (0.35,0.85) {$y$}; \node at (0.65,0.85){$z$};
\node at (1.9,0.2) {\tiny $(\alpha \lhd \beta) \cdot ((\alpha \leftarrow \beta) \lhd \gamma)$};
\node at (1,-0.15) {\tiny $(\alpha \leftarrow \beta) \leftarrow \gamma$};
\xxhu00x
}\\
&\ +\lambda_{\alpha \leftarrow \beta, \gamma} \XX[scale=1.6]{\xxlr0{7.5} \draw(0,0)--(0,0.75);
\xxh0023{x\ \,}{0.3} \xxhu[0.12]0{7.5}{z}
\node at (0.16,0.35) {\tiny $y$};
\node at (0,-0.3) {$\bullet$};
\node at (0.25,0.62) {\tiny $\alpha \lhd \beta$};
\node at (1,-0.15){\tiny $(\alpha \leftarrow \beta) \cdot \gamma$};
}+ \lambda_{\alpha, \beta} \XX[scale=1.6]{\xxr{-5}{5} \draw(-0.5,0.5)--(-0.5,1);
\node at (0,-0.3) {$\bullet$};
\xxhu00z
\node at (-0.72,0.85){$x$};
\node at (-0.38,0.85){$y$};
\node at (-1,0.2){\tiny $(\alpha \cdot \beta) \rhd \gamma$};
\node at (-1,-0.15){\tiny $(\alpha \cdot \beta) \rightarrow \gamma$};
}+ \lambda_{\alpha, \beta} \XX[scale=1.6]{\xxl{7}{7} \draw(0,0)--(0,1);
\node at (0,-0.3) {$\bullet$};
\node at (-0.2,0.4) {$x$};
\node at (0.2,0.4){$y$};
\node at (0.75,0.9){$z$};
\node at (1.1,0.3){\tiny $(\alpha \cdot \beta) \lhd \gamma$};
\node at (1,-0.15) {\tiny $(\alpha \cdot \beta) \leftarrow \gamma$};
}+ \lambda_{\alpha, \beta} \lambda_{\alpha \cdot \beta, \gamma} \XX[scale=1.6]{\xx00{1.6}\xx00{2.4}
\xxh001{1.6}{\ \ \,z}{0.6}
\xxh00{1.6}{2.4}{y}{0.5}
\xxh00{2.4}3{x\ \ }{0.6}
\node at (0,-0.3){$\bullet$};
\node at (1,-0.15){$\tiny (\alpha \cdot \beta) \cdot \gamma$};
},
\end{align*}
and
\begin{align*}
&\ B^{+}_{\alpha}(\stree x) \diamond \big(B^{+}_{\beta}(\stree y) \diamond B^{+}_{\gamma}(\stree z) \big)\\
=&\ B^{+}_{\alpha}(\stree x) \diamond B^{+}_{\beta \rightarrow \gamma}(B^{+}_{\beta \rhd \gamma}(\stree y) \diamond \stree z)+ B^{+}_{\alpha}(\stree x) \diamond B^{+}_{\beta \leftarrow \gamma}(\stree y \diamond B^{+}_{\beta \lhd \gamma} (\stree z))\\
&\ + \lambda_{\beta, \gamma} B^{+}_{\alpha}(\stree x) \diamond B^{+}_{\beta \cdot \gamma} (\stree y \diamond \stree z)\\
=&\ B^{+}_{\alpha \rightarrow (\beta \rightarrow \gamma)} \big(B^{+}_{\alpha \rhd (\beta \rightarrow \gamma)}(\stree x) \diamond (B^{+}_{\beta \rhd \gamma}(\stree y) \diamond \stree z) \big)\\
&\ +B^{+}_{\alpha \leftarrow (\beta \rightarrow \gamma)} \big(\stree x \diamond B^{+}_{\alpha \lhd (\beta \rightarrow \gamma)}(B^{+}_{\beta \rhd \gamma}(\stree y) \diamond \stree z) \big)\\
&\ +\lambda_{\alpha, \beta \rightarrow \gamma} B^{+}_{\alpha \cdot (\beta \rightarrow \gamma)} \big(\stree x \diamond (B^{+}_{\beta \rhd \gamma}(\stree y) \diamond \stree z) \big)\\
&\ +B^{+}_{\alpha \rightarrow (\beta \leftarrow \gamma)}\big(B^{+}_{\alpha \rhd (\beta \leftarrow \gamma)} (\stree x) \diamond (\stree y \diamond B^{+}_{\beta \lhd \gamma}(\stree z)) \big)\\
&\ +B^{+}_{\alpha \leftarrow (\beta \leftarrow \gamma)} \big(\stree x \diamond B^{+}_{\alpha \lhd (\beta \leftarrow \gamma)}(\stree y \diamond B^{+}_{\beta \lhd \gamma}(\stree z)) \big)\\
&\ + \lambda_{\alpha, \beta \leftarrow \gamma} B^{+}_{\alpha \cdot (\beta \leftarrow \gamma)} \big(\stree x \diamond (\stree y \diamond B^{+}_{\beta \lhd \gamma} (\stree z)) \big)\\
&\ +\lambda_{\beta, \gamma} B^{+}_{\alpha \rightarrow (\beta \cdot \gamma)} \big(B^{+}_{\alpha \rhd(\beta \cdot \gamma)}(\stree x) \diamond (\stree y \diamond \stree z) \big)\\
&\ + \lambda_{\beta, \gamma} B^{+}_{\alpha \leftarrow (\beta \cdot \gamma)} \big(\stree x \diamond B^{+}_{\alpha \leftarrow (\beta \cdot \gamma)}(\stree y \diamond \stree z) \big)\\
&\ +\lambda_{\beta, \gamma} \lambda_{\alpha, \beta \cdot \gamma} B^{+}_{\alpha \cdot (\beta \cdot \gamma)} \big(\stree x \diamond (\stree y \diamond \stree z) \big)\\
=&\ B^{+}_{\alpha \rightarrow (\beta \rightarrow \gamma)} \big( B^{+}_{(\alpha \rhd (\beta \rightarrow \gamma)) \rightarrow (\beta \rhd \gamma)}(B^{+}_{(\alpha \rhd (\beta \rightarrow \gamma)) \rhd (\beta \rhd \gamma)}(\stree x) \diamond \stree y) \diamond \stree z \big)\\
&\ +B^{+}_{\alpha \rightarrow (\beta \rightarrow \gamma)} \big(B^{+}_{\alpha \rhd (\beta \rightarrow \gamma) \leftarrow (\beta \rhd \gamma)}(\stree x \diamond B^{+}_{(\alpha \rhd (\beta \rightarrow \gamma) \lhd (\beta \rhd \gamma))}(\stree y)) \diamond \stree z  \big)\\
&\ + \lambda_{\alpha \rhd (\beta \rightarrow \gamma),\beta \rhd \gamma} B^{+}_{\alpha \rightarrow (\beta \rightarrow \gamma)} (B^{+}_{(\alpha \rhd (\beta \rightarrow \gamma)) \cdot (\beta \rhd \gamma)} (\stree x \diamond \stree y) \diamond \stree z)\\
&\ +B^{+}_{\alpha \leftarrow (\beta \rightarrow \gamma)} \big(\stree x \diamond B^{+}_{\alpha \lhd (\beta \rightarrow \gamma)}(B^{+}_{\beta \rhd \gamma}(\stree y) \diamond \stree z) \big)\\
&\ +\lambda_{\alpha, \beta \rightarrow \gamma} B^{+}_{\alpha \cdot (\beta \rightarrow \gamma)} \big(\stree x \diamond (B^{+}_{\beta \rhd \gamma}(\stree y) \diamond \stree z) \big)\\
&\ +B^{+}_{\alpha \rightarrow (\beta \leftarrow \gamma)}\big(B^{+}_{\alpha \rhd (\beta \leftarrow \gamma)} (\stree x) \diamond (\stree y \diamond B^{+}_{\beta \lhd \gamma}(\stree z)) \big)\\
&\ +B^{+}_{\alpha \leftarrow (\beta \leftarrow \gamma)} \big(\stree x \diamond B^{+}_{\alpha \lhd (\beta \leftarrow \gamma)}(\stree y \diamond B^{+}_{\beta \lhd \gamma}(\stree z)) \big)\\
&\ + \lambda_{\alpha, \beta \leftarrow \gamma} B^{+}_{\alpha \cdot (\beta \leftarrow \gamma)} \big(\stree x \diamond (\stree y \diamond B^{+}_{\beta \lhd \gamma} (\stree z)) \big)\\
&\ +\lambda_{\beta, \gamma} B^{+}_{\alpha \rightarrow (\beta \cdot \gamma)} \big(B^{+}_{\alpha \rhd(\beta \cdot \gamma)}(\stree x) \diamond (\stree y \diamond \stree z) \big)\\
&\ + \lambda_{\beta, \gamma} B^{+}_{\alpha \leftarrow (\beta \cdot \gamma)} \big(\stree x \diamond B^{+}_{\alpha \leftarrow (\beta \cdot \gamma)}(\stree y \diamond \stree z) \big)\\
&\ +\lambda_{\beta, \gamma} \lambda_{\alpha, \beta \cdot \gamma} B^{+}_{\alpha \cdot (\beta \cdot \gamma)} \big(\stree x \diamond (\stree y \diamond \stree z) \big)\\
=&\
\XX[scale=1.6]{\xxr{-4}4\xxr{-7.5}{7.5}
\node at (0,-0.3) {$\bullet$};
\node at (-2,0.1) {\tiny $(\alpha \rhd (\beta \rightarrow \gamma)) \rightarrow (\beta \rhd \gamma)$};
\node at (-2.2,0.4) {\tiny $(\alpha \rhd (\beta \rightarrow \gamma)) \rhd (\beta \rhd \gamma)$};
\node at (-1.2,-0.15) {\tiny $\alpha \rightarrow (\beta \rightarrow \gamma)$};
\xxhu00{z} \xxhu[0.1]{-4}4{y} \xxhu[0.1]{-7.5}{7.5}{x}
}+
\XX[scale=1.6]{\xxr{-5}5\xxl{-2}8
\node at (0,-0.3) {$\bullet$};
\node at (-2,0.1) {\tiny $(\alpha \rhd (\beta \rightarrow \gamma))\leftarrow (\beta \rhd \gamma)$};
\node at (-0.2,0.55) {\tiny $(\alpha \rhd (\beta \rightarrow \gamma)) \lhd (\beta \rhd \gamma)$};
\node at (-1.2,-0.15) {\tiny $\alpha \rightarrow (\beta \rightarrow \gamma)$};
\xxhu00z
\xxhu[0.1]{-5}5{x} \xxhu[0.1]{-2}8{y}
}+\lambda_{\alpha \rhd (\beta \rightarrow \gamma),\beta \rhd \gamma} \XX[scale=1.6]{\xxr{-5}{5} \draw(-0.5,0.5)--(-0.5,1);
\node at (0,-0.3) {$\bullet$};
\xxhu00z
\node at (-0.72,0.85){$x$};
\node at (-0.38,0.85){$y$};
\node at (-2,0.2){\tiny $(\alpha \rhd (\beta \rightarrow \gamma)) \cdot (\beta \rhd \gamma)$};
\node at (-1,-0.15){\tiny $\alpha \rightarrow (\beta \rightarrow \gamma)$};
}\\
&\ +
\XX[scale=1.6]{\xxl55\xxr28
\node at (0,-0.3) {$\bullet$};
\node at (1.3,0.15) {\tiny $\alpha \lhd (\beta \rightarrow \gamma)$};
\node at (0,0.5) {\tiny $\beta \rhd \gamma$};
\node at (1.2,-0.15) {\tiny $\alpha \leftarrow (\beta \rightarrow \gamma)$};
\xxhu[0.1]00{x\,}
\xxhu[0.1]55{\,z} \xxhu[0.1]28{y}
}+ \lambda_{\alpha, \beta \rightarrow \gamma} \XX[scale=1.6]{\xxlr0{7.5} \draw(0,0)--(0,0.75);
\xxh0023{x\ \,}{0.3} \xxhu[0.12]0{7.5}{z}
\node at (0.16,0.35) {\tiny $y$};
\node at (0,-0.3) {$\bullet$};
\node at (0.25,0.62) {\tiny $\beta \rhd \gamma$};
\node at (1,-0.15){\tiny $\alpha \cdot (\beta \rightarrow \gamma)$};
}+
\XX[scale=1.6]{\xxr{-6}6\xxl66
\node at (0,-0.3) {$\bullet$};
\node at (-1.2,0.15) {\tiny $\alpha \rhd (\beta \leftarrow \gamma)$};
\node at (0.8,0.15) {\tiny $\beta \lhd \gamma$};
\node at (-1.2,-0.15) {\tiny $\alpha \rightarrow (\beta \leftarrow \gamma)$};
\xxhu00{y} \xxhu66{z} \xxhu{-6}6{x}
}+
\XX[scale=1.6]{\xxl44\xxl{7.5}{7.5}
\node at (0,-0.3) {$\bullet$};
\node at (1.2,0.1) {\tiny $\alpha \lhd (\beta \leftarrow \gamma)$};
\node at (0.9,0.4) {\tiny $\beta \lhd \gamma$};
\node at (1.2,-0.15) {\tiny $\alpha \leftarrow (\beta \leftarrow \gamma)$};
\xxhu00{x} \xxhu[0.1]44{y} \xxhu[0.1]{7.5}{7.5}{z}
}\\
&\ +\lambda_{\alpha, \beta \leftarrow \gamma} \XX[scale=1.6]{\xxl{7}{7} \draw(0,0)--(0,1);
\node at (0,-0.3) {$\bullet$};
\node at (-0.2,0.4) {$x$};
\node at (0.2,0.4){$y$};
\node at (0.75,0.9){$z$};
\node at (1,0.3){\tiny $\beta \lhd \gamma$};
\node at (1,-0.15) {\tiny $\alpha \cdot (\beta \leftarrow \gamma)$};
}+
\lambda_{\beta, \gamma} \XX[scale=1.6]{\xxr{-7.5}{7.5} \draw (0,0)--(0,1);
\node at (-1,-0.15) {\tiny $\alpha \rightarrow (\beta \cdot \gamma)$};
\node at (-1.3,0.4) {\tiny $\alpha \rhd (\beta \cdot \gamma)$};
\node at (-0.15,0.4){$y$};
\node at (0.15,0.4){$z$};
\xxhu[0.1]{-7.5}{7.5}{x}
\node at (0,-0.3) {$\bullet$};
}+\lambda_{\beta, \gamma}\XX[scale=1.6]{\xxl{5}{5} \draw (0.5,0.5)--(0.5,1);
\node at (0,-0.3) {$\bullet$};
\node at (0.35,0.85) {$y$}; \node at (0.65,0.85){$z$};
\node at (1.2,0.2) {\tiny $\alpha \lhd (\beta \cdot \gamma)$};
\node at (1.2,-0.15) {\tiny $\alpha \leftarrow (\beta \cdot \gamma)$};
\xxhu00x
}+ \lambda_{\beta, \gamma} \lambda_{\alpha, \beta \cdot \gamma} \XX[scale=1.6]{\xx00{1.6}\xx00{2.4}
\xxh001{1.6}{\ \ \,z}{0.6}
\xxh00{1.6}{2.4}{y}{0.5}
\xxh00{2.4}3{x\ \ }{0.6}
\node at (0,-0.3){$\bullet$};
\node at (1,-0.15){$\tiny \alpha \cdot (\beta \cdot \gamma)$};
}.
\end{align*}
By Lemma~\ref{lem:asso} and identifying the types of the planar rooted trees, we get that 
$(\Omega, \leftarrow, \rightarrow, \lhd$, $\rhd, \cdot, \lambda)$ is a $\lambda$-$\ets$.

\ref{it3} $\Longrightarrow$ \ref{it1} By Lemma~\ref{lem:asso} and the definition of $\diamond$, $(\bfk \calt, \diamond, (B^{+}_{\omega})_{\omega \in \Omega})$ is an $\Omega$-Rota-Baxter algebra. Now we show the freeness of $\bfk \calt$.

Let $(R, \cdot, (P_{\omega})_{\omega \in \Omega})$ be an $\Omega$-Rota-Baxter algebra of weight $\lambda_{\Omega}$ and $f : X \rightarrow R$ a set map. We extend $f$ to an $\Omega$-Rota-Baxter algebra morphism $\overline{f}: \bfk \calt \rightarrow R$ such that $\overline{f} \circ i=f$.

For $T \in \calt$, we define $\overline{f}(T)$ by induction on $\dep(T)$. If $\dep(T)=1$, then $T$ is of the form
\begin{align*}
T =\treeoo{\cdb o\cdx[1.5]{o}{a1}{160}
\cdx[1.5]{o}{a2}{120}
\cdx[1.5]{o}{a3}{60}
\cdx[1.5]{o}{a4}{20}
\node at (140:1.2*\xch) {$x_1$};
\node at (90:1.2*\xch) {$\cdots$};
\node at (40:1.2*\xch) {$x_m$};
}.
\end{align*}
Define
\begin{align*}
\overline{f}(T):=f(x_1) \cdot f(x_2) \cdots f(x_m).
\end{align*}

For the induction step of $\dep(T) \geq 2$, we define $\overline{f}(T)$ by induction on the branches of $T$. If $\bra(T)=1$, then $T$ is of the form
\begin{align*}
T=\treeoo{\cdb o\ocdx[2]{o}{a1}{160}{T_1}{left}
\ocdx[2]{o}{a2}{120}{T_2}{above}
\ocdx[2]{o}{a3}{60}{T_m}{above}
\ocdx[2]{o}{a4}{20}{T_{m+1}}{right}
\node at (140:\xch) {$x_1$};
\node at (90:\xch) {$\cdots$};
\node at (40:\xch) {$x_m$};
\node[below] at ($(o)!0.7!(a1)$) {$\alpha_1$};
\node[below] at ($(o)!0.7!(a4)$) {$\alpha_{m+1}$};
\node[right] at ($(o)!0.85!(a2)$) {$\alpha_2$};
\node[left] at ($(o)!0.85!(a3)$) {$\alpha_m\!$};
\node at (0,-0.35) {$\bullet$};
\node at (0.2,-0.15) {$\omega$};
}.
\end{align*}
Define
\[\overline{f}(T):=P_{\omega} \left( P_{\alpha_1} \big(\overline{f}(T_1) \big) \cdot f(x_1) \cdot P_{\alpha_2} \big(\overline{f}(T_2) \big) \cdots P_{\alpha_m} \big( \overline{f}(T_m) \big) \cdot f(x_m) \cdot P_{\alpha} \big(\overline{f} (T_{m+1}) \big) \right).\]
If $\bra(T)>1$, then $T$ is of the form
\begin{align*}
T=\treeoo{\cdb o\ocdx[2]{o}{a1}{160}{T_1}{left}
\ocdx[2]{o}{a2}{120}{T_2}{above}
\ocdx[2]{o}{a3}{60}{T_m}{above}
\ocdx[2]{o}{a4}{20}{T_{m+1}}{right}
\node at (140:\xch) {$x_1$};
\node at (90:\xch) {$\cdots$};
\node at (40:\xch) {$x_m$};
\node[below] at ($(o)!0.7!(a1)$) {$\alpha_1$};
\node[below] at ($(o)!0.7!(a4)$) {$\alpha_{m+1}$};
\node[right] at ($(o)!0.85!(a2)$) {$\alpha_2$};
\node[left] at ($(o)!0.85!(a3)$) {$\alpha_m\!$};
}.
\end{align*}
Define
\[\overline{f}(T):=P_{\alpha_1} \big( \overline{f}(T_1) \big) \cdot f(x_1) \cdot P_{\alpha_2} \big(\overline{f}(T_2) \big) \cdots P_{\alpha_m} \big( \overline{f}(T_m) \big) \cdot f(x_m) \cdot P_{\alpha} \big(\overline{f} (T_{m+1}) \big).\]

By  construction of $\overline{f}$, $\overline{f} \circ i=f$ and $ P_{\omega} \overline{f}=\overline{f} B^{+}_{\omega} $ for all $\omega \in \Omega$. Next we show that $\overline{f}$ is an algebra homomorphism, i.e.
\begin{align}
\overline{f}(T \diamond T')=\overline{f}(T) \cdot \overline{f}(T') \,\, \text{ for all $T,T' \in \calt$.}
\label{eq:algmor}
\end{align}

We prove Eq.~(\ref{eq:algmor}) by induction on $\dep(T)+\dep(T')$. If $\dep(T)+\dep(T')=2$, then $\dep(T)=\dep(T')=1$ and
\begin{align*}
T &=\treeoo{\cdb o\cdx[1.5]{o}{a1}{160}
\cdx[1.5]{o}{a2}{120}
\cdx[1.5]{o}{a3}{60}
\cdx[1.5]{o}{a4}{20}
\node at (140:1.2*\xch) {$x_1$};
\node at (90:1.2*\xch) {$\cdots$};
\node at (40:1.2*\xch) {$x_m$};
}&&\text{ and }&T'&=\treeoo{\cdb o\cdx[1.5]{o}{a1}{160}
\cdx[1.5]{o}{a2}{120}
\cdx[1.5]{o}{a3}{60}
\cdx[1.5]{o}{a4}{20}
\node at (140:1.2*\xch) {$y_1$};
\node at (90:1.2*\xch) {$\cdots$};
\node at (40:1.2*\xch) {$y_n$};
},&&\text{ with }\,m,n\geq 0,
\end{align*}
and
\begin{align*}
\overline{f}(T \diamond T')= f(x_1) \cdots f(x_m) \cdot f(y_1) \cdots f(y_n)=(f(x_1) \cdots f(x_m)) \cdot (f(y_1) \cdots f(y_n))=\overline{f}(T) \diamond \overline{f}(T').
\end{align*}

For the induction step of $\dep(T)+\dep(T') \geq 3$. If $T \diamond T'$ belongs to the first three cases, then $\overline{f}(T \diamond T')=\overline{f}(T) \cdot \overline{f}(T')$ by the definition of $\diamond$ and the construction of $\overline{f}$. So we only need to consider the fourth case. Then
\begin{align*}
\overline{f}(T \diamond T')=&\ \overline{f} \Biggl( \Biggl(\treeoo{\cdb o\ocdx[2]{o}{a1}{160}{T_1}{left}
\ocdx[2]{o}{a2}{120}{T_2}{above}
\ocdx[2]{o}{a3}{60}{T_m}{above}
\cdx[2]{o}{a4}{20}
\node at (140:\xch) {$x_1$};
\node at (90:\xch) {$\cdots$};
\node at (40:\xch) {$x_m$};
\node[below] at ($(o)!0.7!(a1)$) {$\alpha_1$};
\node[right] at ($(o)!0.85!(a2)$) {$\alpha_2$};
\node[left] at ($(o)!0.85!(a3)$) {$\alpha_m\!$};
}\diamond \bigg(B^{+}_{\alpha_{m+1}\rightarrow \beta_1}\big(B^{+}_{\alpha_{m+1} \rhd \beta_1}(T_{m+1})\diamond T'_1 \big)
+B^{+}_{\alpha_{m+1} \leftarrow \beta_1} \big( T_{m+1} \diamond B^{+}_{\alpha_{m+1} \lhd \beta_1}(T'_1)\big)\\
&\ + \lambda_{\alpha_{m+1}, \beta_1} B^{+}_{\alpha_{m+1} \cdot \beta_1} \big(T_{m+1} \diamond T'_{1} \big) \bigg) \Biggl) \diamond\treeoo{\cdb o
\cdx[2]{o}{a1}{160}
\ocdx[2]{o}{a2}{120}{T'_2}{above}
\ocdx[2]{o}{a3}{60}{T'_n}{above}
\ocdx[2]{o}{a4}{20}{T'_{n+1}}{right}
\node at (140:\xch) {$y_1$};
\node at (90:\xch) {$\cdots$};
\node at (40:\xch) {$y_n$};
\node[below] at ($(o)!0.7!(a4)$) {$\beta_{n+1}$};
\node[right] at ($(o)!0.85!(a2)$) {$\beta_2$};
\node[left] at ($(o)!0.85!(a3)$) {$\beta_n\!$};
} \Biggl)\\
=&\ \bigg(P_{\alpha_1}(\overline{f}(T_1)) \cdot f(x_1) \cdot P_{\alpha_2}(\overline{f}(T_2)) \cdots P_{\alpha_m}(\overline{f}(T_m)) \cdot f(x_m) \cdot \overline{f} \Big(B^{+}_{\alpha_{m+1}\rightarrow \beta_1}\big(B^{+}_{\alpha_{m+1} \rhd \beta_1}(T_{m+1})\diamond T'_1 \big)\\
&\ +B^{+}_{\alpha_{m+1} \leftarrow \beta_1} \big( T_{m+1} \diamond B^{+}_{\alpha_{m+1} \lhd \beta_1}(T'_1)\big)
+ \lambda_{\alpha_{m+1}, \beta_1} B^{+}_{\alpha_{m+1} \cdot \beta_1} \big(T_{m+1} \diamond T'_{1} \big) \Big) \bigg) \\
&\ \cdot f(y_1) \cdot P_{\beta_2}(\overline{f}(T'_2)) \cdots P_{\beta_{n+1}}(\overline{f}(T'_{n+1}))\\
=&\ \Big( P_{\alpha_1}(\overline{f}(T_1)) \cdot f(x_1) \cdot P_{\alpha_2}(\overline{f}(T_2)) \cdots P_{\alpha_m}(\overline{f}(T_m)) \cdot f(x_m) \cdot P_{\alpha_{m+1}}(\overline{f}(T_{m+1})) \Big)\\
&\ \cdot \Big(P_{\beta_1}(f(T'_1))\cdot f(y_1) \cdot P_{\beta_2}(\overline{f}(T'_2)) \cdots P_{\beta_{n+1}}(\overline{f}(T'_{n+1})) \Big)\\
=&\ \overline{f}(T) \diamond \overline{f}(T').
\end{align*}

Moreover, by the construction of $\overline{f}$, it is the unique way to extend $f$ as an $\Omega$-Rota-Baxter algebra morphism. Hence $(\bfk \calt, \diamond, (B^{+}_{\omega})_{\omega \in \Omega})$ together with the map $i$ is the free $\Omega$-Rota-Baxter algebra generated by $X$.
\end{proof}

\begin{remark}
\begin{enumerate}
\item In Definition~\ref{def:orb}, $\Omega$ is required to be a set with five products $\leftarrow, \rightarrow, \lhd, \rhd, \cdot$ and $\lambda$ is required to be a family of elements in ${\bfk}$ indexed by $\Omega^2$. 
This defines a category of $\Omega$-Rota-Baxter algebras for any such $\Omega$. 
Generally, free $\Omega$-Rota-Baxter algebras are not based on $\Omega$-angularly decorated planar trees.
However, by Theorem~\ref{propRB}, the condition of a free $\Omega$-Rota-Baxter algebra based on the combinatorics
of $\Omega$-angularly decorated planar trees, similar to the one of (classical) Rota-Baxter algebras,
is equivalent to $(\Omega, \leftarrow, \rightarrow, \lhd, \rhd,\cdot,\lambda)$ being a $\lambda$-$\ets$.

\item As a particular case, we recover the description of free family Rota-Baxter algebras of \cite{ZGM}.
An alternative description of free Rota-Baxter algebras (with rooted forests) is done in \cite{EFG08}.
\end{enumerate}
\end{remark}

Taking all elements in $\lambda$ to be 0, we get the following result:
\begin{coro}
Let $\Omega$ be a set with four products $\leftarrow, \rightarrow, \lhd, \rhd$. Then the following conditions are equivalent:
\begin{enumerate}
\item \label{it1bis} $(\bfk \calt, \diamond, (B^{+}_{\omega})_{\omega \in \Omega})$ together with the map $i$ is the free $\Omega$-Rota-Baxter algebra of weight 0 generated by $X$.
\item \label{it2bis} $(\bfk \calt, \diamond, (B^{+}_{\omega})_{\omega \in \Omega})$ is an $\Omega$-Rota-Baxter algebra of weight 0.
\item \label{it3bis} $(\Omega, \leftarrow, \rightarrow, \lhd, \rhd)$ is an EDS.
\end{enumerate}
\end{coro}

\subsection{Commutative $\Omega$-Rota-Baxter algebras on typed words}\label{subsection2c}

Let $\Omega$ be a set and $V$ a vector space. Recall from~\cite{Foi20} that the space of $\Omega$-typed words in $V$ is
\begin{align*}
\shh(V)=\bigoplus_{n \geq 1} (\bfk \Omega)^{\ot (n-1)} \ot V^{\ot n}.
\end{align*}
For the  ease of statement, we redefine the space of $\Omega$-typed words in $V$ as
\begin{align*}
\shh(V)=\bigoplus_{n \geq 0} \underbrace{V \ot (\bfk \Omega) \ot \cdots \ot (\bfk \Omega) \ot V}_{\text{$(n+1)$'s $V$ and $n$'s
$(\bfk \Omega)$}}
\end{align*}
and write each pure tensor $\mathbf{v}=v_0 \ot \omega_1 \ot \cdots \ot \omega_n \ot v_n \in \Omega$ under the form
\begin{align*}
\mathbf{v}=v_0 \ot_{\omega_1} v_1 \ot_{\omega_2} \cdots \ot_{\omega_n} v_n,
\end{align*}
where $n \geq 0$, $\omega_1, \cdots, \omega_n \in \Omega$ and $v_0, \cdots, v_n \in V$ with the convention $\mathbf{v}=v_0$ if $n=0$. We call $\mathbf{v}$ an {\bf $\Omega$-typed word} in $V$ and define its {\bf length} $\ell(\mathbf{v}):=n+1$.

Let $A$ be an algebra with identity $1_A$, $(\Omega,\leftarrow, \rightarrow, \lhd, \rhd, \cdot)$ be a set with five products
and $\lambda=(\lambda_{\alpha,\beta})_{(\alpha,\beta)\in \Omega^2}$ be a family of elements in ${\bfk}$ indexed by $\Omega^2$. For any pure tensors $\mathbf{a}=a_0 \ot_{\alpha_1} \mathbf{a}', \mathbf{b}=b_0 \ot_{\beta_1} \mathbf{b}' \in \shh(A)$ with $\ell(\mathbf{a})=m$ and $\ell(\mathbf{b})=n$, define $\mathbf{a} \diamond \mathbf{b}$ inductively as follows:
\begin{equation}
\mathbf{a} \diamond  \mathbf{b}:=
\left\{
\begin{array}{l}
\ a_0b_0,  \hspace{5cm} \text{if $m=n=0$}, \\
\ a_0b_0 \ot_{\alpha_1} \mathbf{a}',  \hspace{4cm} \text{if $m>0, n=0$},\\
\ a_0b_0 \ot_{\beta_1} \mathbf{b}',  \hspace{4cm} \text{if $m=0, n>0$},\\
\ a_0b_0 \ot_{\alpha_1 \rightarrow \beta_1} \big((1_A \ot_{\alpha_1 \rhd \beta_1} \mathbf{a}') \diamond \mathbf{b}' \big) + a_0b_0 \ot_{\alpha \leftarrow \beta_1} \big(\mathbf{a}' \diamond(1_A \ot_{\alpha_1 \lhd \beta_1} \mathbf{b}') \big)  \\
\hspace{3mm} + \lambda_{\alpha_1 , \beta_1} a_0b_0 \ot_{\alpha_1 \cdot \beta_1}(\mathbf{a}' \diamond \mathbf{b}'),  \hspace{1.5cm} \text{if $m>0, n>0$}.\\
\end{array}
\right .
\mlabel{eq:shuffle}
\end{equation}
Extending bilinearly, we construct a product $\diamond$ on $\shh(A)$.

\begin{lemma}
Let $A$ be an algebra with identity $1_A$, $\Omega$ a set with five products $\leftarrow, \rightarrow, \lhd, \rhd, \cdot$ and $\lambda$ a family of elements in ${\bfk}$ indexed by $\Omega^2$. If $(\Omega, \leftarrow, \rightarrow, \lhd, \rhd, \cdot ,\lambda)$ is a $\lambda$-ESD, then $(\shh(A), \diamond)$ is an associative algebra with identity $1_A$.
\end{lemma}

\begin{proof}
By Eq.~(\ref{eq:shuffle}), $\shh(A)$ is closed under $\diamond$ and  $1_A$ is the identity of $\diamond$.

For pure tensors $\mathbf{a},\mathbf{b},\mathbf{c} \in \shh(A)$, we prove
\begin{align}
(\mathbf{a} \diamond \mathbf{b}) \diamond \mathbf{c}=\mathbf{a} \diamond (\mathbf{b} \diamond \mathbf{c})
\label{eq:asso1}
\end{align}
by induction on $\ell(\mathbf{a})+\ell(\mathbf{b})+\ell(\mathbf{c})$. If $\ell(\mathbf{a})+\ell(\mathbf{b})+\ell(\mathbf{c})=3$, then $\ell(\mathbf{a})=\ell(\mathbf{b})=\ell(\mathbf{c})=1$ and $\mathbf{a}=a_0, \mathbf{b}=b_0, \mathbf{c}=c_0$. Hence
\begin{align*}
(\mathbf{a} \diamond \mathbf{b}) \diamond \mathbf{c}=a_0b_0c_0=\mathbf{a} \diamond (\mathbf{b} \diamond \mathbf{c}).
\end{align*}

Suppose Eq.~(\ref{eq:asso1}) holds for $\ell(\mathbf{a})+\ell(\mathbf{b})+\ell(\mathbf{c})\leq p$, where $p \geq 3$ is a fixed integer. Consider the case of  $\ell(\mathbf{a})+\ell(\mathbf{b})+\ell(\mathbf{c})=p+1$. If one of $\ell(\mathbf{a}), \ell(\mathbf{b}), \ell(\mathbf{c})$ is equal to 1, then Eq.~(\ref{eq:asso1}) holds by  direct calculation. Hence we assume $\ell(\mathbf{a})>1, \ell(\mathbf{b})>1, \ell(\mathbf{c})>1$ and
\begin{align*}
\mathbf{a}=a_0 \ot_{\alpha_1} \mathbf{a}', \,\, \mathbf{b}=b_0 \ot_{\beta_1} \mathbf{b}', \,\, \mathbf{c}=c_0 \ot_{\gamma_1} \mathbf{c}'.
\end{align*}
Then
\begin{align*}
&\ (\mathbf{a} \diamond \mathbf{b}) \diamond \mathbf{c}\\
=&\ (a_0b_0)c_0 \ot_{(\alpha_1 \rightarrow \beta_1) \rightarrow \gamma_1}\big((1_A \ot_{(\alpha_1 \rightarrow \beta_1) \rhd \gamma_1} ((1_A \ot_{\alpha_1 \ot \rhd \beta_1} \mathbf{a}') \diamond \mathbf{b}')) \diamond \mathbf{c}' \big)\\
&\ +(a_0b_0)c_0 \ot_{(\alpha_1 \rightarrow \beta_1) \leftarrow \gamma_1} \big( ((1_A \ot_{\alpha_1 \rhd \beta_1} \mathbf{a}') \diamond \mathbf{b}') \diamond (1_A \ot_{(\alpha_1 \rightarrow \beta_1) \lhd \gamma_1} \mathbf{c}') \big)\\
&\ +\lambda_{(\alpha_1 \rightarrow \beta_1), \gamma_1} (a_0b_0)c_0\ot _{(\alpha_1 \rightarrow \beta_1) \cdot \gamma_1}\big(((1_A \ot_{\alpha_1 \rhd \beta_1} \mathbf{a}') \diamond \mathbf{b}') \diamond \mathbf{c}'\big)\\
&\ +(a_0b_0)c_0 \ot_{(\alpha_1 \leftarrow \beta_1) \rightarrow \gamma_1} \big((1_A \ot_{(\alpha_1 \leftarrow \beta_1) \rhd \gamma_1} (\mathbf{a}' \diamond (1_A \ot_{\alpha_1 \lhd \beta_1} \mathbf{b}'))) \diamond \mathbf{c}' \big)\\
&\ +(a_0b_0)c_0 \ot_{(\alpha_1 \leftarrow \beta_1) \leftarrow \gamma_1} \big(\mathbf{a}' \diamond (1_A \ot_{(\alpha_1 \lhd \beta_1) \rightarrow ((\alpha_1 \leftarrow \beta_1) \lhd \gamma_1)} ((1_A \ot_{(\alpha_1 \lhd \beta_1) \rhd ((\alpha_1 \leftarrow \beta_1) \lhd \gamma_1)} \mathbf{b}') \diamond \mathbf{c}') )\big)\\
&\ +(a_0b_0)c_0 \ot_{(\alpha_1 \leftarrow \beta_1) \leftarrow \gamma_1} \big(\mathbf{a}' \diamond (1_A \ot_{(\alpha_1 \lhd \beta_1) \leftarrow ((\alpha_1 \leftarrow \beta_1)\lhd \gamma_1)}(\mathbf{b}' \diamond (1_A \ot_{(\alpha_1 \lhd \beta_1) \lhd ((\alpha_1 \leftarrow \beta_1) \lhd \gamma_1)} \mathbf{c}'))) \big)\\
&\ +\lambda_{(\alpha_1 \lhd \beta_1) , ((\alpha_1 \leftarrow \beta_1) \lhd \gamma_1)} (a_0b_0)c_0 \ot_{(\alpha_1 \leftarrow \beta_1)\leftarrow \gamma_1} \big(\mathbf{a}' \diamond (1_A \ot_{(\alpha_1 \lhd \beta_1) \cdot ((\alpha_1 \leftarrow \beta_1) \lhd \gamma_1)}(\mathbf{b}' \diamond \mathbf{c}')) \big)\\
&\ +\lambda_{(\alpha_1 \leftarrow \beta_1), \gamma_1} (a_0b_0)c_0 \ot_{(\alpha_1 \leftarrow \beta_1) \cdot \gamma_1} \big((\mathbf{a}' \diamond (1_A \ot_{\alpha_1 \lhd \beta_1} \mathbf{b}'))\diamond \mathbf{c}' \big)\\
&\ +\lambda_{\alpha_1 , \beta_1} (a_0b_0)c_0 \ot_{(\alpha_1 \cdot \beta_1) \rightarrow \gamma_1} \big((1_A \ot_{(\alpha_1 \cdot \beta_1)\rhd \gamma_1} (\mathbf{a}' \diamond \mathbf{b}'))\diamond \mathbf{c}' \big)\\
&\ +\lambda_{\alpha_1 , \beta_1} (a_0b_0)c_0 \ot_{(\alpha_1 \cdot \beta_1)\leftarrow \gamma_1}\big((\mathbf{a}' \diamond \mathbf{b}') \diamond (1_A \ot_{(\alpha_1 \cdot \beta_1)\lhd \gamma_1} \mathbf{c}') \big)\\
&\ +\lambda_{\alpha_1 , \beta_1}\lambda_{(\alpha_1 \cdot \beta_1) , \gamma_1} (a_0b_0)c_0 \ot_{(\alpha_1 \cdot \beta_1) \cdot \gamma_1} \big((\mathbf{a}' \diamond \mathbf{b}') \diamond \mathbf{c}' \big)
\end{align*}
and
\begin{align*}
&\ \mathbf{a} \diamond (\mathbf{b} \diamond \mathbf{c})\\
=&\ a_0(b_0c_0) \ot_{\alpha_1 \rightarrow (\beta_1 \rightarrow \gamma_1)} \big((1_A \ot_{(\alpha_1 \rhd(\beta_1 \rightarrow \gamma_1)) \rightarrow (\beta_1 \rhd \gamma_1)}((1_A \ot_{(\alpha_1 \rhd (\beta_1 \rightarrow \gamma_1)) \rhd (\beta_1 \rhd \gamma_1)} \mathbf{a}') \diamond \mathbf{b}'))\diamond \mathbf{c}' \big)\\
&\ +a_0(b_0c_0) \ot_{\alpha_1 \rightarrow (\beta_1 \rightarrow \gamma_1)} \big((1_A \ot_{(\alpha_1 \rhd(\beta_1 \rightarrow \gamma_1))\leftarrow (\beta_1 \rhd \gamma_1)}(\mathbf{a}' \diamond (1_A \ot_{(\alpha_1 \rhd (\beta_1 \rightarrow \gamma_1)) \lhd (\beta_1 \rhd \gamma_1)} \mathbf{b}'))) \diamond \mathbf{c}' \big)\\
&\ +\lambda_{(\alpha_1 \rhd (\beta_1 \rightarrow \gamma_1)) , (\beta_1 \rhd \gamma_1)} a_0(b_0c_0) \ot_{\alpha_1 \rightarrow(\beta_1 \rightarrow \gamma_1)} \big((1_A \ot_{(\alpha_1 \rhd (\beta_1 \rightarrow \gamma_1)) \cdot (\beta_1 \rhd \gamma_1)} (\mathbf{a}' \diamond \mathbf{b}'))\diamond \mathbf{c}' \big)\\
&\ +a_0(b_0c_0)\ot_{\alpha_1 \leftarrow (\beta_1 \rightarrow \gamma_1)} \big(\mathbf{a}' \diamond (1_A \ot_{\alpha_1 \lhd (\beta_1 \rightarrow \gamma_1)} ((1_A \ot_{\beta_1 \rhd \gamma_1} \mathbf{b}') \diamond \mathbf{c}')) \big)\\
&\ +\lambda_{\alpha_1 , (\beta_1 \rightarrow \gamma_1)} a_0(b_0c_0) \ot _{\alpha_1 \cdot (\beta_1 \rightarrow \gamma_1)} \big(\mathbf{a}' \diamond ((1_A \ot_{\beta_1 \rhd \gamma_1} \mathbf{b}') \diamond \mathbf{c}') \big)\\
&\ +a_0(b_0c_0) \ot _{\alpha_1 \rightarrow (\beta_1 \leftarrow \gamma_1)} \big((1_A \ot_{\alpha_1 \rhd (\beta_1 \leftarrow \gamma_1)} \mathbf{a}') \diamond (\mathbf{b}' \diamond (1_A \ot_{\beta_1 \lhd \gamma_1} \mathbf{c}')) \big)\\
&\ +a_0(b_0c_0) \ot_{\alpha_1 \leftarrow (\beta_1 \leftarrow \gamma_1)} \big(\mathbf{a}' \diamond (1_A \ot_{\alpha_1 \lhd (\beta_1 \leftarrow \gamma_1)}(\mathbf{b}' \diamond (1_A \ot_{\beta_1 \lhd \gamma_1} \mathbf{c}'))) \big)\\
&\ +\lambda_{\alpha_1 , (\beta_1 \leftarrow \gamma_1)} a_0(b_0c_0) \ot_{\alpha_1 \cdot (\beta_1 \leftarrow \gamma_1)} \big(\mathbf{a}' \diamond (\mathbf{b}' \diamond (1_A \ot_{\beta_1 \lhd \gamma_1} \mathbf{c}')) \big)\\
&\ +\lambda_{\beta_1 , \gamma_1} a_0(b_0c_0) \ot_{\alpha_1 \rightarrow (\beta_1 \cdot \gamma_1)} \big((1_A \ot_{\alpha_1 \rhd (\beta_1 \cdot \gamma_1)} \mathbf{a}') \diamond (\mathbf{b}' \diamond \mathbf{c}') \big)\\
&\ +\lambda_{\beta_1 , \gamma_1} a_0(b_0c_0) \ot_{\alpha_1 \leftarrow (\beta_1 \cdot \gamma_1)} \big(\mathbf{a}' \diamond (1_A \ot_{\alpha_1 \lhd (\beta_1 \cdot \gamma_1)} (\mathbf{b}' \diamond \mathbf{c}')) \big)\\
&\ +\lambda_{\beta_1 , \gamma_1} \lambda_{\alpha_1 , (\beta_1 \cdot \gamma_1)} a_0(b_0c_0) \ot_{\alpha_1 \cdot (\beta_1 \cdot \gamma_1)} \big(\mathbf{a}'\diamond (\mathbf{b}' \diamond \mathbf{c}') \big).
\end{align*}
By induction hypothesis and $(\Omega,\leftarrow, \rightarrow,\lhd, \rhd, \cdot ,\lambda)$ being a $\lambda$-$\ets$, $(\mathbf{a} \diamond \mathbf{b}) \diamond \mathbf{c}=\mathbf{a} \diamond (\mathbf{b} \diamond \mathbf{c})$. Hence $(\shh(A), \diamond)$ is an associative algebra with identity $1_A$.
\end{proof}

For each $\omega \in \Omega$, define a linear map $P_{\omega}: \shh(A) \rightarrow \shh(A), \mathbf{a} \mapsto 1_A \ot_{\omega} \mathbf{a}$.
If further $A$ is a commutative algebra and $(\Omega, \leftarrow, \rightarrow, \lhd, \rhd, \cdot,\lambda)$ is a commutative $\lambda$-$\ets$, we get the following result:

\begin{prop} \label{propcommRB}
If $A$ is a commutative algebra with identity $1_A$ and $(\Omega, \leftarrow, \rightarrow, \lhd, \rhd, \cdot,\lambda)$ is a commutative $\lambda$-$\ets$, then $(\shh(A), \diamond,(P_{\omega})_{\omega \in \Omega})$ is the free commutative $\Omega$-Rota-Baxter algebra generated by $A$.
\end{prop}

\begin{proof}
For $\mathbf{a}, \mathbf{b} \in \shh(A)$ and $\alpha, \beta \in \Omega$,
\begin{align*}
&\ P_{\alpha}(\mathbf{a}) \diamond P_{\beta} (\mathbf{b})=(1_A \ot_{\alpha} \mathbf{a}) \diamond (1_A \ot_{\beta} \mathbf{b})\\
=&\ 1_A \ot_{\alpha \rightarrow \beta} \big((1 \ot_{\alpha \rhd \beta} \mathbf{a}) \diamond \mathbf{b} \big)+ 1 \ot_{\alpha \leftarrow \beta} \big(\mathbf{a} \diamond (1_A \ot_{\alpha \lhd \beta} \mathbf{b}) \big)\\
&\ +\lambda_{\alpha , \beta} 1_A \ot_{\alpha \cdot \beta}(\mathbf{a} \diamond \mathbf{b})\\
=&\ P_{\alpha \rightarrow \beta} (P_{\alpha \rhd \beta}(\mathbf{a}) \diamond \mathbf{b})+ P_{\alpha \leftarrow \beta} (\mathbf{a} \diamond P_{\alpha \lhd \beta}(\mathbf{b}))+ \lambda_{\alpha , \beta} P_{\alpha \cdot \beta} (\mathbf{a} \diamond \mathbf{b}),
\end{align*}
hence $\sh(A)$ is an $\Omega$-Rota-Baxter algebra. Next we show
\begin{align}
\mathbf{a} \diamond \mathbf{b}=\mathbf{b} \diamond \mathbf{a}
\label{eq:com}
\end{align}
by induction on $\ell(\mathbf{a})+ \ell(\mathbf{b})$. If $\ell(\mathbf{a})+ \ell(\mathbf{b})=2$, then $\ell(\mathbf{a})=\ell(\mathbf{b})=1$ and
\begin{align*}
\mathbf{a} \diamond \mathbf{b}=a_0 \diamond b_0=a_0b_0=b_0a_0=b_0 \diamond a_0=\mathbf{b} \diamond \mathbf{a}.
\end{align*}
Suppose Eq.~(\ref{eq:com}) holds for $\ell(\mathbf{a})+ \ell(\mathbf{b})< p$, where $p \geq 2$ is a fixed integer. We consider the case of $\ell(\mathbf{a})+\ell(\mathbf{b})=p+1$. If one of $\ell(\mathbf{a}), \ell(\mathbf{b})$ is equal to 1, then Eq.~(\ref{eq:com}) holds directly. We assume that $\mathbf{a}=a_0 \ot_{\alpha_1} \mathbf{a}', \mathbf{b}=b_0 \ot_{\beta_1} \mathbf{b}'$, then
\begin{align*}
&\ \mathbf{a} \diamond \mathbf{b}=(a_0 \ot_{\alpha_1} \mathbf{a}') \diamond (b_0 \ot_{\beta_1} \mathbf{b}')\\
=&\ a_0b_0 \ot_{\alpha_1 \rightarrow \beta_1} \big( (1_A \ot_{\alpha_1 \rhd \beta_1} \mathbf{a}') \diamond \mathbf{b}' \big)+ a_0b_0 \ot_{\alpha_1 \leftarrow \beta_1} \big(\mathbf{a}' \diamond (1_A \ot_{\alpha_1 \lhd \beta_1} \mathbf{b}') \big)+ \lambda_{\alpha_1 , \beta_1} a_0b_0 \ot_{\alpha_1 \cdot \beta_1} (\mathbf{a}' \diamond \mathbf{b}')\\
=&\ b_0a_0 \ot_{\alpha_1 \rightarrow \beta_1} \big( (1_A \ot_{\alpha_1 \rhd \beta_1} \mathbf{a}') \diamond \mathbf{b}' \big)+ b_0a_0 \ot_{\alpha_1 \leftarrow \beta_1} \big(\mathbf{a}' \diamond (1_A \ot_{\alpha_1 \lhd \beta_1} \mathbf{b}') \big)+ \lambda_{\alpha_1 , \beta_1} b_0a_0 \ot_{\alpha_1 \cdot \beta_1} (\mathbf{a}' \diamond \mathbf{b}')\\
&\ \hspace{7cm} \text{(by $A$ being a commutative algebra)}\\
=&\ b_0a_0 \ot_{\alpha_1 \rightarrow \beta_1} \big( \mathbf{b}' \diamond(1_A \ot_{\alpha_1 \rhd \beta_1} \mathbf{a}') \big)+ b_0a_0 \ot_{\alpha_1 \leftarrow \beta_1} \big((1_A \ot_{\alpha_1 \lhd \beta_1} \mathbf{b}')\diamond \mathbf{a}' \big)+ \lambda_{\alpha_1 , \beta_1} b_0a_0 \ot_{\alpha_1 \cdot \beta_1} ( \mathbf{b}' \diamond\mathbf{a}')\\
&\ \hspace{7cm} \text{(by the induction hypothesis)}\\
=&\  b_0a_0 \ot_{\beta_1 \leftarrow \alpha_1} \big( \mathbf{b}' \diamond(1_A \ot_{\beta_1 \lhd \alpha_1} \mathbf{a}') \big)+ b_0a_0 \ot_{\beta_1 \rightarrow \alpha_1} \big((1_A \ot_{\beta_1 \rhd \alpha_1} \mathbf{b}')\diamond \mathbf{a}' \big)+ \lambda_{\beta_1 , \alpha_1} b_0a_0 \ot_{\beta_1 \cdot \alpha_1} ( \mathbf{b}' \diamond\mathbf{a}')\\
&\ \hspace{7cm} \text{(by $\Omega$ being commutative)}\\
=&\ (b_0 \ot_{\beta_1} \mathbf{b}') \diamond (a_0 \ot_{\alpha_1} \mathbf{a}')=\mathbf{b} \ot \mathbf{a}.
\end{align*}
Hence $(\shh(A), \diamond)$ is a commutative algebra.\\

Let $(R,\cdot,(P_{\omega})_{\omega \in \Omega})$ be a commutative $\Omega$-Rota-Baxter algebra and $f : A \rightarrow  R$ a commutative algebra homomorphism. We extend $f$ to
an $\Omega$-Rota-Baxter algebra morphism $\overline{f}: \shh(A) \rightarrow R$ as follows: for $\mathbf{a} \in \shh(A)$, we define $\overline{f}(a)$ by induction on $\ell(\mathbf{a})$. If $\ell(\mathbf{a})=1$, then define $\overline{f}(\mathbf{a})=f(\mathbf{a})$. Suppose $\overline{f}(\mathbf{a})$ has been defined for all $\mathbf{a}$ with $\ell(\mathbf{a}) \leq p$, where $p \geq 1$ is a fixed integer. Consider the case of $\ell(\mathbf{a})=p+1$. We suppose that $\mathbf{a}=a_0 \ot_{\alpha_1} \mathbf{a}'$,
and we then put:
\begin{align*}
\overline{f}(\mathbf{a}) := f(a_0) \cdot P_{\alpha_1}(\overline{f}(\mathbf{a}')).
\end{align*}

We can get that it is the unique way to extend $f$ as an $\Omega$-Rota-Baxter algebra morphism. Hence $(\shh(A), \diamond )$ is the free commutative $\Omega$-Rota-Baxter algebra generated by $A$.
\end{proof}

Let us assume that $A$ is unitary. We denote its unit by $1_A$.
For each $\omega \in \Omega$, define a linear map $P_{\omega}: \sh(A) \rightarrow \sh(A), \mathbf{a} \mapsto 1_A \ot_{\omega} \mathbf{a}$.

\begin{prop}
If $A$ is a unitary commutative algebra and $(\Omega, \leftarrow, \rightarrow, \lhd, \rhd, \cdot,\lambda)$
 is a commutative $\lambda$-$\ets$, then $(\sh(A), \diamond,(P_{\omega})_{\omega \in \Omega})$
is a commutative $\Omega$-Rota-Baxter algebra.
\end{prop}

\begin{proof}
For $\mathbf{a}, \mathbf{b} \in \sh(A)$ and $\alpha, \beta \in \Omega$,
\begin{align*}
&\ P_{\alpha}(\mathbf{a}) \diamond P_{\beta} (\mathbf{b})=(1_A \ot_{\alpha} \mathbf{a}) \diamond (1_A \ot_{\beta} \mathbf{b})\\
=&\ 1_A \ot_{\alpha \rightarrow \beta} \big((1_A \ot_{\alpha \rhd \beta} \mathbf{a}) \diamond \mathbf{b} \big)+ 1_A \ot_{\alpha \leftarrow \beta} \big(\mathbf{a} \diamond (1_A \ot_{\alpha \lhd \beta} \mathbf{b}) \big)\\
&\ +\lambda_{\alpha , \beta} 1_A \ot_{\alpha \cdot \beta}(\mathbf{a} \diamond \mathbf{b})\\
=&\ P_{\alpha \rightarrow \beta} (P_{\alpha \rhd \beta}(\mathbf{a}) \diamond \mathbf{b})+ P_{\alpha \leftarrow \beta} (\mathbf{a} \diamond P_{\alpha \lhd \beta}(\mathbf{b}))+ \lambda_{\alpha , \beta} P_{\alpha \cdot \beta} (\mathbf{a} \diamond \mathbf{b}),
\end{align*}
hence $\sh(A)$ is an $\Omega$-Rota-Baxter algebra. Next we show
\begin{align}
\mathbf{a} \diamond \mathbf{b}=\mathbf{b} \diamond \mathbf{a}
\label{eq:com2}
\end{align}
by induction on $\ell(\mathbf{a})+ \ell(\mathbf{b})$. If $\ell(\mathbf{a})+ \ell(\mathbf{b})=2$, then $\ell(\mathbf{a})=\ell(\mathbf{b})=1$ and
\begin{align*}
\mathbf{a} \diamond \mathbf{b}=a_0 \diamond b_0=a_0b_0=b_0a_0=b_0 \diamond a_0=\mathbf{b} \diamond \mathbf{a}.
\end{align*}
Suppose Eq.~(\ref{eq:com2}) holds for $\ell(\mathbf{a})+ \ell(\mathbf{b})< p$, where $p \geq 2$ is a fixed integer. We consider the case of $\ell(\mathbf{a})+\ell(\mathbf{b})=p+1$. If one of $\ell(\mathbf{a}), \ell(\mathbf{b})$ is equal to 1, then Eq.~(\ref{eq:com2}) holds directly. So assume $\mathbf{a}=a_0 \ot_{\alpha_1} \mathbf{a}', \mathbf{b}=b_0 \ot_{\beta_1} \mathbf{b}'$, then
\begin{align*}
&\ \mathbf{a} \diamond \mathbf{b}=(a_0 \ot_{\alpha_1} \mathbf{a}') \diamond (b_0 \ot_{\beta_1} \mathbf{b}')\\
=&\ a_0b_0 \ot_{\alpha_1 \rightarrow \beta_1} \big( (1_A \ot_{\alpha_1 \rhd \beta_1} \mathbf{a}') \diamond \mathbf{b}' \big)+ a_0b_0 \ot_{\alpha_1 \leftarrow \beta_1} \big(\mathbf{a}' \diamond (1_A \ot_{\alpha_1 \lhd \beta_1} \mathbf{b}') \big)+ \lambda_{\alpha_1 , \beta_1} a_0b_0 \ot_{\alpha_1 \cdot \beta_1} (\mathbf{a}' \diamond \mathbf{b}')\\
=&\ b_0a_0 \ot_{\alpha_1 \rightarrow \beta_1} \big( (1_A \ot_{\alpha_1 \rhd \beta_1} \mathbf{a}') \diamond \mathbf{b}' \big)+ b_0a_0 \ot_{\alpha_1 \leftarrow \beta_1} \big(\mathbf{a}' \diamond (1_A \ot_{\alpha_1 \lhd \beta_1} \mathbf{b}') \big)+ \lambda_{\alpha_1 , \beta_1} b_0a_0 \ot_{\alpha_1 \cdot \beta_1} (\mathbf{a}' \diamond \mathbf{b}')\\
&\ \hspace{7cm} \text{(by $A$ being a commutative algebra)}\\
=&\ b_0a_0 \ot_{\alpha_1 \rightarrow \beta_1} \big( \mathbf{b}' \diamond(1_A \ot_{\alpha_1 \rhd \beta_1} \mathbf{a}') \big)+ b_0a_0 \ot_{\alpha_1 \leftarrow \beta_1} \big((1_A \ot_{\alpha_1 \lhd \beta_1} \mathbf{b}')\diamond \mathbf{a}' \big)+ \lambda_{\alpha_1 , \beta_1} b_0a_0 \ot_{\alpha_1 \cdot \beta_1} ( \mathbf{b}' \diamond\mathbf{a}')\\
&\ \hspace{7cm} \text{(by the induction hypothesis)}\\
=&\  b_0a_0 \ot_{\beta_1 \leftarrow \alpha_1} \big( \mathbf{b}' \diamond(1_A \ot_{\beta_1 \lhd \alpha_1} \mathbf{a}') \big)+ b_0a_0 \ot_{\beta_1 \rightarrow \alpha_1} \big((1_A \ot_{\beta_1 \rhd \alpha_1} \mathbf{b}')\diamond \mathbf{a}' \big)+ \lambda_{\beta_1 , \alpha_1} b_0a_0 \ot_{\beta_1 \cdot \alpha_1} ( \mathbf{b}' \diamond\mathbf{a}')\\
&\ \hspace{7cm} \text{(by $\Omega$ being commutative)}\\
=&\ (b_0 \ot_{\beta_1} \mathbf{b}') \diamond (a_0 \ot_{\alpha_1} \mathbf{a}')=\mathbf{b} \ot \mathbf{a}.
\end{align*}
Hence $(\sh(A), \diamond)$ is a commutative algebra.
\end{proof}

Let $A$ be a commutative algebra. We put $uA=\mathbf{k}\oplus A$ and give it a product defined by
\[(\lambda+a)(\mu+b)=\lambda\mu+(\lambda b+\mu a+ab).\]
Then $uA$ is a commutative unitary algebra and its unit $1_A$ is the unit $1$ of $\mathbf{k}$.

\begin{theorem} \label{theocomRB}
We put
\[\sh'(A)=A\oplus\bigoplus_{n \geq 2} \underbrace{uA\ot (\bfk \Omega) \ot \cdots \ot (\bfk \Omega) \ot uA}_{\text{$n$'s $V$ and $(n-1)$'s $(\bfk \Omega)$}}.\]
Then $\sh'(A)$ is the free commutative $\Omega$-Rota-Baxter algebra generated by the algebra $A$.
\end{theorem}

\begin{proof}
Let $(R,\cdot,(P_{\omega})_{\omega \in \Omega})$ be a commutative $\Omega$-Rota-Baxter algebra and $f : A \rightarrow  R$ a  (nonunitary) algebra homomorphism. We extend $f$, first from $uA$ to $R$
as a unitary algebra morphism by sending $1_{uA}$ to $1_R$, then as
an $\Omega$-Rota-Baxter algebra morphism $\overline{f}: \sh'(A) \rightarrow R$ as follows:
for $\mathbf{a} \in \sh(A)$, we define $\overline{f}(a)$ by induction on $\ell(\mathbf{a})$. If $\ell(\mathbf{a})=1$, then define $\overline{f}(\mathbf{a})=f(\mathbf{a})$. Suppose $\overline{f}(\mathbf{a})$ has been defined for all $\mathbf{a}$ with $\ell(\mathbf{a}) \leq p$, where $p \geq 1$ is a fixed integer. Consider the case of $\ell(\mathbf{a})=p+1$. Suppose $\mathbf{a}=a_0 \ot_{\alpha_1} \mathbf{a}'$, then define
\begin{align*}
\overline{f}(\mathbf{a}) := f(a_0) \cdot P_{\alpha_1}(\overline{f}(\mathbf{a}')).
\end{align*}
For any $\mathbf{a}\in \sh'(A)$ and for any $\alpha \in \Omega$:
\[\overline{f}\circ P_\alpha(\mathbf{a})
=\overline{f}(1_A\otimes_\alpha \mathbf{a})
=1_B\cdot P_\alpha(\overline{f}(\mathbf{a}))=P_\alpha\circ\overline{f}(\mathbf{a}).\]
Let us prove that this is an algebra morphism. Let $\mathbf{a},\mathbf{b}\in \sh'(A)$, let us prove that
$\overline{f}(\mathbf{a}\diamond \mathbf{b})=\overline{f}(\mathbf{a})\overline{f}(\mathbf{b})$
by induction on $n=\ell(\mathbf{a})+\ell(\mathbf{b})$. If $\ell(\mathbf{a})=\ell(\mathbf{b})=1$, then
\[\overline{f}(\mathbf{a}\diamond \mathbf{b})=\overline{f}(a_0b_0)=f(a_0b_0)
=f(a_0)\cdot f(b_0)=\overline{f}(\mathbf{a})\cdot \overline{f}(\mathbf{b}).\]
If $\ell(\overline{a})=1$ and $\ell(\overline{b})>1$, then
\begin{align*}
\overline{f}(\mathbf{a}\diamond \mathbf{b})
&=\overline{f}(a_0b_0\otimes_{\alpha_1}\mathbf{a}')\\
&=f(a_0b_0)\cdot P_{\alpha_1}\circ \overline{f}(\mathbf{a}')\\
&=f(a_0)\cdot f(b_0)\cdot P_{\alpha_1}\circ \overline{f}(\mathbf{a}')\\
&=\overline{f}(\mathbf{a})\cdot \overline{f}(\mathbf{b}).
\end{align*}
This is similar if  $\ell(\overline{a})>1$ and $\ell(\overline{b})=1$.
If  $\ell(\overline{a})>1$ and $\ell(\overline{b})>1$, then
\begin{align*}
\overline{f}(\mathbf{a}\diamond \mathbf{b})&=
\overline{f}(\ a_0b_0 \ot_{\alpha_1 \rightarrow \beta_1} ((1 \ot_{\alpha_1 \rhd \beta_1} \mathbf{a}') \diamond \mathbf{b}' ))
+\overline{f}(a_0b_0 \ot_{\alpha \leftarrow \beta_1} (\mathbf{a}' \diamond(1 \ot_{\alpha_1 \lhd \beta_1} \mathbf{b}') ))  \\
&\ +\overline{f}(\lambda_{\alpha_1 , \beta_1} a_0b_0 \ot_{\alpha_1 \cdot \beta_1}(\mathbf{a}' \diamond \mathbf{b}'))\\
&=f(\ a_0b_0)\cdot P_{\alpha_1 \rightarrow \beta_1}\circ \overline{f}((1 \ot_{\alpha_1 \rhd \beta_1} \mathbf{a}') \diamond \mathbf{b}' )
+f(a_0b_0)\cdot P_{\alpha \leftarrow \beta_1}\circ\overline{f}(\mathbf{a}' \diamond(1 \ot_{\alpha_1 \lhd \beta_1} \mathbf{b}') ))  \\
&+\lambda_{\alpha_1 , \beta_1} f(a_0b_0)\cdot P_{\alpha_1 \cdot \beta_1}\circ \overline{f}(\mathbf{a}' \diamond \mathbf{b}')\\
&=f(a_0)\cdot f(b_0)\cdot \overline{f}(
P_{\alpha_1 \rightarrow \beta_1}(P_{\alpha_1 \rhd \beta_1}( \mathbf{a}')\diamond \mathbf{b}' ))
+f(a_0)\cdot f(b_0)\cdot \overline{f}(
P_{\alpha \leftarrow \beta_1}(\mathbf{a}' \diamond P_{\alpha_1 \lhd \beta_1}(\mathbf{b}')))\\
&+\lambda_{\alpha_1 , \beta_1}f(a_0)\cdot f(b_0)\cdot \overline{f}(P_{\alpha_1 \cdot \beta_1}
(\mathbf{a}' \diamond \mathbf{b}'))\\
&=f(a_0)\cdot f(b_0)\cdot \overline{f}(P_{\alpha_1}(\mathbf{a}')P_{\beta_1}(\mathbf{b}'))\\
&=f(a_0)\cdot f(b_0)\cdot \overline{f}(P_{\alpha_1}(\mathbf{a}'))\cdot \overline{f}(P_{\beta_1}(\mathbf{b}'))
\hspace{1.5cm} \text{(by the induction hypothesis)}\\
&=f(a_0)\cdot f(b_0)\cdot P_{\alpha_1}\circ \overline{f}(\mathbf{a}')\cdot P_{\beta_1}\circ \overline{f}(\mathbf{b}')\\
&=f(a_0)\cdot P_{\alpha_1}\circ \overline{f}(\mathbf{a}')\cdot f(b_0)\cdot P_{\beta_1}\circ \overline{f}(\mathbf{b}')
\hspace{1.3cm} \text{(as $B$ is commutative)}\\
&=\overline{f}(\mathbf{a})\cdot \overline{f}(\mathbf{b}).
\end{align*}
We get that it is the unique way to extend $f$ as an $\Omega$-Rota-Baxter algebra morphism. Hence $\sh'(A)$
is the free commutative $\Omega$-Rota-Baxter algebra generated by $A$.
\end{proof}

\section{More results on $\lambda$-$\ets$ and $\ets$}

\subsection{Description in terms of linear and bilinear maps}

As in Lemma 5 of \cite{Foi20}, we obtain:
\begin{lemma}
Let $(\Omega,\leftarrow,\rightarrow,\lhd,\rhd,\cdot)$ be a set with five operations
and $\lambda=(\lambda_{\alpha,\beta})_{\alpha,\beta \in \Omega}$ be a family of elements in ${\bf k}$ indexed by $\Omega^2$ . We denote by $\K\Omega$
the vector space generated by $\Omega$. We put:
\begin{align*}
\varphi_\leftarrow&:\left\{\begin{array}{rcl}
\K\Omega^{\otimes 2}&\longrightarrow&\K\Omega^{\otimes 2}\\
\alpha\otimes \beta&\longrightarrow&\alpha \leftarrow \beta\otimes \alpha \lhd \beta,
\end{array}\right.\\
\varphi_\rightarrow&:\left\{\begin{array}{rcl}
\K\Omega^{\otimes 2}&\longrightarrow&\K\Omega^{\otimes 2}\\
\alpha \otimes \beta&\longrightarrow&\alpha \rightarrow \beta \otimes \alpha \rhd \beta,
\end{array}\right.\\
\psi_\cdot&:\left\{\begin{array}{rcl}
\K\Omega^{\otimes 2}&\longrightarrow&\K\Omega\\
\alpha \otimes \beta&\longrightarrow&\lambda_{\alpha, \beta} \alpha \cdot \beta.
\end{array}\right.
\end{align*}
Then $(\Omega,\leftarrow,\rightarrow,\lhd,\rhd,\cdot,\lambda)$ is a $\lambda$-$\ets$
if, and only if:
\begin{align}
\label{eq34dend} (\tau \otimes \id)\circ (\id \otimes \varphi_\leftarrow)\circ (\tau \otimes \id)\circ (\varphi_\rightarrow\otimes \id)
&= (\varphi_\rightarrow\otimes \id)\circ (\id \otimes \varphi_\leftarrow),\\
\label{eq35dend} (\id \otimes \varphi_\leftarrow)\circ (\tau \otimes \id)\circ (\id \otimes \varphi_\leftarrow)\otimes
(\tau \otimes \id)\circ (\varphi_\leftarrow \otimes \id)&=(\varphi_\leftarrow\otimes \id)\circ (\id \otimes \varphi_\leftarrow),\\
\label{eq36dend} (\id \otimes \varphi_\rightarrow)\circ (\tau \otimes \id)\circ (\id \otimes \varphi_\leftarrow)
\circ (\tau \otimes \id)\circ (\varphi_\leftarrow\otimes \id)&=(\varphi_\leftarrow\otimes \id)\circ (\id\otimes \varphi_\rightarrow),\\
\label{eq37dend} (\id \otimes \varphi_\leftarrow)\circ (\varphi_\rightarrow\otimes \id)\circ (\id \otimes \varphi_\rightarrow)
&=(\varphi_\rightarrow\otimes \id)\circ (\id \otimes \tau)\circ (\varphi_\leftarrow \otimes \id),\\
\label{eq38dend} (\id \otimes \varphi_\rightarrow)\circ (\varphi_\rightarrow\otimes \id)\circ (\id \otimes \varphi_\rightarrow)
&=(\varphi_\rightarrow\otimes \id)\circ (\id \otimes \tau)\circ (\varphi_\rightarrow \otimes \id),\\
\label{eq:equ1} \varphi_\rightarrow \circ (\id \otimes \psi_\cdot) &=(\psi_\cdot \otimes \id)\circ (\id \otimes \tau)\circ (\varphi_\rightarrow \otimes \id),\\
\label{eq:equ2}(\psi_\cdot\otimes \id)\circ (\id \otimes \tau)\circ (\id \otimes \varphi_\leftarrow)\circ (\tau \otimes \id)
\circ (\varphi_\leftarrow \otimes \id)&=\tau \circ \varphi_\leftarrow \circ (\id \otimes \psi_\cdot),\\
\label{eq:equ3} (\id \otimes \psi_\cdot)\circ (\varphi_\rightarrow \otimes \id)\circ (\id \otimes \varphi_\rightarrow)
&=\varphi_\rightarrow\circ (\psi_\cdot \otimes \id),\\
\label{eq:equ4} (\psi_\cdot\otimes \id)\circ (\id \otimes \tau)\circ (\varphi_\leftarrow\otimes \id)&=(\psi_\cdot \otimes \id)
\circ (\id \otimes \varphi_\rightarrow),\\
\label{eq:equ5} (\psi_\cdot \otimes \id)\circ (\id \otimes \varphi_\leftarrow)&=\varphi_\leftarrow\circ (\psi_\cdot \otimes \id),\\
\label{eq:equ6} \psi_\cdot\circ(\psi_\cdot \otimes \id)&=\psi_\cdot\circ(\id \otimes \psi_\cdot).
\end{align}
In particular, $\psi_\cdot$ is an associative product.
\end{lemma}

\begin{proof}
By Lemma~5 in~\cite{Foi20}, Eqs.~(\ref{eq34dend})-(\ref{eq38dend}) are equivalent to $(\Omega, \leftarrow, \rightarrow, \lhd ,\rhd)$ being an EDS. Moreover, direct computations prove that Eq.~(\ref{eq:equ1}) is equivalent to Eq.~(\ref{EQ11}) and condition~(a); Eq.~(\ref{eq:equ2}) is equivalent to Eq.~(\ref{EQ12}) and condition~(b); Eq.~(\ref{eq:equ3}) is equivalent to Eq.~(\ref{EQ13}) and condition~(c); Eq.~(\ref{eq:equ4}) is equivalent to Eq.~(\ref{EQ14}) and condition~(d); Eq.~(\ref{eq:equ5}) is equivalent to Eq.~(\ref{EQ15}) and condition~(e); Eq.~(\ref{eq:equ6}) is equivalent to Eq.~(\ref{EQ16}) and condition~(f) in Definition~\ref{def:leds}.
\end{proof}

Similarly, we obtain for $\ets$:

\begin{lemma}
Let $(\Omega,\leftarrow,\rightarrow,\lhd,\rhd,\ast,\cdot)$ be a set with six operations. We put:
\begin{align*}
\varphi_\leftarrow&:\left\{\begin{array}{rcl}
 \Omega^{ 2}&\longrightarrow& \Omega^{ 2}\\
(\alpha, \beta) &\longrightarrow& (\alpha \leftarrow \beta, \alpha \lhd \beta),
\end{array}\right.\\
\varphi_\rightarrow&:\left\{\begin{array}{rcl}
 \Omega^{2}&\longrightarrow& \Omega^{2}\\
(\alpha, \beta)&\longrightarrow& ( \alpha \rightarrow \beta, \alpha \rhd \beta),
\end{array}\right.\\
\varphi_\ast&:\left\{\begin{array}{rcl}
 \Omega^{ 2}&\longrightarrow& \Omega^{ 2}\\
(\alpha, \beta) &\longrightarrow& (\alpha\cdot \beta, \alpha \ast \beta).
\end{array}\right.
\end{align*}
Then $(\Omega,\leftarrow,\rightarrow,\lhd,\rhd,\ast,\cdot)$ is an $\ets$
if, and only if, (34)-(38) of \cite{Foi20} are satisfied and:
\begin{align}
\label{eq:equu1} 
(\varphi_\rightarrow\otimes \id)\circ (\id \otimes \varphi_\ast)&=
(\tau \otimes \id)\circ (\id \otimes \varphi_\ast)\circ (\tau \otimes \id)\circ (\varphi_\rightarrow \otimes \id),\\
%
\label{eq:equu2}
 (\varphi_\leftarrow \otimes \id)\circ (\id \otimes \varphi_\ast)&=
(\id \otimes \varphi_\ast)\circ (\tau \otimes \id)\circ(\id \otimes \varphi_\leftarrow)\circ (\tau \otimes \id)
\circ (\varphi_\leftarrow \otimes \id),\\
%
\label{eq:equu3} (\id \otimes \varphi_\ast)\circ (\varphi_\rightarrow \otimes \id)\circ (\id \otimes \varphi_\rightarrow)
&=(\varphi_\rightarrow \otimes \id)\circ (\id \otimes \tau)\circ  (\id \otimes \varphi_\ast),\\
%
\label{eq:equu4} (\id \otimes \varphi_\ast)\circ (\tau \otimes \id)\circ (\varphi_\leftarrow \otimes \id)
&=(\id \otimes \varphi_\ast)\circ (\tau \otimes \id)\circ (\id \otimes \tau)\circ (\id \otimes \varphi_\rightarrow),\\
\label{eq:equu5}(\varphi_\ast\otimes \id)\circ (\id \otimes \varphi_\leftarrow)&=(\tau \otimes \id)\circ
(\id \otimes \varphi_\leftarrow)\circ(\tau \otimes \id)\circ  (\varphi_\ast \otimes \id),\\
\label{eq:equu6}
(\varphi_\ast\otimes \id)\circ (\id \otimes \tau)\circ (\varphi_\ast\otimes \id)
&=(\id \otimes \tau)\circ(\varphi_\ast\otimes \id)\circ (\id \otimes \varphi_\ast).
\end{align}
\end{lemma}

\begin{proof}
By Lemma~5 in~\cite{Foi20}, Eqs.~(34)-(38) are equivalent to $(\Omega, \leftarrow, \rightarrow, \lhd ,\rhd)$ being an EDS. Moreover, direct computations prove that Eq.~(\ref{eq:equu1}) is equivalent to Eqs.~(\ref{eq17}),(\ref{eq18}) and (\ref{EQ17}); Eq.~(\ref{eq:equu2}) is equivalent to Eqs.~(\ref{eq19}), (\ref{eq20}) and (\ref{EQ20}); Eq.~(\ref{eq:equu3}) is equivalent to Eqs.~(\ref{eq21}), (\ref{eq22}) and (\ref{EQ23}); Eq.~(\ref{eq:equu4}) is equivalent to Eqs.~(\ref{eq23}), (\ref{eq24}) and (\ref{EQ26}); Eq.~(\ref{eq:equu5}) is equivalent to Eqs.~(\ref{eq25}), (\ref{eq26}) and (\ref{EQ29}); Eq.~(\ref{eq:equu6}) is equivalent to Eqs.~(\ref{eq27}), (\ref{EQ32}) and (\ref{EQ33}).
\end{proof}

\subsection{A description of all $\lambda$-$\ets$ of cardinality two}

The following table gives all $\lambda$-$\ets$. We slightly generalize our definition, by accepting more general
maps $\varphi_\cdot:\K\Omega^{\otimes 2}\longrightarrow \K\Omega$.
The underlying set is $\{a,b\}$ and all the products are given by a $2\times 2$ table.
Here, $\lambda,\mu$ are elements of the base field $\K$.

\[\begin{array}{|c|c|c|c|c|c|c|}
\hline\mbox{Type}&\leftarrow&\rightarrow&\triangleleft&\triangleright&\varphi_{\ast}&\mbox{Name}\\
\hline A&\begin{pmatrix}
a&a\\
a&a
\end{pmatrix}&
\begin{pmatrix}
a&a\\
a&a
\end{pmatrix}
&\begin{pmatrix}
a&a\\
a&a
\end{pmatrix}
&\begin{pmatrix}
a&a\\
a&a
\end{pmatrix}
&\begin{pmatrix}
(\lambda+\mu)a&(\lambda+\mu)a\\
(\lambda+\mu)a&\lambda a+\mu b
\end{pmatrix}&A_1(\lambda,\mu)\\
\cline{4-7}&&&\begin{pmatrix}
a&b\\
a&b
\end{pmatrix}&
\begin{pmatrix}
a&a\\
b&b
\end{pmatrix}&\begin{pmatrix}
(\lambda+\mu)a&(\lambda+\mu)a\\
(\lambda+\mu)a&\lambda a+\mu b
\end{pmatrix}&A_2(\lambda,\mu)\\
\hline B&\begin{pmatrix}
a&a\\
a&a
\end{pmatrix}&
\begin{pmatrix}
a&b\\
a&b
\end{pmatrix}&
\begin{pmatrix}
a&a\\
a&a
\end{pmatrix}&
\begin{pmatrix}
a&a\\
a&a
\end{pmatrix}&
\begin{pmatrix}
\lambda a&\lambda a\\
\lambda a&\lambda a
\end{pmatrix},\:
\begin{pmatrix}
\lambda a&\lambda b\\
\lambda a&\lambda b
\end{pmatrix}&B_1'(\lambda),\:B_1''(\lambda)\\
\cline{4-5}\cline{7-7}&&&
\begin{pmatrix}
a&b\\
a&b
\end{pmatrix}
&\begin{pmatrix}
a&a\\
b&b
\end{pmatrix}&&B_2'(\lambda),\:B_2''(\lambda)\\
\hline C&
\begin{pmatrix}
a&a\\
a&b
\end{pmatrix}&\begin{pmatrix}
a&a\\
a&b
\end{pmatrix}
&\begin{pmatrix}
a&a\\
a&a
\end{pmatrix}
&\begin{pmatrix}
a&a\\
a&a
\end{pmatrix}
&\begin{pmatrix}
\lambda a&\lambda a\\
\lambda a&\lambda b
\end{pmatrix}&C_1(\lambda)\\
\cline{4-5}\cline{7-7}&&&\begin{pmatrix}
a&b\\
a&b
\end{pmatrix}
&\begin{pmatrix}
a&a\\
b&b
\end{pmatrix}&&C_3(\lambda)\\
\cline{4-5}\cline{7-7}&&&
\begin{pmatrix}
b&b\\
b&b
\end{pmatrix}&
\begin{pmatrix}
b&b\\
b&b
\end{pmatrix}&&C_5(\lambda)\\
\cline{4-7}&&&
\begin{pmatrix}
a&a\\
a&a
\end{pmatrix}&
\begin{pmatrix}
b&b\\
b&b
\end{pmatrix}&\begin{pmatrix}
0&0\\
0&0
\end{pmatrix}&C_2\\
\cline{4-5}\cline{7-7}&&&
\begin{pmatrix}
b&b\\
b&b
\end{pmatrix}&
\begin{pmatrix}
a&a\\
a&a
\end{pmatrix}&&C_4\\
\hline \end{array}\]

\[\begin{array}{|c|c|c|c|c|c|c|}
\hline\mbox{Type}&\leftarrow&\rightarrow&\triangleleft&\triangleright&\varphi_{\ast}&\mbox{Name}\\
\hline D&\begin{pmatrix}
a&a\\
b&b
\end{pmatrix}&
\begin{pmatrix}
a&a\\
a&a
\end{pmatrix}
&\begin{pmatrix}
a&a\\
a&a
\end{pmatrix}&
\begin{pmatrix}
a&a\\
a&a
\end{pmatrix}&
\begin{pmatrix}
\lambda a&\lambda a\\
\lambda a&\lambda a
\end{pmatrix},\:
\begin{pmatrix}
\lambda a&\lambda a\\
\lambda b&\lambda b
\end{pmatrix}&D_1'(\lambda),\:D_1''(\lambda)\\
\cline{4-5}\cline{7-7}&&&\begin{pmatrix}
a&b\\
a&b
\end{pmatrix}&
\begin{pmatrix}
a&a\\
b&b
\end{pmatrix}&&D_2'(\lambda),\:D_2''(\lambda)\\
\hline E&\begin{pmatrix}
a&a\\
b&b
\end{pmatrix}&
\begin{pmatrix}
a&a\\
b&b
\end{pmatrix}&
\begin{pmatrix}
a&a\\
a&a
\end{pmatrix}&
\begin{pmatrix}
a&a\\
a&a
\end{pmatrix}&
\begin{pmatrix}
\lambda a&\lambda a\\
\lambda b&\lambda b
\end{pmatrix}&E_1(\lambda)\\
\cline{4-5}\cline{7-7}&&&\begin{pmatrix}
a&b\\
a&b
\end{pmatrix}&
\begin{pmatrix}
a&a\\
b&b
\end{pmatrix}&&E_3(\lambda)\\
\cline{4-7}&&&\begin{pmatrix}
a&a\\
a&a
\end{pmatrix}&
\begin{pmatrix}
b&b\\
b&b
\end{pmatrix}&
\begin{pmatrix}
0&0\\
0&0
\end{pmatrix}&E_2\\
\hline F&\begin{pmatrix}
a&a\\
b&b
\end{pmatrix}&
\begin{pmatrix}
a&b\\
a&b
\end{pmatrix}&
\begin{pmatrix}
a&a\\
a&a
\end{pmatrix}&
\begin{pmatrix}
a&a\\
a&a
\end{pmatrix}&
\begin{pmatrix}
(\lambda+\mu) a&(\lambda+\mu)a\\
(\lambda+\mu)a&\lambda a+\mu b
\end{pmatrix},\:
\begin{pmatrix}
(\lambda+\mu)a&(\lambda+\mu)b\\
(\lambda+\mu)b&\lambda a+\mu b
\end{pmatrix},&F'_1(\lambda,\mu),\:F_1''(\lambda,\mu)\\
&&&&&\begin{pmatrix}
\lambda a&\lambda b\\
\lambda a&\lambda b
\end{pmatrix},\:
\begin{pmatrix}
\lambda a&\lambda a\\
\lambda b&\lambda b
\end{pmatrix}&F_1'(\lambda),\:F_1''(\lambda)\\
\cline{4-7}&&&
\begin{pmatrix}
a&b\\
a&b
\end{pmatrix}&
\begin{pmatrix}
a&a\\
b&b
\end{pmatrix}&\mbox{any associative product $*$}&F_3(*)\\
\cline{4-7}&&&\begin{pmatrix}
a&b\\
b&a
\end{pmatrix}&
\begin{pmatrix}
a&b\\
b&a
\end{pmatrix}&
\begin{pmatrix}
\lambda a&0\\
0&\lambda b
\end{pmatrix}&F_4(\lambda)\\
\cline{4-7}&&&
\begin{pmatrix}
a&a\\
a&a
\end{pmatrix}&
\begin{pmatrix}
b&b\\
b&b
\end{pmatrix}&
\begin{pmatrix}
0&0\\
0&0
\end{pmatrix}&F_2\\
\cline{4-5}\cline{7-7}&&&\begin{pmatrix}
a&b\\
b&a
\end{pmatrix}&
\begin{pmatrix}
b&a\\
a&b
\end{pmatrix}&&F_5\\
\hline G&\begin{pmatrix}
a&b\\
a&b
\end{pmatrix}&
\begin{pmatrix}
a&b\\
a&b
\end{pmatrix}&
\begin{pmatrix}
a&a\\
a&a
\end{pmatrix}&
\begin{pmatrix}
a&a\\
a&a
\end{pmatrix}&
\begin{pmatrix}
\lambda a&\lambda b\\
\lambda a&\lambda b
\end{pmatrix}&G_1(\lambda)\\
\cline{4-5}\cline{7-7}&&&
\begin{pmatrix}
a&b\\
a&b
\end{pmatrix}&
\begin{pmatrix}
a&a\\
b&b
\end{pmatrix}&&G_3(\lambda)\\
\cline{4-7}&&&\begin{pmatrix}
a&a\\
a&a
\end{pmatrix}&
\begin{pmatrix}
b&b\\
b&b
\end{pmatrix}&
\begin{pmatrix}
0&0\\
0&0
\end{pmatrix}&G_2\\
\hline H&\begin{pmatrix}
a&b\\
b&a
\end{pmatrix}&
\begin{pmatrix}
a&b\\
b&a
\end{pmatrix}&
\begin{pmatrix}
a&a\\
a&a
\end{pmatrix}&
\begin{pmatrix}
a&a\\
a&a
\end{pmatrix}&
\begin{pmatrix}
\lambda a&\lambda b\\
\lambda b&\lambda a
\end{pmatrix}&H_1(\lambda)\\
\cline{4-5}\cline{7-7}&&&
\begin{pmatrix}
a&b\\
a&b
\end{pmatrix}&
\begin{pmatrix}
a&a\\
b&b
\end{pmatrix}&&H_2(\lambda)\\
\hline\end{array}\]

The commutative  $\lambda$-$\ets$ are the ones of type $A$ and $H$, $C_1(\lambda)$, $C_3(\lambda)$, $C_5(\lambda)$,
$F'_1(\lambda,\mu)$, $F_1''(\lambda,\mu)$ and $F_4(\lambda)$.
The opposite of $B_1'(\lambda)$, $B_1''(\lambda)$, $B_2'(\lambda)$ and $B_2''(\lambda)$ are respectively
$D_1'(\lambda)$, $D_1''(\lambda)$, $D_2'(\lambda)$ and $D_2''(\lambda)$.
The opposite of $C_2$ is $C_4$. The opposite of $E_1(\lambda)$, $E_2$ and $E_3(\lambda)$ are respectively
$G_1(\lambda)$, $G_2$ and $G_3(\lambda)$. The opposite of $F'_1(\lambda)$ is $F''_1(\lambda)$.
The $\lambda$-$\ets$ $F_2$ and $F_5$ are not commutative but are isomorphic to their opposite in a non trivial way.
Finally, if $*$ is an associative product, the opposite of $F_3(*)$ is $F_3(*^{op})$.

\subsection{A description of all $\ets$ of cardinality two}

The following table gives all the $\ets$ of cardinality 2.

\[\begin{array}{|c|c|c|c|c|c|c|}
\hline\mbox{Type}&\leftarrow&\rightarrow&\triangleleft&\triangleright&\ast&\cdot\\
\hline A_1&\begin{pmatrix}
a&a\\
a&a
\end{pmatrix}&
\begin{pmatrix}
a&a\\
a&a
\end{pmatrix}
&\begin{pmatrix}
a&a\\
a&a
\end{pmatrix}
&\begin{pmatrix}
a&a\\
a&a
\end{pmatrix}
&\begin{pmatrix}
a&a\\
a&a
\end{pmatrix},\:
\begin{pmatrix}
b&b\\
b&b
\end{pmatrix}&
\begin{pmatrix}
a&a\\
a&a
\end{pmatrix},\:
\begin{pmatrix}
a&a\\
a&b
\end{pmatrix}\\
\cline{1-1}\cline{4-5}A_2&&&\begin{pmatrix}
a&b\\
a&b
\end{pmatrix}&
\begin{pmatrix}
a&a\\
b&b
\end{pmatrix}&&\\
\hline B_1&\begin{pmatrix}
a&a\\
a&a
\end{pmatrix}&
\begin{pmatrix}
a&b\\
a&b
\end{pmatrix}&
\begin{pmatrix}
a&a\\
a&a
\end{pmatrix}&
\begin{pmatrix}
a&a\\
a&a
\end{pmatrix}&
\begin{pmatrix}
a&a\\
a&a
\end{pmatrix},\:
\begin{pmatrix}
b&b\\
b&b
\end{pmatrix}&
\begin{pmatrix}
a&a\\
a&a
\end{pmatrix},\:
\begin{pmatrix}
a&b\\
a&b
\end{pmatrix}\\
\cline{1-1}\cline{4-5}B_2&&&
\begin{pmatrix}
a&b\\
a&b
\end{pmatrix}&
\begin{pmatrix}
a&a\\
b&b
\end{pmatrix}&&\\
\hline C_1&\begin{pmatrix}
a&a\\
a&b
\end{pmatrix}&\begin{pmatrix}
a&a\\
a&b
\end{pmatrix}
&\begin{pmatrix}
a&a\\
a&a
\end{pmatrix}
&\begin{pmatrix}
a&a\\
a&a
\end{pmatrix}
&\begin{pmatrix}
a&a\\
a&a
\end{pmatrix},\:
\begin{pmatrix}
b&b\\
b&b
\end{pmatrix}&
\begin{pmatrix}
a&a\\
a&b
\end{pmatrix}\\
\cline{1-1}\cline{4-5}C_3&&&\begin{pmatrix}
a&b\\
a&b
\end{pmatrix}
&\begin{pmatrix}
a&a\\
b&b
\end{pmatrix}&&\\
\cline{1-1}\cline{4-5}C_5&&&
\begin{pmatrix}
b&b\\
b&b
\end{pmatrix}&
\begin{pmatrix}
b&b\\
b&b
\end{pmatrix}&&\\
\hline D_1&\begin{pmatrix}
a&a\\
b&b
\end{pmatrix}&
\begin{pmatrix}
a&a\\
a&a
\end{pmatrix}
&\begin{pmatrix}
a&a\\
a&a
\end{pmatrix}&
\begin{pmatrix}
a&a\\
a&a
\end{pmatrix}&
\begin{pmatrix}
a&a\\
a&a
\end{pmatrix},\:
\begin{pmatrix}
b&b\\
b&b
\end{pmatrix}&
\begin{pmatrix}
a&a\\
a&a
\end{pmatrix},\:
\begin{pmatrix}
a&a\\
b&b
\end{pmatrix}\\
\cline{1-1}\cline{4-5}D_2&&&
\begin{pmatrix}
a&b\\
a&b
\end{pmatrix}&
\begin{pmatrix}
a&a\\
b&b
\end{pmatrix}&&\\
\hline E_1&\begin{pmatrix}
a&a\\
b&b
\end{pmatrix}&
\begin{pmatrix}
a&a\\
b&b
\end{pmatrix}&
\begin{pmatrix}
a&a\\
a&a
\end{pmatrix}&
\begin{pmatrix}
a&a\\
a&a
\end{pmatrix}&
\begin{pmatrix}
a&a\\
a&a
\end{pmatrix},\:
\begin{pmatrix}
b&b\\
b&b
\end{pmatrix}&
\begin{pmatrix}
a&a\\
b&b
\end{pmatrix}\\
\cline{4-6} E_3&&
&
\begin{pmatrix}
a&b\\
a&b
\end{pmatrix}&
\begin{pmatrix}
a&a\\
b&b
\end{pmatrix}&
\begin{pmatrix}
a&a\\
a&a
\end{pmatrix}&\\
\hline F_1&\begin{pmatrix}
a&a\\
b&b
\end{pmatrix}&
\begin{pmatrix}
a&b\\
a&b
\end{pmatrix}&
\begin{pmatrix}
a&a\\
a&a
\end{pmatrix}&
\begin{pmatrix}
a&a\\
a&a
\end{pmatrix}&
\begin{pmatrix}
a&a\\
a&a
\end{pmatrix},\:
\begin{pmatrix}
b&b\\
b&b
\end{pmatrix}&
\begin{pmatrix}
a&a\\
a&a
\end{pmatrix},\:
\begin{pmatrix}
a&a\\
a&b
\end{pmatrix},\:
\begin{pmatrix}
a&a\\
b&b
\end{pmatrix},\\
&&&&&&
\begin{pmatrix}
a&b\\
a&b
\end{pmatrix},\:
\begin{pmatrix}
a&b\\
b&a
\end{pmatrix},\:
\begin{pmatrix}
a&b\\
b&b
\end{pmatrix}\\
\cline{1-1}\cline{4-7} F_3&&&
\begin{pmatrix}
a&b\\
a&b
\end{pmatrix}&
\begin{pmatrix}
a&a\\
b&b
\end{pmatrix}&
\begin{pmatrix}
a&a\\
a&a
\end{pmatrix}&
\begin{pmatrix}
a&a\\
a&a
\end{pmatrix},\:
\begin{pmatrix}
a&a\\
a&b
\end{pmatrix},\:
\begin{pmatrix}
a&a\\
b&b
\end{pmatrix},\\
&&&&&&
\begin{pmatrix}
a&b\\
a&b
\end{pmatrix},\:
\begin{pmatrix}
a&b\\
b&a
\end{pmatrix},\:
\begin{pmatrix}
a&b\\
b&b
\end{pmatrix},\\
&&&&&&
\begin{pmatrix}
b&a\\
a&b
\end{pmatrix},\:
\begin{pmatrix}
b&b\\
b&b
\end{pmatrix}\\
\cline{6-7}&&&&&
\begin{pmatrix}
a&a\\
b&b
\end{pmatrix},\:
\begin{pmatrix}
b&b\\
a&a
\end{pmatrix}&
\begin{pmatrix}
a&b\\
a&b
\end{pmatrix}\\
\cline{6-7}&&&&&
\begin{pmatrix}
a&b\\
a&b
\end{pmatrix},\:
\begin{pmatrix}
b&a\\
b&a
\end{pmatrix}
&
\begin{pmatrix}
a&a\\
b&b
\end{pmatrix}\\
\hline G_1&\begin{pmatrix}
a&b\\
a&b
\end{pmatrix}&
\begin{pmatrix}
a&b\\
a&b
\end{pmatrix}&
\begin{pmatrix}
a&a\\
a&a
\end{pmatrix}&
\begin{pmatrix}
a&a\\
a&a
\end{pmatrix}&
\begin{pmatrix}
a&a\\
a&a
\end{pmatrix},\:
\begin{pmatrix}
b&b\\
b&b
\end{pmatrix}&
\begin{pmatrix}
a&b\\
a&b
\end{pmatrix}\\
\cline{1-1}\cline{4-7}G_3&&&
\begin{pmatrix}
a&b\\
a&b
\end{pmatrix}&
\begin{pmatrix}
a&a\\
b&b
\end{pmatrix}&
\begin{pmatrix}
a&a\\
a&a
\end{pmatrix}&
\begin{pmatrix}
a&b\\
a&b
\end{pmatrix}\\
\hline H_1&\begin{pmatrix}
a&b\\
b&a
\end{pmatrix}&
\begin{pmatrix}
a&b\\
b&a
\end{pmatrix}&
\begin{pmatrix}
a&a\\
a&a
\end{pmatrix}&
\begin{pmatrix}
a&a\\
a&a
\end{pmatrix}&
\begin{pmatrix}
a&a\\
a&a
\end{pmatrix},\:
\begin{pmatrix}
b&b\\
b&b
\end{pmatrix}&
\begin{pmatrix}
a&a\\
a&b
\end{pmatrix}\\
\cline{1-1}\cline{4-5}H_2&&&
\begin{pmatrix}
a&b\\
a&b
\end{pmatrix}&
\begin{pmatrix}
a&a\\
b&b
\end{pmatrix}&&\\
\hline\end{array}\]

\smallskip

\noindent {\bf Acknowledgments}: the authors acknowledge support from the grant ANR-20-CE40-0007
\emph{Combinatoire Alg\'ebrique, R\'esurgence, Probabilit\'es Libres et Op\'erades}. The second author is also supported
by China Scholarship Council to visit ULCO and he thanks Prof. Dominique Manchon for valuable suggestions.
\medskip

\noindent\textbf{Data avalaibility}: this article has no associated data.

\bibliographystyle{amsplain}
\bibliography{biblio}

\end{document}

\end{document}